\documentclass[11pt]{amsart}
\usepackage{amssymb,amsmath,amsthm,mathrsfs}

\usepackage{a4wide}

\usepackage{enumerate,paralist}

\numberwithin{equation}{section}
\newtheorem{theorem}{Theorem}[section]
\newtheorem{proposition}[theorem]{Proposition}
\newtheorem{lemma}[theorem]{Lemma}
\newtheorem{corollary}[theorem]{Corollary}
\theoremstyle{definition}
\newtheorem{definition}[theorem]{Definition}
\newtheorem{remark}[theorem]{Remark}

\def\F{{\mathcal{F}}}
\def\L{{\mathcal{L}}}
\def\H{{\mathcal{H}}}

\def\Bg{\tilde{B}^{2}}
\def\Cbar{\overline{C}}

\def\t{{\mathfrak{t}}}

\begin{document}

\title[Tau function and Chern-Simons invariant]
{Tau function and Chern-Simons invariant}

\author{Andrew Mcintyre}
\address{Bennington College\\
1 College Drive\\
Bennington\\ Vermont 05201\\U.S.A.}
\email{amcintyre@bennington.edu}

\author{Jinsung Park}
\address{School of Mathematics\\ Korea Institute for Advanced Study\\
207-43\\ Hoegiro 85\\ Dong\-daemun-gu\\ Seoul 130-722\\
Korea }
\email{jinsung@kias.re.kr}

\thanks{2010 Mathematics Subject Classification
	58J28, 58J52, 37K10, 32G15, 14J15.}

\date{\today}

\begin{abstract}
We define a Chern-Simons invariant for a certain class of
infinite volume hyperbolic 3-manifolds.
We then prove an expression relating the Bergman tau function
on a cover of the Hurwitz space, to the lifting of the function $F$ defined
by Zograf on Teichm\"uller space, and another holomorphic
function on the cover of the Hurwitz space which we introduce.
If the point in cover of the Hurwitz space corresponds to a Riemann
surface $X$, then this function is constructed from the renormalized
volume and our Chern-Simons invariant for the bounding 3-manifold
of $X$ given by Schottky uniformization, together with a regularized Polyakov
integral relating determinants of Laplacians on $X$ in the
hyperbolic and singular flat metrics. Combining this with a result
of Kokotov and Korotkin, we obtain a similar expression for the
isomonodromic tau function of Dubrovin. We also obtain a relation between
the Chern-Simons invariant and the eta invariant of the bounding
3-manifold, with defect given by the phase of the Bergman tau function
of $X$.
\end{abstract}

\maketitle


\section{Introduction} \label{s:Introduction}

Let $\mathfrak{M}_g$ be the moduli space
of compact Riemann surfaces of genus $g$, and let $\mathfrak{T}_g$
be the corresponding Teichm\"uller space of marked surfaces.
Let ${H}_{g,n}(k_1,\ldots,k_{\ell})$ be the Hurwitz space of
equivalence classes $[\lambda:X\to \mathbb{CP}^1]$ of degree
$n$ holomorphic maps
from a compact Riemann surface $X$ to the Riemann sphere
with ramification index $(k_1,\ldots,k_\ell)$ at infinity,
and all ramification points being simple.
Equipping $X$ with a marking---a choice of standard generators
of $\pi_1(X)$---gives a covering space
$\tilde{H}_{g,n}(k_1,\ldots,k_\ell)$, in the same way that
one obtains the covering $\mathfrak{T}_g$ of $\mathfrak{M}_g$.
We will also be concerned with
a space
$\mathcal{H}_g(k_1,\ldots,k_m)$, whose fiber
over a point in $\mathfrak{M}_g$ is the space of
holomorphic 1-forms on the corresponding Riemann
surface with zeroes of order $k_1,\ldots,k_m$, and
we write $\tilde{\mathcal{H}}_g(k_1,\ldots,k_m)$ for the
corresponding fiber space over $\mathfrak{T}_g$.
(See Section \ref{s:background}
for precise definitions.)

In \cite{KKfirst},
Kokotov and Korotkin introduced the object $\tau_B$,
referred to as the Bergman tau function,
with the property that $\tau_B^{24}$ is a globally
well-defined holomorphic function on $\tilde{H}_{g,n}(k_1,\ldots,k_\ell)$.
In \cite{KK}, they defined $\tau_B$ in the same way for
$\tilde{\mathcal{H}}_g(k_1,\ldots,k_m)$,
such that $\tau_B^{24}$ is a globally
well-defined holomorphic function on
$\tilde{\mathcal{H}}_g(k_1,\ldots,k_\ell)$.

The main result of this paper is the following theorem.
\begin{theorem} \label{t:main theorem-intro}
Over $\tilde{H}_{g,n}(1,\ldots,1)$,
$g\geq 1$,
we have the following equality:
\begin{align}\label{e:main-equality}
\tau_B^{24}
= c\, \exp\Big(\, 4\pi \mathbb{CS} +\frac{1}{\pi} I\, \Big)\, {F}^{24}.
\end{align}
The same equality holds for the function $\tau_B^{24}$ on
$\tilde{\mathcal{H}}_{g}(1,\ldots,1)$, $g\geq 1$.
\end{theorem}

Here $c$ represents a constant, depending on $g$, $n$,
and a topological choice that will be explained in Section
\ref{s:proof-main-theorem}.
The complex-valued function $\mathbb{CS}$
on $\tilde{H}_{g,n}(1,\ldots,1)$ or $\tilde{\mathcal{H}}_{g}(1,\ldots,1)$
is defined as follows. Each marked compact
Riemann surface $X$ has a Schottky uniformization given by a unique
marked normalized
Schottky group $\Gamma$; the group naturally defines
an infinite volume hyperbolic 3-manifold $M_X$ whose conformal
boundary is $X$.
A 1-form $\Psi$ on $X$ (here $\Psi$ is either
$d\lambda$ for the meromorphic function $\lambda$, or
it is the holomorphic 1-form $\Phi$) determines
a singular framing on $X$, and there exists a singular
framing $s_\Psi$ on $M_X$ which extends the framing on $X$
in a sense we prescribe. In Section \ref{s:chern-simons}
we define an invariant $\mathbb{CS}(M,s)$ for a certain class
of 3-manifolds $M$ and singular framings $s$ on $M$. The
value of $\mathbb{CS}$ at a point corresponding to $(X,\Psi)$
is then defined to equal $\mathbb{CS}(M_X,s_\Psi)$. Our definition of
$\mathbb{CS}(M,s)$ is motivated by the work of Meyerhoff \cite{Mey}
and Yoshida \cite{Yo} for finite volume hyperbolic 3-manifolds with cusps.
In subsection \ref{ss:def-CS} we show
\begin{equation} \label{e:first-equality}
\mathbb{CS}(M,s)=\frac{1}{\pi^{2}}W(M)+ 2iCS(M,s),
\end{equation}
where $W(M)$ is the renormalized volume of $M$
(see \cite{Kr}, \cite{TT}, \cite{KS}; we use the
definition of
Section 8 of \cite{KS}), and $CS(M,s)$ is the integral of
the usual Chern-Simons 3-form over $M$ with the framing $s$,
together with a correction term corresponding to the singularities
of the framing.
Let us remark that $CS(M,s)$ is finite by
our construction without any renormalization process and is well defined
only up to $\frac12\mathbb{Z}$.

The function $I$ is real-valued, and is given by an explicit integral over
the Riemann surface, involving the 1-form $\Psi$.
We refer to $I$ as a regularized Polyakov integral, since it plays the role
of the usual Polyakov integral in relating the determinant of
the Laplacian in the hyperbolic metric on $X$ to that in the
flat singular metric on $X$ defined by $\Psi$, as we show in
Corollary \ref{c:polyakov-formula}. Its precise definition is
given in \eqref{e:def-I}  and  \eqref{e:new-def-I}.
The combined expression $\exp( 4\pi \mathbb{CS} +\frac{1}{\pi} I)$
gives a holomorphic function over $\tilde{H}_{g,n}(1,\ldots,1)$ or
$\tilde{\mathcal{H}}_g(1,\ldots,1)$ (although by itself, $I$ is actually a
function over $H_{g,n}(1,\ldots,1)$
or $\mathcal{H}_g(1,\ldots,1)$). The function $F$ is the
holomorphic function over $\mathfrak{T}_g$ defined by
Zograf in \cite{Z} (it is related to determinants of Laplacians---see below).

Theorem \ref{e:main-equality} allows us to interpret the Bergman
tau function as a higher genus generalization of the Dedekind eta
function.
When $g=1$,
it is known that $\tau_B=\eta(\tau)^{2}$
and ${F}=\prod_{m=1}^\infty(1-q^m)^2$  on
$\mathcal{H}_1\simeq\mathfrak{T}_1\times\mathbb{C}^\times$
where $q=e^{2\pi i\tau}$,
$\tau\in H^2\simeq\mathfrak{T}_1$, and by elementary
computation we have
$\mathbb{CS}=i\tau$ and $I=0$. Consequently in this case,
Theorem  \ref{t:main theorem-intro}
reduces to the 48-th power of the defining equation of the Dedekind
eta function
\[
\eta(\tau)= q^{\frac{1}{24}} \prod^\infty_{m=1} (1-q^m).
\]
In \cite{KK06}, \cite{KK}, it was shown that $\tau_B^{24}$
satisfies a modular property with respect to the mapping class group,
which reduces to the modular property of $\eta^{48}$ in genus 1.
Further, the function $F$ was shown in \cite{Z1}
(see also \cite{MT}) to have an
infinite product expansion on a subset of $\mathfrak{T}_g$:
\begin{equation}\label{e:F-expansion}
F=\prod_{\{\gamma\}}\prod^\infty_{m=1} (1-q_\gamma^{m}).
\end{equation}
Here the first product runs over all primitive closed geodesics $\gamma$
in $M_X$, and the complex number $q_\gamma$ has modulus
$e^{-\text{length}(\gamma)}$ and argument given by the
holonomy around $\gamma$ in an orthogonal plane.
The equation \eqref{e:F-expansion} is valid whenever the exponent
of convergence $\delta$
of $\Gamma$ is strictly less than 1.

The relation between objects on the 2-manifold $X$
and the bounding infinite volume 3-manifold $M_X$
given by Theorem \ref{t:main theorem-intro} fits well with
principle of ``holography''---for example, see \cite{Ma}
and \cite{TT}. In this context, the Schottky uniformization
provides a natural choice of bounding 3-manifold $M_X$.

In \cite{KK04}, Kokotov and Korotkin showed that
the Bergman tau function $\tau_B$ is related to the
isomonodromic tau function $\tau_I$ for
$\tilde{H}_{g,n}(k_1,\ldots,k_\ell)$ considered as an underlying
space of a Frobenius manifold in the sense of Dubrovin in
\cite{Du}, \cite{DZ}, by the equation
$\tau_B=\tau_I^{-2}$. This implies the corollary
\begin{corollary}
Over $\tilde{H}_{g,n}(1,\ldots,1)$,
$g\geq 1$,
we have the following equality:
\begin{align}
\tau_I^{48}
= c\, \exp\Big(-4\pi \mathbb{CS} -\frac{1}{\pi} I\, \Big)\,
	{F}^{-24}.
\end{align}
\end{corollary}
Here and below, as in Theorem \ref{t:main theorem-intro},
$c$ represents a constant depending on $g$, $n$, and possibly
a topological choice. However, it does not always
represent the same constant.

To state the second corollary of Theorem \ref{t:main theorem-intro},
we need a result  about the phase of $F$.
In \cite{GMP}, it is
shown that the eta
invariant of the odd signature operator over $M_X$
is well-defined, without any additional renormalization,
and it is
proved that the phase of $F$ at $X$ is $\exp(-\frac{\pi i}{2}\eta(M_X))$,
whenever the marked Schottky group $\Gamma$ has exponent of
convergence $\delta<1$.  We refer
to \cite{GMP} for more details. Combining this with
\eqref{e:main-equality}, we have

\begin{corollary} \label{c:eta-phase}
The following equality holds
\[
\exp \Big(\, 8\pi i CS - 12 \pi i \eta \, \Big)
= c\, \left(\frac{\tau_B}{|\tau_B|}\right)^{24}
\]
over the
subset of $\tilde{H}_{g,n}(1,\ldots,1)$ or
$\tilde{\mathcal{H}}_g(1,\ldots,1)$, $g\geq 1$, for which the
corresponding marked Schottky group $\Gamma$ has exponent of
convergence $\delta<1$.
\end{corollary}

Let us remark that $\exp(4\pi i CS(M))= \exp(6\pi i \eta(M))$
for any closed $3$-manifold $M$.
Hence  Corollary \ref{c:eta-phase} generalizes this equality for
Schottky hyperbolic $3$-manifolds, where the boundary Riemann
surface $X$ produces a defect term given by the phase of
$\tau_B$.

The quantities in the main theorem are related to regularized
determinants of Laplacians. In \cite{Z} (see also \cite{MT}), it was shown that
\begin{equation*}\label{e:Zograf-det}
\frac{\mathrm{det}\Delta_{\mathrm{hyp}}}
{A_{\mathrm{hyp}}\det \langle\Phi_j,\Phi_k\rangle}
= c \exp\big(-\frac{1}{12\pi}S\big) |F|^2
\qquad \text{over} \quad \mathfrak{T}_g
\end{equation*}
where $\Delta_{\mathrm{hyp}}$ is the Laplacian in the
unique metric of constant curvature $-1$ on $X$,
$A_{\mathrm{hyp}}$ is the area of $X$ in that metric,
$\{\Phi_1,\ldots,\Phi_g\}$ is a basis of holomorphic
1-forms normalized with respect to the marking, and $S$
is the real valued classical Liouville action functional
over $\mathfrak{T}_g$.
Note that this is distinct from the usual expression of
$\mathrm{det}\Delta_{\mathrm{hyp}}$ in terms of the
Selberg zeta function; in particular, $F$ is holomorphic
in moduli.
It is known that $S(X)=-4W(M_X)$,
when $M_X$ is related to $X$ as above
(see \cite{Kr}, \cite{TT}, \cite{KS}).
In \cite{KK}, Kokotov and Korotkin showed that
\begin{equation}\label{e:KK-det}
\vert\tau_B\vert^2
=c \frac{\mathrm{det}\Delta_{\mathrm{flat}}}
	{A_{\mathrm{flat}}\det \langle\Phi_j,\Phi_k\rangle}
\qquad \text{over} \quad
	\tilde{\mathcal{H}}_g(1,\ldots,1)
\end{equation}
where $\Delta_{\mathrm{flat}}$ is the Laplacian in the
flat (singular) metric defined by $\Phi$, and
$A_{\mathrm{flat}}$ is the area of $X$ in that metric.
Combining these, we have
\begin{equation}\label{e:compare-Z-KK}
\vert\tau_B\vert^{24}
= c \exp \big(\frac{4}{\pi}W\big)
\Big(\, \frac{\mathrm{det}\Delta_{\mathrm{flat}}}{A_{\mathrm{flat}}}
	\cdot
	\frac{A_{\mathrm{hyp}}}{\mathrm{det}\Delta_{\mathrm{hyp}}}
	\, \Big)^{12}
	\, |F|^{24}  \qquad \text{over} \quad \tilde{\H}_g(1,\ldots,1).
\end{equation}
Observing that $\tau_B^{24}$ and $F^{24}$ in
\eqref{e:compare-Z-KK} are holomorphic functions
over $\tilde{\H}_g(1,\ldots,1)$, it is natural to
expect that there might exist a holomorphic function
over $\tilde{\H}_g(1,\ldots,1)$ whose modulus is
$\exp \big(\frac{4}{\pi}W\big)
\big(\, \frac{\mathrm{det}\Delta_{\mathrm{flat}}}{A_{\mathrm{flat}}}
	\cdot
	\frac{A_{\mathrm{hyp}}}{\mathrm{det}\Delta_{\mathrm{hyp}}}
	\, \big)^{12}$.
One motivation for this work was to find such a
holomorphic function, and Theorem \ref{t:main theorem-intro}
gives an answer to this question. Combining Theorem
\ref{t:main theorem-intro} and \eqref{e:compare-Z-KK},
and using the fact that $I$ descends to $\H_g(1,\ldots,1)$,
we have the following Polyakov formula,

\begin{corollary} \label{c:polyakov-formula}
\[
\frac{\mathrm{det}\Delta_{\mathrm{flat}}}{A_{\mathrm{flat}}}\cdot
\frac{A_{\mathrm{hyp}}}{\mathrm{det}\Delta_{\mathrm{hyp}}}
=c \, \exp\big(\frac{1}{12\pi} I \big)
\qquad \text{over} \quad \H_g(1,\ldots,1), \,g\geq 1.
\]
\end{corollary}

Note that the usual argument proving the Polyakov formula for
two smooth metrics does not apply in our case, since the domains
of $\Delta_{\mathrm{flat}}$ and $\Delta_{\mathrm{hyp}}$ are different.
Let us also remark that this formula can be proved combining the results in \cite{KKfirst}
and \cite{KK}.

We have restricted attention to $\tilde{H}_{g,n}(1,\ldots,1)$
and $\tilde{\mathcal{H}}_g(1,\ldots,1)$ for simplicity, but
we expect the results above will hold for other
$\tilde{H}_{g,n}(k_1,\ldots,k_\ell)$ and
$\tilde{\mathcal{H}}_g(k_1,\ldots,k_m)$,
with only minor adjustments in the definitions of
$\mathbb{CS}$ and $I$ and slight changes in the proofs.
We also note in passing that our constructions of
$\mathbb{CS}(M,s)$ and $I(X,\Psi)$ can be extended
in a straightforward way to apply when $M$ is any convex
co-compact hyperbolic 3-manifold with conformal boundary $X$.
In this case we expect that our methods will show that
$\exp(4\pi \mathbb{CS}+\frac{1}{\pi}I)$ is locally
a holomorphic function on the associated deformation space.
This is a parallel of Yoshida's result in \cite{Yo} for finite volume
hyperbolic manifolds with cusps, where $I$ is a new ``defect'' term
coming from the boundary of genus $g>1$.

In Section \ref{s:background}, we give the necessary background
and make precise definitions. In Sections \ref{s:framing} through
\ref{s:proof-main-theorem}, for simplicity of exposition, we present
the proof of Theorem \ref{t:main theorem-intro} over
$\tilde{\mathcal{H}}_g(1,\ldots,1)$ only. In the last Section, we
describe the necessary modifications to establish
Theorem~\ref{t:main theorem-intro} over $\tilde{H}_{g,n}(1,\ldots,1)$.

{\bf{Acknowledgments}}\ We are grateful to Leon Takhtajan for his helpful comments and questions
for the early version of this paper. We are also thankful to Aleksey Kokotov and Dmitry Korotkin
for useful discussions about their works on the Bergman tau function. The work of the second author is partially supported by
the SRC-GaiA.

\section{Preliminary background}\label{s:background}

\subsection{Hurwitz spaces and Tau functions} \label{ss:Hurwitz}

Let ${H}_{g,n}(k_1,\ldots,k_{\ell})$ be the Hurwitz space of
equivalence classes $[\lambda:X\to \mathbb{CP}^1]$ of $n$-fold
branched coverings
\begin{equation*}
\lambda:X\to \mathbb{CP}^1
\end{equation*}
where $X$ is a compact Riemann surface of genus $g$ and the
holomorphic map $\lambda$ of degree $n$ satisfies the following
conditions:
\begin{compactenum}[i)]
\item the map $\lambda$ has $m$ simple ramification points
	$p_1,\ldots,p_m\in X$ with distinct finite images
	$\lambda_1,\ldots,\lambda_m\in\mathbb{C}\subset \mathbb{CP}^1$,
\item the preimage $\lambda^{-1}(\infty)$ consists of $\ell$ points:
	$\lambda^{-1}(\infty)=\{q_1,\ldots,q_\ell\}$ and the ramification
	index of the map $\lambda$ at the point $q_j$
	is $k_j$ for $1\leq j\leq \ell$.
\end{compactenum}
Here two branched coverings $\lambda:X\to \mathbb{CP}^1$
and $\lambda':X'\to\mathbb{CP}^1$ are equivalent if there exists
a biholomorphic map $f:X\to X'$
such that $\lambda'\circ f =\lambda$. Note that $n=k_1+\cdots+k_\ell$
and $m=2g-2+n+\ell$ by the Riemann-Hurwitz formula.
We also introduce the covering $\hat{H}_{g,n}(k_1,\ldots,k_\ell)$
of the space $H_{g,n}(k_1,\ldots,k_\ell)$
consisting of pairs
\begin{equation*}
\big(\ [\lambda:X\to \mathbb{CP}^1], \{a_i,b_i\ | \ 1\leq i\leq g\}\ \big)
\end{equation*}
where $[\lambda:X\to \mathbb{CP}^1]\in H_{g,n}(k_1,\ldots,k_\ell)$
and $\{a_i,b_i\ |\ 1\leq i\leq g\}$ denotes a Torelli marking on $X$,
that is, a canonical basis of $H_1(X,\mathbb{Z})$. The space
$\hat{H}_{g,n}(k_1,\ldots,k_{\ell})$ is a connected complex manifold
of dimension $m=2g-2+n+\ell$, and the local coordinates on this
manifold are given by the finite critical values of the map $\lambda$,
that is, $\lambda_1,\ldots,\lambda_m$.

In \cite{KK04}, \cite{KK}, the Bergman tau function $\tau_B$ over
$\hat{H}_{g,n}(k_1,\ldots,k_{\ell})$ is defined in terms of the Bergman
kernel. The Bergman kernel on a Riemann surface $X$ with a Torelli
marking is defined by $B(p,q):=d_pd_q \log E(p,q)$ for $p,q\in X$
where $E(p,q)$ is the prime form on $X$. Near the diagonal $p=q$,
the Bergman kernel $B(p,q)$ has the expression
\begin{equation*}
B(z(p),z(q))
=  \big(\, \big(z(p)-z(q)\big)^{-2}
	+H\big(z(p),z(q)\big)\, \big)dz(p) dz(q)
\end{equation*}
where $z(p),z(q)$ are local coordinates of points $p,q$ in $X$,
and the Bergman projective connection $R_B$ is defined in a local
coordinate by
\begin{equation}\label{e:def-B-proj}
R_B(z(p))=6\lim_{q\to p} H(z(p),z(q)).
\end{equation}
The meromorphic function $\lambda$ also defines a projective
connection $R_{d\lambda}$, which is defined in a local
coordinate to be $\mathcal{S}(\lambda)$, where
$\mathcal{S}$ is the Schwarzian derivative defined by
\[
\mathcal{S}(f)
=\left(\frac{f_{zz}}{f_z}\right)_z
	-\frac{1}{2}\left(\frac{f_{zz}}{f_z}\right)^2.
\]
Now the Bergman tau function $\tau_B$ over
$\hat{H}_{g,n}(k_1,\ldots,k_{\ell})$ is locally defined to be a
holomorphic solution of  the system of compatible equations
\[
\frac{\partial \log \tau_B }{\partial \lambda_i}
	= \frac{\sqrt{-1}}{12\pi}
		\int_{s_i} \frac{R_B-R_{d\lambda}}{\lambda_z}\, dz
\qquad \text{for}\quad i=1,\ldots, m,
\]
where $s_i$ is a small circle around the ramification point
$p_i\in X$, in a local coordinate $z$ near $p_i$.  Note
that the difference $R_B-R_{d\lambda}$ is a meromorphic
quadratic differential and
$\frac{R_B-R_{d\lambda}}{\lambda_z}\, dz$
is a meromorphic $1$-form. It follows from \cite{KK06} that
$\tau_B^{24}$ is globally well-defined on
$\hat{H}_{g,n}(k_1,\ldots,k_{\ell})$.

The Bergman tau function $\tau_B$ is related to the
isomonodromic tau function $\tau_I$ of Dubrovin \cite{Du}, \cite{DZ}
by a theorem of Kokotov and Korotkin \cite{KK04}:
\begin{theorem}
\[
\tau_B = \tau_I^{-2}
	\qquad \text{over} \quad \hat{H}_{g,n}(k_1,\ldots,k_{\ell}) .
\]
\end{theorem}
Here $\hat{H}_{g,n}(k_1,\ldots,k_{\ell})$ is considered
as the underlying space of a Frobenius manifold where
the isomonodromic tau function $\tau_I$ is defined;
see \cite{Du}, \cite{DZ}, \cite{KK04} for details.

We also define the space $\mathcal{H}_g$ to be the
moduli space of pairs $(X,\Phi)$ where $X$ is a compact
Riemann surface of genus $g>1$ and $\Phi$ is a holomorphic
$1$-form over $X$. We denote by $\mathcal{H}_g(k_1,\ldots,k_m)$
the stratum of $\mathcal{H}_g$ consisting of differentials $\Phi$
which have $m$ zeroes on $X$ of multiplicities $(k_1,\ldots,k_m)$.
For more details about these spaces, we refer to \cite{KZ}.
As before, we also introduce a covering
$\hat{\mathcal{H}}_g(k_1,\ldots,k_m)$  of
$\mathcal{H}_g(k_1,\ldots,k_m)$
consisting of triples $(X, \Phi, \{a_i,b_i\ | \ 1\leq i\leq g\} )$
where $\{a_i,b_i\ | \ 1\leq i\leq g\}$ is a canonical basis of
$H_1(X,\mathbb{Z})$.

Cutting the Riemann surface along the cycles given by a
Torelli marking $\{a_i,b_i\ |\ 1\leq i\leq g\}$, we get the
fundamental polygon $\hat{X}$. Inside of $\hat{X}$ we choose
$(m-1)$-paths $l_j$ which connect the zero $p_1$ with the
other zeros $p_j$ for $j=2,\ldots,m$. The set of paths
$a_i,b_i,l_j$ gives a basis in the relative homology group
$H_1(X, (\Phi),\mathbb{Z})$ where $(\Phi)=\sum_{j=1}^m k_j p_j$
denotes the divisor of $\Phi$. Following \cite{KK}, local coordinates on
$\hat{\mathcal{H}}_g(k_1,\ldots,k_m)$ can be chosen as follows:
\begin{equation}\label{e:coordinates}
A_i:=\int_{a_i} \Phi, \qquad B_i:=\int_{b_i} \Phi,
	\qquad Z_j:=\int_{l_j} \Phi,
\end{equation}
where $i=1,\ldots,g$ and $j=1,\ldots,m-1$.
For simplicity, we also use another notation $\zeta_i$ for the
coordinates defined by
\begin{equation*}
\zeta_i:=A_i, \qquad \zeta_{g+i}:=B_i, \qquad \zeta_{2g+j}:=Z_{j+1}.
\end{equation*}
Define cycles $s_i$ for $i=1,\ldots,2g+m-1$ by
\begin{equation*}
s_i=-b_i, \qquad s_{g+i}= a_i
\end{equation*}
for $i=1,\ldots,g$ and define the cycle $s_{2g+i}$ to be
a small circle with positive orientation around $p_{i+1}$.

As before, Kokotov and Korotkin \cite{KK} also define
the Bergman tau function $\tau_B$ over the stratum
$\hat{\mathcal{H}}_g(k_1,\ldots,k_m)$ to be a holomorphic
solution of the following compatible system of equations:
\begin{equation}\label{e:def-B-tau-0}
\frac{\partial \log \tau_B}{\partial \zeta_i}
=\frac{\sqrt{-1}}{12\pi} \int_{s_i} \frac{R_B-R_\Phi}{h}\, dz
\qquad\text{for}\quad i=1,\ldots,2g+m-1,
\end{equation}
where $\Phi(z)=h(z)\, dz$ for a local coordinate $z$.
Here $R_B$ denotes the Bergman projective connection
defined in \eqref{e:def-B-proj} and $R_\Phi$ is the projective
connection given by the Schwarzian derivative
$\mathcal{S}(\int^z \Phi)$ with respect to a local coordinate
$z$. It is shown in \cite{KK} that $\tau_B$ does not depend
on the choice of the $l_j$, and that $\tau_ B^{24}$ is a
globally well-defined function on
$\hat{\mathcal{H}}_g(k_1,\ldots,k_m)$.

Finally we introduce covering spaces
$\tilde{H}_{g,N}(k_1,\ldots,k_{\ell})$ and
$\tilde{\mathcal{H}}_g(k_1,\ldots,k_m)$
of $\hat{H}_{g,N}(k_1,\ldots,k_{\ell})$ and
$\hat{\mathcal{H}}_g(k_1,\ldots,k_m)$
respectively, by marking an ordered set of
generators $\{a_i,b_i\ | \ 1\leq i\leq g\}$ of $\pi_1(X)$
rather than of $H_1(X,\mathbb{Z})$. There are canonical maps
from these spaces to the Teichm\"uller space $\mathfrak{T}_g$ of
marked Riemann surfaces of genus $g$. Note that the tau functions
$\tau_I$, $\tau_B$ can be lifted to these spaces.
For simplicity we will mainly work over the spaces
$\tilde{H}_{g,n}(1,\ldots,1)$  and $\tilde{\mathcal{H}}_g(1,\ldots,1)$
whose dimensions are $m=2g-2+2n$ and $4g-3$ respectively.

\subsection{Basic facts on Schottky groups and Schottky spaces}
	\label{ss:Schottky}

Given a compact Riemann surface $X$ of genus $g\geq 1$,
there exists a Schottky uniformization of $X$, described as
follows. A subgroup $\Gamma$ of ${PSL}_2(\mathbb{C})$
is called a \emph{Schottky group} if it is generated by
$L_1,\ldots, L_g$ satisfying the following condition:
there exist $2g$ smooth Jordan curves
$C_r$, $r=\pm1,\ldots, \pm g$,
which form the oriented boundary of a domain
$D\subset \hat{\mathbb{C}}=\mathbb{C}\cup\{\infty\}$
such that $L_rC_r=-C_{-r}$, $r=1,\ldots, g$ where
${PSL}_2(\mathbb{C})$ acts on $\hat{\mathbb{C}}$
in the usual way and the negative signs indicate opposite
orientation. Any Schottky group gives a compact Riemann
surface $X=\Gamma\backslash \Omega$ where
$\Omega=\cup_{\gamma\in\Gamma}\gamma D$ is the
set of discontinuity of the action of $\Gamma$ on
$\hat{\mathbb{C}}$, and every compact Riemann surface
arises in this way. A Schottky group is \emph{marked} if
it is equipped with a particular choice of ordered set of
free generators $L_1,\ldots, L_g$. If the Riemann surface
$X$ is marked, then requiring the $b_1,\ldots,b_g\in\pi_1(X)$
to map to $L_1,\ldots,L_g$ fixes the marked Schottky group up to
overall conjugation in $PSL_2(\mathbb{C})$.

We define a \emph{Schottky 3-manifold}
to be a smooth 3-manifold with boundary that is topologically
a closed solid 3-dimensional handlebody $\overline{M}:=M\cup X$,
where $M$ is the corresponding open handlebody, and the
boundary $X$ is a compact
smooth 2-dimensional surface. We call a Schottky 3-manifold
\emph{hyperbolic} if it is equipped with a
complete hyperbolic metric $g_M$ on $M$, and we call it
\emph{marked} if it is equipped with an ordered choice of
generators of $\pi_1(M)$.

Any compact Riemann surface $X$ with a
uniformization by a marked Schottky group $\Gamma$
gives a marked Schottky hyperbolic
3-manifold $M \cup X$ in the following way:
$M=\Gamma\backslash H^3$ (where ${PSL}_2(\mathbb{C})$
acts on $H^3$ in the usual way),
$X=\Gamma\backslash \Omega$, and the topology on
$M\cup X$ is that inherited from
$\overline{H^3}:=H^3\cup \hat{\mathbb{C}}$.
The choice of the ordered set of generators $L_1,\dotsc,L_g$
gives the marking on $\pi_1({M})$, by identifying elements
of $\Gamma$ with deck transformations of the universal
cover of $\overline{M}$. Conversely, by means of the
developing map, every marked Schottky hyperbolic
3-manifold $M$ arises from a marked Schottky group
in this way, and the group is unique up to
an overall conjugation in ${PSL}_2(\mathbb{C})$.
When a marked Schottky group $\Gamma$
and a marked Schottky hyperbolic 3-manifold $M\cup X$
correspond in this way, we will
say that the group $\Gamma$ \emph{uniformizes} the manifold
$\overline{M}=M\cup X$.

In summary, given a compact marked Riemann surface $X$,
we obtain a unique marked Schottky hyperbolic 3-manifold
$M\cup X$ whose conformal boundary is $X$.
We will sometimes write
$M=M_X$ if we want to emphasize that the manifold
$M$ is determined by the marked surface $X$.

For a fixed $g$, the \emph{Schottky space of genus $g$},
denoted by $\mathfrak{S}_g$, is the set of all marked
Schottky groups with $g$ generators, modulo overall
conjugation in ${PSL}_2(\mathbb{C})$. It is known
that $\mathfrak{S}_g$ has a canonical complex
manifold structure of dimension $3g-3$, and its universal
cover is the Teichm\"uller space $\mathfrak{T}_g$, with
the covering map being holomorphic. The generators
$L_i$, $i=1,\dotsc,g$, are holomorphic maps from
$\mathfrak{S}_g$ to ${PSL}_2(\mathbb{C})$. In view
of the uniformization discussed above, we implicitly
identify $\mathfrak{S}_g$ with the deformation space of
marked Schottky hyperbolic 3-manifolds.

Every Schottky hyperbolic 3-manifold is conformally compact:
in some neighborhood $N\subset \overline{M}$ of $X$,
there exists a smooth boundary defining function
$r: N \to \mathbb{R}_{\geq 0}$
such that
\begin{compactenum}[i)]
\item $r>0$ on $N \cap M$, $r=0$ on $X$, and $d r=0$
	restricted to $X$,
\item the rescaled metric $\overline{g}:=r^2 g_M$ extends
	smoothly to $N\cap \overline{M}$,
\item $|d r|_{\overline{g}}^2 = 1$ in $N$.
\end{compactenum}
We also write $\overline{g}$ for the extension of the metric
$\overline{g}$ to $N\cap \overline{M}$.
The conformal class
of the metric $\overline{g}\big\vert_{TX}$ is independent
of the choice of boundary defining function; hence the
choice of a metric $g_M$ induces a unique conformal
class of metrics on the conformal boundary $X$.
For genus $g>1$, in each
conformal class of metrics on $X$, there is a unique
hyperbolic metric $g_{X}$ of constant curvature $-1$.
For genus $g=1$, in each conformal class of metrics
on $X$ there is a unique flat metric $g_X$ in which
$\text{Area}(X)=1$.
We will need a parametrization of a neighborhood
$N\subset\overline{M}$ of the conformal boundary $X$.
If we demand that $\overline{g}\big\vert_{TX}$ is equal
to the metric $g_{X}$, then  the boundary
defining function satisfying the conditions above is unique.
For a sufficiently small $a>0$, this defining function $r$
determines an identification of $X\times [0,a)$ with a
subneighborhood $N_{[0,a)}\subset N$, by letting
$(p,t)\in X\times [0,a)$ correspond to the point obtained
by following the integral curve $\phi_t$ of
$\nabla_{\bar{g}} r$ emanating from $p$ for $t$ units
of time. Throughout the rest of the paper, we will fix
such an $a$. For this defining function $r$, the $t$-coordinate
is just $r$ and $\nabla_{\bar{g}}r$ is orthogonal to the
slices $X\times\{t\}$. Hence identifying $t$ with $r$ on
$X\times[0,a)$, the hyperbolic metric $g_M$ over $M$
has the form
\begin{equation*}
g_M= r^{-2} (g_r +dr^2)
\end{equation*}
over $N_{[0,a)}$,
where $g_r$ denotes a Riemannian metric over
$X^r:=X\times \{r\}$. See \cite{G} for more details.

\section{Framings over Schottky hyperbolic $3$-manifolds}
	\label{s:framing}

From here on, $\overline{M}=M\cup X$ will denote
a marked Schottky hyperbolic $3$-manifold with
conformal boundary $X$. In this section, we define what
we mean by a ``singular framing'' over $M$ or over $X$,
and we define a class of ``admissible'' singular framings
which we will use to define the Chern-Simons invariant.
We then describe how to assign, to each holomorphic
$1$-form $\Phi$ on $X$ with only simple zeroes, an
admissible singular framing on $X$. In Section
\ref{s:last section} we will describe how to relax the
assumptions on $\Phi$. Finally, we prove that an
admissible singular framing on $X$ ``extends'' (in a
sense to be defined below) to an admissible singular
framing on $M$.

\subsection{Admissible singular framings}

Let $F(M)$ denote the $SO(3)$ frame bundle with the
projection map $p:F(M)\to M$. For a subset
$U\subset M$, by a \emph{framing over $U$}
we mean a section of $F(M)$ over $U$.

Let $\L$ denote an union of disjoint simple curves in $M$.
A framing over $\L$ in $M$, written as
$(e_1(y), e_2(y), e_3(y))\in T_yM\oplus T_yM\oplus T_yM$
for each $y\in\L$, is called a \emph{reference framing on $\L$},
if $e_1(y)$ is tangent to $\L$ at each $y\in\L$.

Let $\mathcal{N}^\epsilon(\L)$ be an $\epsilon$-neighborhood
of $\L$ in the metric $g_M$. A choice of reference framing
$\kappa$ over $\L$ allows us to construct the
\emph{deleted $\epsilon$-tube around $\L$}, which by
definition we take to be a map
\[
\alpha: (0,\epsilon)\times \L\times S^1
	\to (\mathcal{N}^\epsilon(\L))\subset M,
\]
constructed as follows: for each
$ (\rho, y, v) \in (0,\epsilon) \times \L \times S^1$,
we take the unique geodesic starting at $y$ with
initial vector $ \cos(v) e_2(y) + \sin(v) e_3(y) $, and travel
a distance $\rho$ from $y$ to the point
$\alpha(\rho, y, v)$.

Given a reference
framing $\kappa$ on $\L$, we define the corresponding
\emph{reference framing of the deleted
$\epsilon$-tube around $\L$} by parallel translating the
reference framing $\kappa$ along the unique geodesic
connecting $y$ and $\alpha(\rho, y, v)$. This gives
a lifting
\[
\tilde{\alpha}: (0,\epsilon)\times \L\times S^1
	\to p^{-1}(\mathcal{N}^\epsilon(\L))\subset F(M)
\]
of the map $\alpha$. The \emph{standard cylinder over $\L$}
is the map
\[
\psi :  \L\times  S^1\to p^{-1}(\L) \subset F(M)
\]
which takes the point $(y,v)\in \L \times S^1$ to the framing
\[
\psi (y,v)
	:= (e_1(y), \cos(v) e_2(y) + \sin(v) e_3(y), -\sin(v) e_2(y) + \cos(v) e_3(y))
\]
at the point $y$.

A matrix function
\[
A: (0,\epsilon) \times \L\times S^1\to {SO}(3)
\]
acts on a framing $\tilde{\alpha}$ of the deleted
$\epsilon$-tube around $\L$ by fiberwise right
multiplication:
\[(e_1, e_2, e_3)\cdot A(\rho, y,v)
=(\sum_{i=1}^3 e_i a_{i1},\sum_{i=1}^3 e_i a_{i2},\sum_{i=1}^3 e_i a_{i3}),
\]
over a point $\alpha (\rho, y,v)$ where $a_{ij}$
denotes $(i,j)$-entry of $A(\rho, y,v)$. We denote the
resulting framing by $\tilde{\alpha}\cdot A$. A matrix
function $A: \L\times S^1 \to {SO}(3)$ acts on the standard
cylinder $\psi$ to give $\psi\cdot A$ in the same fashion.

For a connected simple curve $\ell\subset M$,
the \emph{special singularity of index $n$ at $\ell$} is the
framing $\tilde{\alpha}\cdot A_n$  over the deleted
$\epsilon$-tube around $\ell$, where $\tilde{\alpha}$
is the reference framing on the deleted $\epsilon$-tube
around $\ell$, and $A_n$ is the matrix function on
$(0,\epsilon) \times \ell \times S^1$ defined by
\[
A_n(\rho, y,v)=\begin{pmatrix}
1 & 0 & 0 \\
0 & \cos(nv) & -\sin(nv) \\
0 & \sin(nv) & \phantom{-}\cos(nv) \\
\end{pmatrix}.
\]
For fixed $y\in\ell$ and $v\in S^1$,
the limit of $\tilde{\alpha}\cdot A_n$ as $\rho\to 0$ exists,
and equals the framing
$(e_1(y), \cos(nv)e_2(y)+\sin(nv)e_3(y), -\sin(nv)e_2(y)+\cos(nv)e_3(y))$
over $y$.
Hence the map consisting of these limits as $\rho\to 0$ for
all $y\in\ell$ and $v\in S^1$ is given by $n$-copies of
the standard cylinder over $\ell$. Here a negative integer
$n$ indicates opposite orientation. For $\L$ a disjoint union
of simple curves, we say that a framing $\F$ over $M\setminus \L$
has a \emph{special singularity  at $\L$} if $\F\circ \alpha $
has the special singularity of index $n$ for an integer $n$ on
each connected component of $(0,\epsilon)\times \L \times S^1$.
Let us remark that $n$ could be different over each component
of $\L$. Our definition of special singularity coincides with
Meyerhoff's \cite{Mey} when $n=1$.

For a connected simple curve $\ell\subset M$, the
\emph{admissible singularity of index $n$ at $\ell$}
is the special singularity framing of index $n$ at $\ell$, acted
on by a matrix function $A$:
\begin{equation}\label{e:def-A-0}
\tilde{\alpha}\cdot A_n \cdot A:
	(0,\epsilon)\times \ell \times S^1
		\to p^{-1}(\mathcal{N}^\epsilon(\ell))\subset F(M),
\end{equation}
where $A: (0,\epsilon)\times \ell\times S^1\to {SO}(3)$
satisfies the condition that $\lim_{\rho\to 0} A(\rho,y,v)$
exists and is independent of $v$, for all $y\in\ell$ and
$v\in S^1$. We say that a framing $\F$ over $M\setminus \L$
has an \emph{admissible singularity at $\L$}  if the limit
of $\F\circ\alpha$ as $\rho\to 0$ exists for all $y\in\L$
and $v\in S^1$ and the map given by this limit
is the same as the map given by the limit of
$\tilde{\alpha}\cdot A_n\cdot A$ as $\rho\to 0$,
that is, $n$-copies of the standard cylinder
acted by $A$ over each connected component of $\L$.

Recall that, on a neighborhood of $X$ in $\overline{M}$,
we have a rescaled metric $\overline{g}=r^2 g_M$
which extends to $X$ and coincides with the
metric $g_{X}$ there.
Now, an \emph{admissible singular framing} $(\F,\kappa,\L)$
over $M$ consists of a union of disjoint simple curves $\L$ in $M$,
a reference framing $\kappa$ over $\L$, and a framing $\F$ over
$M\setminus\L$, satisfying
\begin{compactenum}[i)]
\item the closure $\overline{\L}$ is smooth in $\overline{M}$,
	and $\overline{\L}$ is orthogonal to $X$ in $\overline{g}$
	at the intersection,
\item the framings $r^{-1}\F$ and $r^{-1}\kappa$ extend
	smoothly to $\overline{M}\setminus\overline{\L}$ and
	$\overline{\L}$ respectively,
\item the first vector $e_1$ of $\F$ is tangent to the gradient
	flow curves of $r$ over $N_{(0,\epsilon)}\setminus \L$ for
	$0<\epsilon <a$, and
\item the framing $\F$ has an admissible singularity at $\L$.
\end{compactenum}

Let $\ell_1,\ldots,\ell_g$ be closed curves in $M$
representing the marked generators of $\pi_1(M)$, with
the property that there exist discs $D_1,\ldots,D_{g-1}$
such that $M\setminus\cup D_i$ is the disjoint union
of $g$ solid tori $\ell_i\times D$, where $D$ is
the unit disc.
Given an admissible singular framing $(\F,\kappa,\L)$,
define $\L^1$ to be the set of connected components of
$\L$ that are closed, and define $\L^2:=\L\setminus \L^1$.
Then $(\F,\kappa,\L)$ will be called \emph{standard}
if
\begin{compactenum}[i)]
\item $\F$ has a special singularity of index $1$ at each curve in $\L^1$ where the set $\L^1$ is a subset of $\{\ell_1,\ldots,\ell_g\}$ and
\item the index of the admissible singularity of $\F$ at each curve in $\L^2$ is $-1$.
\end{compactenum}

We define an admissible singular framing on a
surface $X$ with the metric $g_{X}$ in a similar way. Let
$Z$ consist of finitely many points in $X$. A reference
framing on $Z$ is a choice of a frame $(e_2,e_3)$ at
each point $z\in Z$, orthonormal with respect to the
metric $g_{X}$. A reference framing on $Z$ defines a
geodesic polar coordinate
$\alpha: (0,\epsilon)\times Z\times S^1 \to
	\mathcal{N}^\epsilon(Z)\setminus Z$
which takes $(\rho,z,v)$ to the point at distance $\rho$
from $z\in Z$ along the geodesic with initial vector
$\cos(v) e_2(y) + \sin(v) e_3(y)$.
Parallel translation gives a corresponding reference framing
$\tilde{\alpha}$ over $ (0,\epsilon)\times Z\times S^1$.
The special singularity of index $n$ at $z\in Z$ is the framing
$\tilde{\alpha}\cdot A_{n}$ on
$(0,\epsilon)\times \{z\}\times S^1$ where $\tilde{\alpha}$
denotes the reference framing and $A_{n}$ is the matrix function
given by
\[
A_n(\rho,v)
=\begin{pmatrix}
\cos (nv) & -\sin(nv) \\
\sin (nv) & \phantom{-}\cos (nv)
\end{pmatrix}.
\]

An admissible singularity of index $n$ at $z$ is the special singularity,
right-multiplied by a matrix function $A(\rho,z,v)$ with the property
that $\lim_{\rho\to 0} A(\rho,z,v)$ exists and is independent of $v$.
An admissible singular framing $(\F,\kappa,Z)$ on $X$ consists of
a finite set $Z$ in $X$, a reference framing on $Z$, and a framing
$\F$ of $X\setminus Z$ such that the limit of $\F$ as $\rho\to 0$
exists for all $v\in S^1$ and the map given by  this limit
is the same as the map given by the limit of an admissible singularity
at each point of $Z$.

\subsection{Admissible singular framings associated
to holomorphic 1-forms}\label{ss:adm-sing}

Suppose that $X$ is a Riemann surface, with metric $g_{X}$
compatible with its complex structure.
We now describe how to assign, to a holomorphic 1-form $\Phi$
with only simple zeroes, an admissible singular framing with
index $-1$ singular points at the zeroes of $\Phi$.

The metric $g_{X}$ is a collection
$\{ e^{\phi_{\alpha}} |dz_\alpha|^2 \}_{\alpha\in A}$ on
an atlas
$\{(U_\alpha,z_{\alpha})\}_{\alpha\in A}$ of $X$
for which the functions
$\phi_\alpha\in C^\infty(U_\alpha,\mathbb{R})$
satisfy
\begin{equation}\label{e:metric-patch}
\phi_\alpha+\log|f'_{\alpha\beta}(z_\beta)|^2
	=\phi_\beta
	\qquad \text{on} \quad U_\alpha\cap U_\beta,
\end{equation}
where
$f_{\alpha\beta}=z_\alpha\circ z_\beta^{-1}
	: z_\beta(U_\alpha\cap U_\beta)
		\to z_\alpha(U_\alpha\cap U_\beta)$
are the holomorphic transition functions.
A holomorphic $1$-form $\Phi$ on $X$ is a collection
$\{h_\alpha dz_\alpha\}$ for the atlas
$\{(U_\alpha,z_{\alpha})\}$
for which $h_\alpha$ is a holomorphic function on
$U_\alpha$ satisfying
\begin{equation}\label{e:h-trans}
h_\alpha  f_{\alpha\beta}'(z_\beta)
	=h_\beta  \qquad \text{on} \quad U_\alpha\cap U_\beta.
\end{equation}
The phase function $e^{i\theta_\alpha}:=h_\alpha/|h_\alpha|$
is well defined over $X\setminus Z$ where $Z$ denotes the zero
set of $\Phi$.
The transformation law \eqref{e:h-trans} implies
\begin{equation}\label{e:theta-patch}
i\theta_\alpha
	+\log \frac{f'_{\alpha\beta}(z_\beta)}{|f_{\alpha\beta}'(z_\beta)|}
	= i\theta_\beta \qquad \text{on} \quad U_\alpha\cap U_\beta.
\end{equation}
Note that $\theta_\alpha$ is defined only up to an integer multiple
of $2\pi$.
By \eqref{e:metric-patch}, \eqref{e:theta-patch}, it follows that
$e^{\phi_\alpha/2+i\theta_\alpha} dz_\alpha$ defines an orthonormal
co-framing $\omega_2,\omega_3$ given by
\[
{\omega_2}_\alpha
	= e^{\phi_\alpha/2}
		(\cos\theta_\alpha dx_\alpha -\sin\theta_\alpha dy_\alpha),
\qquad {\omega_3}_\alpha	
	= e^{\phi_\alpha/2}
		(\sin\theta_\alpha dx_\alpha+ \cos\theta_\alpha dy_\alpha)
\]
on $U_\alpha\setminus  Z $ where
$z_\alpha=x_\alpha+iy_\alpha$.  Now we obtain an orthonormal
framing
\[
\F_{\Phi}=(f_2,f_3) \qquad \text{where}
	\quad f_2=\omega_2^*, f_3=\omega^*_3
\]
over  $X\setminus Z$, which has admissible singularities
at $Z$ of index $-1$.

For the singular part $Z$, let  $z_{i\alpha}$ denote the
co-ordinate of a zero of $\Phi$ in a patch $U_\alpha$.
Then $h_\alpha$ has an expression
$h_\alpha=(z_\alpha-z_{i\alpha})\tilde{h}_{i\alpha}$,
where $\tilde{h}_{i\alpha}$ is non-vanishing at the zero.
Now we put
$e^{i\tilde{\theta}_{i,\alpha}}
	:= \tilde{h}_{i,\alpha}/|\tilde{h}_{i,\alpha}|$.
Since $\tilde{h}_{i\alpha}$ is non-vanishing at the zero,
$\tilde{\theta}_{i\alpha}$ is well-defined at the zero up to an integer
multiple of $2\pi$. By \eqref{e:metric-patch}, \eqref{e:h-trans}, it follows that
$e^{\frac12(\phi_\alpha+i\tilde{\theta}_{i,\alpha})} dz_\alpha$ defines
the following orthonormal co-framing at the zero,
\begin{equation}\label{e:frame-by-theta-til}
\begin{split}
\tilde{\omega}_{2\alpha}
	=& \ \ e^{\phi_\alpha/2}
		(\cos({\tilde{\theta}_\alpha}/{2}) dx_\alpha
-\sin({\tilde{\theta}_\alpha}/{2}) dy_\alpha),\\
 \tilde{\omega}_{3\alpha}=& \ \
	e^{\phi_\alpha/2}
		(\sin({\tilde{\theta}_\alpha}/{2}) dx_\alpha
			+ \cos({\tilde{\theta}_\alpha}/{2}) dy_\alpha),
\end{split}
\end{equation}
and the corresponding orthonormal framing
$(\tilde{f}_2$, $\tilde{f}_3)$ at the zero.
By the transformation law for $\tilde{h}$, this orthonormal framing
transforms correctly under change of coordinate.
Note however that this co-frame and frame are well defined only up to sign.

We select $g-1$ of the points in $Z$ to have the framing
$(\tilde{f}_2,\tilde{f}_3)$, and let the other $g-1$ points in $Z$
have the framing $(\tilde{f}_2,-\tilde{f}_3)$;
we denote the resulting framing at $Z$ by $\kappa_\Phi$. When we
extend the framing $\F_\Phi$ to $M$, these will correspond to
``outgoing'' and ``incoming'' endpoints of curves in $M$ respectively.

\subsection{Existence of admissible extensions}
	\label{ss:existence-admissible}
On a subset of $X$, we can identify any
${SO}(2)$ framing with respect to
$g_{X}$ with an ${SO}(3)$ framing with respect
to $\overline{g}$, by taking each framing $(f_2,f_3)$
to the framing $(f_1,f_2,f_3)$, where $f_1$ is the inward
unit normal vector to $X$ with respect to $\overline{g}$.
We say that an admissible singular framing $(\F_X,\kappa_X,Z)$
has an \emph{admissible extension}
to
$M$ if there exists an admissible singular framing
$(\F,\kappa,\L)$ over $M$ such that $\partial\overline{\L}=Z$,
and such that the extension of $r^{-1}\F$ and $r^{-1}\kappa$
equals the given framing $\F_X$ and $\kappa_X$, respectively,
under the identification above.

Now, our goal is to show that, for a holomorphic 1-form
$\Phi$ with only simple zeroes on $X$, the associated
admissible singular framing $(\F_\Phi,\kappa_\Phi,Z)$
on $X$ extends to an admissible singular framing
$(\F,\kappa,\L)$ on $M$. (A similar proof shows that
any admissible singular framing on $X$ extends to $M$.)

Before proving the existence of such an admissible extension,
we establish two lemmas.

\begin{lemma}\label{l:framing-ext}
Suppose $\overline{W}=W \cup \partial W$ is a
marked smooth 3-dimensional closed handlebody of genus
$p$ with metric $g_{\overline{W}}$, and suppose that
$\F_{\partial W}$ is a smooth (non-singular)
${SO}(3)$ framing of $\partial W$. Then there exists an
admissible extension of $\F_{\partial W}$ to $W$ which
has a special singularity of index $1$ at $\L^1$. Its set of
singular curves $\L^1$ may be taken to consist of at most
$p$ closed curves, each representing a distinct marked
generator of $\pi_1(W)$.
\end{lemma}

\begin{proof}
There exists a smooth embedding of $W$ into
$\mathbb{R}^3$, which gives a global framing $\F_0$ on
$W$, by which we can identify any other framing on $W$
with a map to $SO(3)$. Let $\L^0$ be the union of $p$
closed simple curves representing the marked generators
of $\pi_1(W)$. Given a connected curve $\ell$ in $\L^0$,
there exists a disc $D$ in $W$ such that $W\setminus D$ is
the disjoint union of a handlebody of genus $p-1$ and a solid
torus $T$ satisfying $T\cap \L^0 = \ell$ and
$\partial T\simeq \ell\times S^1$. Since $\partial D$ is
homologically trivial in $\partial W$, it is a commutator
in $\pi_1(\partial W)$ and so its image in
$SO(3)$ under the framing $\F_{\partial W}$ is
homotopically trivial. Hence $\F_{\partial W}$ can be smoothly
extended to $D\subset \partial(W\setminus D)$.
In this way the problem reduces to finding a framing on
each solid torus $T$. If $\pi_1(T)$ is represented by $\ell$,
identify $\partial T$ with $\ell\times S^1$. The image of
this $S^1$ in $SO(3)$ given by $\F_{\partial W}$ is either
homotopically trivial, in which case the framing extends
smoothly to all of $T$, or it is homotopically nontrivial, in
which case the framing has the same homotopy type as a special
singularity framing of index $1$ around $\ell$ and can thus
be extended to a framing on $T\setminus \ell$ with this singularity.
\end{proof}

From now on, we put $a_1=\frac a4$ for simplicity,
where $a$ is defined as in subsection
\ref{ss:Schottky}.

\begin{lemma}\label{l:framing near bottom}
Let $\overline{M} = M \cup X$ be a marked Schottky hyperbolic
3-manifold, and let $a>0$ be such that the neighborhood
$N_{[0,a]} \subset \overline{M}$ of $X$ exists. Let
$\Phi$ be a holomorphic 1-form with only simple zeroes on $X$ and
$(\F_\Phi,\kappa_\Phi,Z)$ be the associated admissible singular
framing as defined above. Then $(\F_\Phi,\kappa_\Phi,Z)$ has
an admissible extension to $N_{(0,a_1]}$.
\end{lemma}

\begin{proof}
If $Z$ is the singular set of the framing $\F_\Phi$ on $X$, then we
can take the set of singular curves to be
the $g_M$ geodesics given by
$\L=\{\phi_r(x): x\in Z, r\in (0,a_1]\}$.
Given an admissible singular framing $\F_\Phi=(f_1,f_2,f_3)$ over
$X \setminus Z$ with respect to
$\bar{g}=r^2g_M$, one can find an admissible singular framing
$\F=(e_1,e_2,e_3)$
with respect to $g_M$ that is parallel near infinity
and extends $\F_\Phi$, by rewriting the parallel
transport equation for $e_i$ with respect to $g_M$
in terms of $b_i$, where
$e_i(r)=rb_i(r)
	= r(b_i^1(r) \frac{\partial}{\partial t}
		+b_i^2(r) \frac{\partial}{\partial x}
			+b_i^3(r) \frac{\partial}{\partial y})$.
The parallel transport equation along the gradient flow curve $\phi_r$
becomes
\begin{equation*}
r\dot{b_i^m}(r)+b_i^m(r)
	+r\sum_{j,k} \Gamma^m_{j,k}(\phi_r) \dot{\phi}^j_r b_i^k(r)
	=0,
\end{equation*}
and we use the solution, with initial conditions $b_i(0)=f_i$,
to define $e_i$. We extend the reference
framing on $\L$ in the same manner, using the reference
framing on $Z$ as the initial condition.
\end{proof}

\begin{theorem}\label{t:framing-ext}
If $\overline{M} = M \cup X$ is a marked Schottky hyperbolic
3-manifold and $\Phi$ is a holomorphic 1-form with only
simple zeroes on $X$, then the associated admissible singular
framing $(\F_\Phi,\kappa_\Phi,Z)$ on $X$ extends to an
admissible singular framing $(\F,\kappa,\L)$ on $M$. The framing
$(\F,\kappa,\L)$ can be taken to be standard.
\end{theorem}

\begin{proof}
We begin by defining the $\L^2$ part of the singular curve
of $\F$.  In Lemma \ref{l:framing near bottom}, the $\L^2$
part in $N_{(0,\frac{a}{4}]}$ is defined to be the gradient
flow curves. Now we extend them by taking pairs of two ends
in $X^{a_1}$ of those curves and making curves to connect
them smoothly within $N_{(0,a)}$. We may assume that each
connected curve $\ell_i$, $i=1,\ldots, g-1$ in $\L^2$
meets level surface $X^\epsilon$ at two points for
$a_1\leq \epsilon < \frac{a}2$ and at one point for
$\epsilon=\frac{a}2$.
By construction, the end points of $\L^2$ are given by the zero
set $Z=\{p_1,\ldots,p_{2g-2}\}$ of $\Phi$.
As we mentioned in the end of subsection \ref{ss:adm-sing}, we
may assume that if the reference framing is taken to be
$(\tilde{f}_2,\tilde{f}_3)$ on one end of $\ell_i$, then the
reference framing is taken to be $(\tilde{f}_2,-\tilde{f}_3)$
on the other end of $\ell_i$.

Let us choose a reference framing $\kappa^2$ on $\L^2$ which
extends $(\tilde{f}_2,\tilde{f}_3)$ and $(\tilde{f}_2,-\tilde{f}_3)$
at each end point respectively, and which satisfies the parallel
condition over $\L^2\cap N_{(0,a_1]}$.
We also let $\F$ be the admissible extension  of $\F_\Phi$ on
the set $N_{(0,\frac{a}{4}]}$ guaranteed to exist by
Lemma~\ref{l:framing near bottom}. Note that $\F$ has an
admissible singularity of index $-1$ at  $\L^2\cap N_{(0,a_1]}$
by definition.

Now we define $\F$ over
$ \mathcal{N}^\epsilon(\L^2)\cap N_{[a_1, a)}$
so that $\F$ has an admissible singularity of index $-1$ at
$\L^2\cap N_{[a_1, a)}$. Let $\beta_i$ be a diffeomorphism from
$\overline{\ell}_i\subset \overline{M}$ to $[-1,1]$ which maps the end
with the reference framing $(\tilde{f}_2,\tilde{f}_3)$
to $-1$ and the end with
the reference framing $(\tilde{f}_2,-\tilde{f}_3)$ to $1$, and  maps
$\ell_i\cap  N_{[a_1, a)}$ to $[-\frac12,\frac12]$. Let $\xi$ be a
smooth increasing function on the interval $[-1,1]$ whose derivative
is supported in $(-\frac13,\frac13)$ whose values are $0$ on
$[-1,-\frac13]$ and $\pi$ on $[\frac13,1]$. We define
$\mathfrak{\chi}:\overline{\L^2}\to [0,\pi]$ by the composition of
$\xi$ and $\beta_i$ over $\ell_i$ and let
\begin{equation}\label{e:matrix-A}
A(\rho,v,y)
=\begin{pmatrix}
	\cos \chi(y) & 0 & -\sin \chi(y) \\
	0 & 1 & 0\\
	\sin \chi(y) & 0 & \cos \chi(y)
\end{pmatrix}
\qquad \text{on} \quad
	(0,\epsilon)\times ({\L^2}\cap N_{[\frac{a}3,a)})\times S^1
\end{equation}
and $A$ over
$(0,\epsilon)\times ({\L^2}\cap N_{[a_1,\frac{a}3]})\times S^1$
is defined to connect the above matrix in \eqref{e:matrix-A} and
the matrix $A$ determining the admissible framing
$\F$ over $\mathcal{N}^\epsilon(\L^2)\cap X^{a_1}$.
We may assume that $\lim_{\rho\to 0} A(\rho,v,y)$ exists and is
independent of $v$, for all $y\in\L^2$ and $v\in S^1$.
Then, for the reference framing $\tilde{\alpha}$ of the deleted
$\epsilon$-tube around $\L$ obtained from $\kappa^2$, we define
$\F$ by the equality
$\F\circ\alpha=\tilde{\alpha}\cdot A_{-1}\cdot A$ over
$\mathcal{N}^\epsilon(\L^2)\cap N_{[a_1, a)}$, which extends
the previously constructed framing $\F$ over $N_{(0,a_1]}$.
Note that this extension of $\F$ is independent of the choice of a
reference framing $\kappa^2$ on $\L^2$ satisfying the
conditions above. In particular, the extension of $\F$ does not depend
on the
choice of signs in $\kappa_\Phi$.
By definition, this framing $\F$ has an admissible singularity of index
$-1$ at $\L^2\cap N_{[a_1, a)}$.

So far an admissible framing $\F$ has been constructed over
$N_{(0,a_1]}\cup \mathcal{N}^\epsilon(\L^2)$.
Now we extend it over $M\setminus (\L^1\cup \L^2)$
by appropriately choosing $\L^1$. First let
$W_0$ denote the closure of
$M^{a_1}\setminus \mathcal{N}^\epsilon(\L^2)$ where $M^{a_1}=M\setminus N_{(0,a_1)}$.
Then there is a homotopy which deforms $W_0$ to a
closed handlebody $W_1$ of genus $2g-1$. Given a
set of generators of $\pi_1(M)\simeq \pi_1(M^{a_1})$,
there exist $(g-1)$-closed discs $D_i \subset W_1$,
$i=1,\ldots, g-1$ such that these decompose $W_1$
into one handlebody of genus $g$ and solid tori $T_i$,
$i=1,\ldots, g-1$ satisfying the following conditions:
the decomposed handlebody of genus $g$ contains
the homotopic images of loops realizing the given
generators of $\pi_1(M^{\frac{a}{4}})$. For a generator
$\tilde{\gamma}_i$ of $\pi_1(T_i)$, there is a closed curve
$\gamma_i$ in $W_0$ given by the (inverse) homotopic
image of the loop realizing $\tilde{\gamma}_i$.
By this construction, the set $G$ of generators of
$\pi_1(W_0)$ is given by the union of the chosen
generators of $\pi_1(M^{a_1})$ by marking
and the set of $\gamma_1,\ldots,\gamma_{g-1}$.

Applying Lemma \ref{l:framing-ext} for the framing
defined as above over the boundary of the closure of
$W_0$, we obtain an admissible extension of
$(\F_\Phi, \kappa_\Phi,Z)$.
To show that we can take it to be standard,
we have to modify the construction so that $\L^1$
consists of representatives of the marked generators of
$\pi_1(M)$. Suppose that $\L^1$ contains a representative
of a generator $\gamma_i$. Then we may replace the
reference framing $\tilde{\alpha}$ with another framing
with an additional rotation $2\pi$ along the corresponding
part of $\L^2$. This will change the homotopy type of the
admissible singular framing $\F$ along it since
$\pi_1({SO}(3))=\mathbb{Z}/2\mathbb{Z}$.
Hence it can be extended over the subset of $W_0$
corresponding to $T_i$ without removing a curve
representing $\gamma_i$. This means $\L^1$ can be
taken to represent a subset of the given generators of
$\pi_1(M)$. Then this completes the proof.
\end{proof}

\section{Definition of the invariant $\mathbb{CS}$}
	\label{s:chern-simons}

\subsection{The form $C$ on $PSL_2(\mathbb{C})$}

If $H^3$ is the hyperbolic space of dimension $3$,
the frame bundle $F(H^3)$ can be identified with
$PSL_2(\mathbb{C})$ canonically.
Let
\[
h=\begin{pmatrix} 1 & 0 \\ 0 &-1 \end{pmatrix}, \ \
e=\begin{pmatrix} 0 & 1 \\ 0 & 0 \end{pmatrix}, \ \
f=\begin{pmatrix} 0 & 0 \\ 1 & 0 \end{pmatrix}.
\]
Then
$\{h,e,f\}$ form a base of the Lie
algebra $\mathfrak{sl}_2(\mathbb{C})$ of $PSL_2(\mathbb{C})$.
Let $\{h^*_{\mathbb{C}},e^*_{\mathbb{C}},f^*_{\mathbb{C}}\}$ be its dual base of
$\mathrm{Hom}_{\mathbb{C}}(\mathfrak{sl}_2(\mathbb{C}),\mathbb{C})$.
In Section 3 in \cite{Yo}, Yoshida defines the form
$C$ as the left-invariant differential form on $PSL_2(\mathbb{C})$ whose value
at the identity is given by$\frac{i}{\pi^2}h^*_{\mathbb{C}}\wedge e^*_{\mathbb{C}}\wedge f^*_{\mathbb{C}}$, and
proves the following:

\begin{proposition} The form $C$
on ${PSL}_2(\mathbb{C})$ is complex analytic, closed, and bi-invariant,
and has the following expression
\begin{align*}
C=&\ \ \frac{1}{4\pi^2}( 4\,\theta_1\wedge \theta_2\wedge \theta_3
					- d(\theta_1\wedge \theta_{23}
						+\theta_2\wedge \theta_{31}
						+ \theta_3\wedge \theta_{12}))\\
 &+ \frac{i}{4\pi^2} (\theta_{12}\wedge \theta_{13}\wedge \theta_{23}
 					- \theta_{12}\wedge \theta_1\wedge \theta_2
					-\theta_{13}\wedge\theta_1\wedge \theta_3
					- \theta_{23}\wedge \theta_2\wedge \theta_3) .
\end{align*}
Here $\theta_i$ and $\theta_{ij}$ denote the fundamental
form and the connection form respectively on
${PSL}_2(\mathbb{C})$ of the Riemannian connection of $H^3$.
\end{proposition}

Since $H^3$ has constant sectional curvature $-1$,
$\Omega_{ij}=-\theta_i\wedge \theta_j$ for $i,j=1,2,3$.
Thus $C$ is a complex analytic form on $PSL_2(\mathbb{C})$
whose real part, up to scalar multiplication, is the volume form
plus an exact form, and whose
imaginary part, up to scalar multiplication, is the Chern-Simons form
defined in \cite{CS}. Using the equalities
$d\theta_i=-\sum_j \theta_{ij}\wedge \theta_j$,
$d\theta_{ij}=-\sum_k \theta_{ik}\wedge\theta_{kj}+\Omega_{ij}$,
one can obtain
\begin{proposition}\label{p:new express} The form $C$ on
${PSL}_2(\mathbb{C})$ has the following expressions
\begin{align*}
C&=-\frac{i}{4\pi^2}\, \eta\wedge d\eta
\\
&= -\frac{1}{4\pi^2}
		(d\theta_{23}\wedge \theta_1
			+d\theta_1\wedge \theta_{23} )
	+\frac{i}{4\pi^2}
		(d\theta_{23}\wedge\theta_{23}
			-d\theta_1\wedge \theta_1),
\end{align*}
where $\eta=\theta_1-i\theta_{23}$.
\end{proposition}

For an oriented
smooth hyperbolic manifold $M=\Gamma\backslash H^3$
of dimension $3$, let $\tilde{M}$ be the universal cover of
$M$ and $d:\tilde{M}\to H^3$ be a developing map.
Taking the differential of $d$, we obtain the $SO(3)$-bundle
map $\tilde{d}:F(\tilde{M})\to PSL_2(\mathbb{C})$.
Since the form $C$ is left invariant, $\tilde{d}^*C$ projects to a
closed form on $F(M)=\Gamma\backslash F(\tilde{M})$ which
by abuse of notation we
denote also by $C$.
Now, for the rest of this section,
suppose that $M$ is a marked
Schottky hyperbolic 3-manifold. For an admissible singular framing
$(\F,\kappa,\L)$ over $M$, we introduce a map
\begin{equation}\label{e:def-s}
s:(M\setminus\L)\cup\L\to F(M)
\end{equation}
defined by the admissible singular framing $\F$ over
$M\setminus\L$ and the reference framing $\kappa$ on
$\L$. For $0<\epsilon<a_1$, we now define
\begin{align}\label{e:def-f-epsilon}
\mathbb{CS}^\epsilon(M,s)
=\int_{s(M^\epsilon\setminus \L)} C
	-\sum_j \frac {n(j)}{2\pi}
		\int_{s(\ell_j^\epsilon)} (\theta_1-i\theta_{23})
\end{align}
where $M^\epsilon:=M\setminus N_{(0,\epsilon)}$,
$\ell_j$ denotes a connected component of $\L$, and
$\ell_j^{\epsilon}:=\ell_j\cap M^\epsilon$ . Here the sum
is over the connected components $\ell_j$ of $\L$
and $n(j)$ is the index of the admissible singularity of
$\F$ at $\ell_j$. The complex-valued invariant we define
will be a suitably regularized value of
$\mathbb{CS}^\epsilon(M,s)$ as $\epsilon\to 0$.

For a standard admissible framing $(\F,\kappa,\L)$ over $M$,
the singular curve $\L$ consists of two parts: $\L^1$ is a union of
simple closed curves and $\L^2$ is a union of curves connecting
two end points in $X=\partial\overline{M}$. Then the quantity
defined in \eqref{e:def-f-epsilon} is given by
\begin{equation}\label{e:def-f-epsilon-new}
\mathbb{CS}^\epsilon(M,s)
=\int_{s(M^\epsilon\setminus \L)} C
	- \frac {1}{2\pi} \int_{s(\L^1)} (\theta_1-i\theta_{23})
	+  \frac {1}{2\pi} \int_{s(\L^{2,\epsilon})} (\theta_1-i\theta_{23})
\end{equation}
where $\L^{2,\epsilon}:=\L^2\cap M^\epsilon$.

\subsection{Boundaries of
	$\overline{s(M^\epsilon\setminus\L)}$}
	\label{ss:boundary}

For a standard admissible framing $(\F,\kappa,\L)$ over $M$,
we investigate the structure of the boundaries of
$\overline{s(M^\epsilon\setminus\L)}$
where the closure is taken in $F(M)$. The boundary
$\partial(\overline{s(M^\epsilon\setminus\L)})$
consists of three parts which we are going to describe below.

One part of the boundary
$\partial(\overline{s(M^\epsilon\setminus\L)})$
is given by the closure of $s(X^\epsilon\setminus\L^2)$ in $F(M)$,
which we denote by $B^{0,\epsilon}$. Note that the boundary of
$B^{0,\epsilon}$ consists of a disjoint union of circles.

The second part of the boundary
$\partial(\overline{s(M^\epsilon\setminus\L)})$ is given by
$\bigcup_{y\in\L^1} \lim_{\delta\to 0}s(S_\delta(y))$, where
$S_\delta(y)$ denotes the circle consisting of points in the
orthogonal disc to $\L^1$ of distance $\delta$ from $y\in\L^1$.
For $y\in\L^1$, the limit of $ s(S_\delta(y))$  as $\delta\to 0$
exists since the framing $\F$ has a special singularity of
index $1$ at $\L^1$. We denote this part of boundary,
which does not depend on $\epsilon$, by $B^1$.
Actually $B^1$ is given by the standard cylinder
over $\L^1$: there is a map
\[
\psi : \L^1 \times S^1 \to p^{-1}(\L^{1}) \subset F(M)
\]
which takes the point $(y,v)\in \L^1 \times S^1$ to the framing
\begin{equation}\label{e:boun-para}
\psi (y,v)
	:= (e_1(y), \cos(v) e_2(y) +\sin(v) e_3(y), -\sin(v) e_2(y) + \cos(v) e_3(y))
\end{equation}
at the point $y\in\L^1$. Here  $(e_1,e_2,e_3)$ is the
reference framing $\kappa^1$ on $\L^1$.  The boundary
orientation of $B^1$ is induced from $\F$ and is given by
$(\psi_*\frac{\partial}{\partial y},\psi_*\frac{\partial}{\partial v} )$
so that $\psi$ is orientation-preserving.

The remaining part of boundary
$\partial(\overline{s(M^\epsilon\setminus\L)})$
is given by
$\bigcup_{y\in\L^{2,\epsilon}} \lim_{\delta\to 0}s(S_\delta(y))$.
For $y\in\L^2$, the limit of $ s(S_\delta(y))$  as $\delta\to 0$
exists since the framing $\F$ has an admissible singularity of
index\, $-1$ at $\L^2$. We denote this part by $B^{2,\epsilon}$.
Note that $B^{2,\epsilon}$ has circle boundaries which are the
boundaries of $B^{0,\epsilon}$ with the opposite orientation.
As the case of $B^1$, $B^2=\lim_{\epsilon\to 0} B^{2,\epsilon}$
can be described in terms of the standard cylinder over $\L^2$
with some modification. There is a map
\[
\psi : \L^2 \times S^1 \to p^{-1}(\L^{2}) \subset F(M)
\]
which takes the point $(y,v)\in \L^2 \times S^1$ to the framing
given by
\begin{equation}\label{e:boun-para-1}
\psi(y,v)=(e_1(y), \cos(v) e_2(y)+\sin(v) e_3(y), -\sin(v) e_2(y)+\cos(v) e_3(y)),
\end{equation}
where $(e_1,e_2,e_3)$ is the reference framing $\kappa^2$
on $\L^2$. We denote by $\Bg$ the image of $\psi$.
We take the orientation of $\Bg$ to be given by
$(\psi_*\frac{\partial}{\partial y},-\psi_*\frac{\partial}{\partial v})$,
so that $\psi$ is orientation-reversing by definition. The $\Bg$
and $B^{2}$ do not coincide completely, but we can describe their
difference explicitly:

\begin{lemma}\label{l:matrix-A} The  fiberwise right multiplication
of $A$ appearing in equation \eqref{e:def-A-0} induces an
orientation preserving diffeomorphism
$\mathcal{A}$ of
$p^{-1}(\overline{\mathcal{N}^\epsilon}(\L^2))\subset F(M)$
mapping $\tilde{B}^2$ to $B^2$ over $\L^2$.
\end{lemma}

\begin{proof}
The claim follows directly from the definition of admissible singularity.
\end{proof}

\subsection{Real part of $\mathbb{CS}^\epsilon(M,s)$}
	\label{ss:real}

We start with

\begin{lemma}\label{l:second-fun}
For $s$ corresponding to an admissible singular framing
$(\F,\kappa,\L)$, the following equalities hold over
$N_{(0,a_1)}\setminus \L^2$,
\begin{equation*}
\omega_{12} = II(e_2, e_2)\omega_2 +II(e_3,e_2)\omega_3,
\qquad
\omega_{13}=II(e_2,e_3)\omega_2+II(e_3,e_3)\omega_3,
\end{equation*}
where $\omega_i=s^*\theta_i$, $\omega_{ij}=s^*\theta_{ij}$
denote the fundamental forms and connection forms
pulled back by $s$ respectively, and $II(*,*)$ denotes the
second fundamental form.
\end{lemma}

\begin{proof}
By definition of $\F=(e_1,e_2,e_3)$,  $e_1$ is tangent to a
geodesic which is also trajectory of the gradient flow of the
defining function $r$ and $e_2,e_3$ are tangent to the level
surface $X^\epsilon$ with $r=\epsilon$. We use the equality
$\omega_{ij}(e_k)=- g_M( \nabla_{e_k} e_i, e_j)$
to obtain
$\omega_{1j}(e_1)=0$ and
$\omega_{1j}(e_k) = - g_M(\nabla_{e_k} e_1, e_j)= II (e_k, e_j)$
for $j=2,3$, $k=2,3$. This completes the proof.
\end{proof}

The mean curvature $H$ is defined to be
the trace of $II$. (Note that $H$ is defined to the half of the
trace of $II$ in some of the literature.) In \cite{KS}, $W$-volume
of $M^\epsilon$ is defined by
\[
W(M^\epsilon)
	:=\mathrm{Vol}(M^\epsilon)
		-\frac14\int_{X^\epsilon} H \mathrm{dvol}
\]
where $\mathrm{Vol}(M^\epsilon)$ denotes the volume
of $M^\epsilon$ and $\mathrm{dvol}$ denotes the area form
over $X^\epsilon$ induced by $g_M$. One nice property of
$W$-volume proved in  Lemma 4.5 in \cite{KS} is the following
equality: for $0<\epsilon<a$,
\begin{equation}\label{e:W-volume}
W(M^\epsilon)
	= 2\pi (1-g)\log\epsilon+W_{\mathrm{f.p.}}(M^\epsilon),
\end{equation}
where
$W(M):=\lim_{\epsilon\to 0}W_{\mathrm{f.p.}}(M^\epsilon)$
exists and defines the renormalized volume $W(M)$ of $M$ as
in Section 8 of \cite{KS}.

\begin{proposition} \label{p:real-f-epsilon}
For $s$ defined by a standard admissible singular framing
$(\F,\kappa,\L)$,
\[
\mathrm{Re}\,\mathbb{CS}^\epsilon(M,s)
	=\frac{1}{\pi^2} W(M^\epsilon)
		\qquad \text{for} \quad 0 <\epsilon < a_1.
\]

\end{proposition}

\begin{proof}
By the definition, we have
\begin{equation}\label{e:real-C1}
\begin{split}
\int_{\overline{s(M^\epsilon\setminus\L)}} \mathrm{Re}\, C
=&\frac{1}{4\pi^2}
	\int_{\overline{s(M^\epsilon\setminus\L)}}
		\big(4\theta_{1}\wedge \theta_2\wedge \theta_3
		- d(\theta_1\wedge\theta_{23}+\theta_2\wedge\theta_{31}
				+\theta_3\wedge\theta_{12})  \big)\\
=&\frac{1}{\pi^2} \mathrm{Vol}(M_{\epsilon})
	- \frac{1}{4\pi^2}\int_{\partial(\overline{s(M^\epsilon\setminus\L)})}
			\theta_1\wedge\theta_{23}
			+\theta_2\wedge\theta_{31}
			+\theta_3\wedge\theta_{12}.
\end{split}
\end{equation}
For the second equality in \eqref{e:real-C1}, we apply Stokes'
theorem. Now we consider the integrals over the boundary
$\partial(\overline{s(M^\epsilon\setminus\L)})
	=B^{0,\epsilon}\cup B^1\cup B^{2,\epsilon}$.
For the boundary integral over $B^{0,\epsilon}$, we have
\begin{equation*}
\begin{split}
-\frac{1}{4\pi^2}\int_{B^{0,\epsilon}}
	\theta_1\wedge\theta_{23}
	+\theta_2\wedge\theta_{31}
	+\theta_3\wedge\theta_{12}
=&\ \ \frac{1}{4\pi^2} \int_{X^\epsilon}
	\omega_1\wedge\omega_{23}
	+\omega_2\wedge\omega_{31}
	+\omega_3\wedge\omega_{12}\\
 = -\frac{1}{4\pi^2} \int_{X^\epsilon}
 	\mathrm{tr} II\ \omega_2\wedge\omega_3
= &-\frac{1}{4\pi^2} \int_{X^\epsilon} H \mathrm{dvol},
\end{split}
\end{equation*}
where  $X^\epsilon$ is oriented by $\omega_2\wedge\omega_3$
and the second equality follows from Lemma \ref{l:second-fun}.

For the boundary integral over $B^1$, recall that the
boundary $B^1$ is diffeomorphic to $\L^1\times S^1$ by
$\psi$ in \eqref{e:boun-para}, and that
$\psi_* \frac{\partial}{\partial v}$ is a vertical vector field
and $\psi^*\theta_{1j}(\frac{\partial}{\partial v})=0$ for
$j=2,3$ by definition of $B^1$, hence
$\psi^*(\theta_2\wedge\theta_{31})(\frac{\partial}{\partial v},*)=0$,
$\psi^*(\theta_3\wedge\theta_{12})(\frac{\partial}{\partial v},*)=0$.
Moreover, by definition,
$\psi^*\theta_{23}(\frac{\partial}{\partial v})=-1$.
This implies
\begin{align*}
-\frac{1}{4\pi^2}\int_{B^1}
	\theta_1\wedge\theta_{23}
	+\theta_2\wedge\theta_{31}
	+\theta_3\wedge\theta_{12}
=& -\frac{1}{4\pi^2}\int_{\L^1\times S^1}
	\psi^*(\theta_1\wedge\theta_{23})\\
=& \ \ \frac{1}{2\pi}\int_{\L^1} \psi^*\theta_1
=\frac{1}{2\pi}\int_{\L^1}s^*\theta_1
= \frac{1}{2\pi} \int_{s(\L^1)} \theta_1.
\end{align*}
Hence the boundary integral over $B^1$ cancels the real part of the
second integral in \eqref{e:def-f-epsilon-new}.

For the boundary integral $B^{2,\epsilon}$,
\begin{align*}
&-\frac{1}{4\pi^2}\int_{B^{2,\epsilon}}
	\theta_1\wedge\theta_{23}
	+\theta_2\wedge\theta_{31}
	+\theta_3\wedge\theta_{12} \\
=& -\frac{1}{4\pi^2}\int_{\L^{2,\epsilon}\times S^1}
	\psi^*\mathcal{A}^*( \theta_1\wedge\theta_{23}
						+\theta_2\wedge\theta_{31}
						+\theta_3\wedge\theta_{12}),
\end{align*}
where $\psi$ is given by \eqref{e:boun-para-1}.
Using
$\mathcal{A}^*\mathbb{\theta}=A^{-1}\cdot \theta$
and
$\mathcal{A}^* \Theta= A^{-1}\cdot dA+A^{-1}\cdot\Theta\cdot A$
with $\theta=(\theta_1,\theta_2,\theta_3)^t$, $\Theta=(\theta_{ij})$,
\begin{equation}\label{e:change-framing-real}
\begin{split}
&\mathcal{A}^*( \theta_1\wedge\theta_{23}
				+\theta_2\wedge\theta_{31}
				+\theta_3\wedge\theta_{12})\\
=&\ \theta_1\wedge\theta_{23}
	+\theta_2\wedge\theta_{31}
	+\theta_3\wedge\theta_{12}
	+ \sum_{j=1}^3 \theta_j \wedge
		(a_{j1} A_2\cdot dA_3
			+ a_{j2} A_3\cdot dA_1
			+ a_{j3} A_1\cdot d A_2)
\end{split}
\end{equation}
where $a_{jk}$ denotes the entry in $A$ and $A_j$ denotes
the column vector of $A$, and $A_j\cdot dA_k$ denotes the
inner product of two vectors. By Lemma \ref{l:matrix-A} and
\eqref{e:change-framing-real}, and repeating the
computation of the integral over $B^1$,
\begin{align*}
&-\frac{1}{4\pi^2}\int_{\L^{2,\epsilon}\times S^1}
	\psi^*\mathcal{A}^*( \theta_1\wedge\theta_{23}
						+\theta_2\wedge\theta_{31}
						+\theta_3\wedge\theta_{12})\\
 =& - \frac{1}{2\pi} \int_{s(\L^{2,\epsilon})} \theta_1
 	 -\frac{1}{4\pi^2}\int_{\L^{2,\epsilon}\times S^1}
	 	\psi^*\big( \sum_{j=1}^3
			\theta_j \wedge
			(a_{j1} A_2\cdot dA_3
				+ a_{j2} A_3\cdot dA_1
				+ a_{j3} A_1\cdot d A_2)\big)\\
 =& - \frac{1}{2\pi} \int_{s(\L^{2,\epsilon})} \theta_1.
\end{align*}
Here we use that $\psi:\L^2\times S^1\to \Bg$ in
\eqref{e:boun-para-1} is orientation reversing, and that
the form involving $A$ vanishes on the vertical vector field
$\psi_* \frac{\partial}{\partial v}$.	Hence the boundary
integral over $B^{2,\epsilon}$ cancels the real part of the
third integral in \eqref{e:def-f-epsilon-new}. This completes the proof.
\end{proof}

\subsection{Imaginary part of $\mathbb{CS}^\epsilon(M,s)$}

Now we prove
\begin{proposition}\label{p:im-ind}
For $s$ corresponding to an admissible singular framing
$(\F,\kappa,\L)$, the imaginary part of
$\mathbb{CS}^\epsilon(M,s)$ converges to a finite value
as $\epsilon\to 0$.
\end{proposition}

\begin{proof}
Over $N_{(0,a_1)}\setminus \L^2$, the pull back of the
imaginary part of $C$ by $s$ is given by
\begin{align}\label{e:im-C}
\frac{1}{4\pi^2}\big(
		\omega_{12}\wedge \omega_{13}\wedge \omega_{23}
		-\omega_{12}\wedge\omega_1\wedge\omega_2
		-\omega_{13}\wedge\omega_1\wedge\omega_3
		- \omega_{23}\wedge\omega_2\wedge\omega_3
		\big).
\end{align}
The first and the last terms in \eqref{e:im-C} vanish respectively
since they are sum of triple wedge products of $\omega_2,\omega_3$
by Lemma \ref{l:second-fun}. The second and the third terms in
\eqref{e:im-C} cancel each other by Lemma \ref{l:second-fun}
and the fact $II(e_2,e_3)=II(e_3,e_2)$. Hence the imaginary part
of the first integral in \eqref{e:def-f-epsilon} is finite and independent
of $0<\epsilon <a_1$. For the imaginary part of the line integral
over $\L$, note that for $\ell_j\in \L^2$, the integral
$\int_{\ell_j\cap N_{[\epsilon,a_1]}}\omega_{23}$
measures the
total rotation of $\kappa$ with respect to parallel
translation on $\ell_j\cap N_{[\epsilon,a_1]}$.
Since $r^{-1}\kappa$
extends smoothly to $\overline{M}$ by definition,
the limit of the line integral as $\epsilon\to 0$ has a finite value.
This completes the proof.
\end{proof}

\begin{proposition}\label{p:im-ind2}
For a given marked Schottky hyperbolic 3-manifold $M$,
if $s_0$, $s_1$ are defined by standard admissible framings
$(\F_0,\kappa_0,\L_0)$
and $(\F_1,\kappa_1,\L_1)$ on $M$ which are
related by a homotopy of standard admissible framings
which are fixed outside of $M^{a_1}$, then
\[
\mathrm{Im}\,\mathbb{CS}^\epsilon(M,s_0)
=\mathrm{Im}\,\mathbb{CS}^\epsilon(M,s_1).
\]
\end{proposition}
\begin{proof}
Let $(\F_u,\kappa_u,\L_u)$,
with $u\in [0,1]$ be the homotopy connecting
$(\F_0,\kappa_0,\L_0)$
and $(\F_1,\kappa_1,\L_1)$.
The framing $\F_u$ defines a section $s:W_\epsilon\to F(M)$ over
$W_\epsilon
	:=[0,1]\times M^\epsilon \setminus
		\{(u, y_u) \ | \ y_u\in\L_u, u\in[0,1]\}$.
Denoting by $Q$ the imaginary part of $C$, we have
\begin{align}\label{e:im-f1}
0
= \int_{\overline{s(W_\epsilon)}} dQ
= \int_{\overline{s_1(M^\epsilon\setminus\L_1)}} Q
	- \int_{\overline{s_0(M^\epsilon\setminus \L_0)}} Q
	+ \int_{B_W} Q.
\end{align}
The boundary $B_W$ consists of three parts
$\hat{B}^0$, $\hat{B}^1$, and $\hat{B}^2$, consisting of the
trajectories under the homotopy $\F_u$ of
$B^{0,\epsilon}$, $B^1$, and $B^{2,\epsilon}$ respectively.
For the integral over the part $\hat{B}^0$,
$\theta_i(s_*\frac{\partial}{\partial u})=0$ and
$\theta_{ij}(s_*\frac{\partial}{\partial u})=0$
over $\hat{B}^0=B^{0,\epsilon}\subset F(M)$. Therefore
\begin{equation}\label{e:im-f2}
\int_{\hat{B}^0} Q =  0.
\end{equation}
The boundary $\hat{B}^1$ is diffeomorphic to
$[0,1]\times \L^1\times S^1 $ by
\begin{equation*}
\psi(u,y,v)
=\{u\}\times
	(e_1(y), \cos(v) e_2(y)+\sin(v) e_3(y), - \sin(v) e_2(y) +\cos(v) e_3(y))
\end{equation*}
where $(e_1(y),e_2(y),e_3(y))$ denotes the reference
framing $\kappa_u(y)$ for $y\in \L^1$. Here and
below, we identify $\L^1_u$ with $\L^1=\L^1_0$ implicitly.
The orientation
$(\frac{\partial}{\partial u},\frac{\partial}{\partial y},-\frac{\partial}{\partial v})$
on $[0,1]\times \L^1\times S^1$ makes $\psi$ orientation
preserving. As before,
$\psi^*\Omega_{ij}(\frac{\partial}{\partial v}, *)=0$
for $1\leq i,j \leq 3$,
$(\psi^*\theta_{12})(\frac{\partial}{\partial v})=0$,
$(\psi^*\theta_{13})(\frac{\partial}{\partial v})=0$, and
$(\psi^*\theta_{23})(\frac{\partial}{\partial v})=-1$.
From above facts, we have
\begin{align*}
\psi^*Q
=\frac{1}{4\pi^2} \psi^*(\theta_{12}\wedge\theta_{13}\wedge\theta_{23}
						+\theta_{23}\wedge\Omega_{23})
=\frac{1}{4\pi^2}\psi^*(\theta_{23}\wedge d\theta_{23}),
\end{align*}
and
\begin{equation*}
\psi^*\theta_{23}=-dv+ q^*s^*\theta_{23},
\end{equation*}
where  $q:[0,1]\times \L^1\times S^1 \to [0,1]\times \L^1$
is the natural projection, $s_u:\L^1\to F(M)$ is the section
defined by $\kappa_u$, and $s:[0,1]\times \L^1\to F(M)$ is
the corresponding family given by $s(u,\cdot)=s_u$.
It follows that
$\psi^* Q= - \frac{1}{4\pi^2} dv\wedge d(q^*s^*\theta_{23})$.
With the above orientation convention, by Stokes' theorem,
we have
\begin{equation}\label{e:im-f3-b}
\begin{split}
\int_{\hat{B}^1} Q
=& \int_{[0,1]\times \L^1\times S^1} \psi^*Q
= - \frac{1}{4\pi^2}\int_{[0,1]\times \L^1\times S^1}
			dv\wedge d(q^*s^*\theta_{23})\\
=&\frac{1}{2\pi} \int_{[0,1]\times \L^1} d (s^*\theta_{23})
=\frac{1}{2\pi} \big(\int_{\L^1} s_1^*\theta_{23}
		- \int_{\L^1} s_0^*\theta_{23} \big).
\end{split}
\end{equation}
The right hand side of \eqref{e:im-f3-b} is the same as the
difference of the imaginary parts of the second integrals for
$u=1$ and $u=0$ in the definition of
$\mathbb{CS}^\epsilon(M,s)$ in \eqref{e:def-f-epsilon-new}.

For the boundary integral over $\hat{B}^{2,\epsilon}$,
as in the proof of Proposition \ref{p:real-f-epsilon} we have
\begin{align*}
\int_{\hat{B}^{2,\epsilon}} Q
= \int_{ [0,1]\times \L^{2,\epsilon}\times S^1}
		\psi^*\mathcal{A}^*Q
\end{align*}
where $\psi$ is the orientation reversing diffeomorphism
defined in \eqref{e:boun-para-1}. We also have
\begin{equation}\label{e:change-framing-im}
\begin{split}
\mathcal{A}^* Q
=\  Q &+\frac1{24\pi^2} \mathrm{Tr}( (A^{-1} dA)^3)\\
&+\frac1{4\pi^2} d\,(\theta_{12}\wedge d\hat{A}_1 \cdot \hat{A}_2
					+\theta_{13}\wedge d\hat{A}_{1}\cdot \hat{A}_3
					+\theta_{23}\wedge d\hat{A}_2\cdot\hat{A}_3)
\end{split}
\end{equation}
where $\hat{A}_j$ denotes the row vector of $A$.
Hence, in a similar way as \eqref{e:im-f3-b},
\begin{equation} \label{e:im-f3-b2}
\begin{split}
& \int_{\hat{B}^{2,\epsilon}} Q  \\
=& -\frac{1}{2\pi} \big(\int_{\L^{2,\epsilon}} s_1^*\theta_{23}
				- \int_{\L^{2,\epsilon}} s_0^*\theta_{23} \big)
	-\frac{1}{2\pi}\big(\int_{\L^{2,\epsilon}}
						\psi_1^*d\hat{A}_2\cdot\hat{A}_3
				-\int_{ \L^{2,\epsilon}}
					\psi_0^*d\hat{A}_2\cdot\hat{A}_3
				\big)  \\
=&-\frac{1}{2\pi} \big(\int_{\L^{2,\epsilon}} s_1^*\theta_{23}
					- \int_{\L^{2,\epsilon}} s_0^*\theta_{23} \big)
\end{split}
\end{equation}
where $\psi_1$ and $\psi_0$ represent $\psi$
taken at $u=1$ and $u=0$.
Here the first equality follows from that $\psi^*\theta_{12}$,
$\psi^*\theta_{13}$, and the form involving $A$ vanish on the
vertical vector field $\psi_* \frac{\partial}{\partial v}$.
The expression
$(1/2\pi)\int_{\L^{2,\epsilon}}\psi^*_u d\hat{A}_2\cdot\hat{A}_3$
can be shown to be the total rotation angle of $\hat{A}_2$
about the axis $\hat{A}_1$ along $\L^{2,\epsilon}$. Since
$A$ is fixed at the endpoints of $\L^{2,\epsilon}$ through the
homotopy, this total rotation angle does not change, so the
second equality follows. The right hand side of \eqref{e:im-f3-b2}
is the same as the difference of the imaginary parts of the
third integrals for $u=1$ and $u=0$ in the definition of
$\mathbb{CS}^\epsilon(M,s)$ in \eqref{e:def-f-epsilon-new}.
Combining \eqref{e:im-f1}, \eqref{e:im-f2}, \eqref{e:im-f3-b}
and \eqref{e:im-f3-b2} completes the proof.
\end{proof}

\subsection{Definition of the invariant $\mathbb{CS}(M,s)$
	and the function $\mathbb{CS}$}\label{ss:def-CS}

For $s:M\to F(M)$ corresponding to an admissible singular
framing $(\F,\kappa,\L)$ as explained after equation
\eqref{e:def-s}, we define the Chern-Simons invariant of
$(M,s)$ to be
\[
CS(M,s)
:=\frac{1}{2}\lim_{\epsilon\to 0}
	\mathrm{Im}\, \mathbb{CS}^\epsilon(M,s),
\]
where  the limit exists by Proposition \ref{p:im-ind}.
By Proposition \ref{p:im-ind2}, $CS(M,s)$ is independent of a
homotopic change of an admissible singular framing $(\F,\kappa,\L)$
inside of $M^{a_1}$. We can now define the invariant
$\mathbb{CS}(M,s)$.
\begin{definition}\label{d:CS}
For $s:M\to F(M)$ corresponding to an admissible singular
framing $(\F,\kappa,\L)$,
\[
\mathbb{CS}(M,s)
:= \lim_{\epsilon\to 0} \,
	( \mathbb{CS}^\epsilon(M,s)
		+\frac{2}{\pi}(g-1)\log\epsilon).
\]
\end{definition}

By \eqref{e:W-volume} and Proposition \ref{p:real-f-epsilon},
as we stated in \eqref{e:first-equality}, we have
\[
\mathbb{CS}(M,s)
	=\frac{1}{\pi^2}W(M) + 2iCS(M,s).
\]

Now, suppose we are given a compact marked Riemann
surface $X$ and a holomorphic 1-form $\Phi$ on $X$,
with corresponding admissible singular framing
$(\F_\Phi,\kappa_\Phi,Z)$ over $X$. Then we have
associated to this data a
unique marked Schottky hyperbolic 3-manifold $M_X$
and a standard  admissible extension $(\F,\kappa,\L)$ over $M_X$
corresponding to $s_\Phi:M_X\to F(M_X)$.
We now consider to what extent the invariant
$\mathbb{CS}(M_X,s_\Phi)$ depends on our choice of
admissible extension $s_\Phi$. We have already shown in Proposition
\ref{p:im-ind2} that it is independent of a homotopic change
of $(\F,\kappa,\L)$ in $M^{a_1}_X$. Now we show

\begin{proposition}\label{p:ind-signs}
The quantity $\exp(4\pi\mathbb{CS}(M_X,s_\Phi))$
is independent of the choice of signs in $\kappa_\Phi$ and
of the choice of $\kappa$.
\end{proposition}

\begin{proof}
Note that the modulus of $\exp(4\pi\mathbb{CS}(M_X,s_\Phi))$
depends only on $M_X$ by Proposition \ref{p:real-f-epsilon}.
For the argument of $\exp(4\pi\mathbb{CS}(M_X,s_\Phi))$,
there is a choice of a reference framing $\kappa$ which can
rotate along $\L$, but a change of a rotation number results
in only an integer difference  in the imaginary part of
$\mathbb{CS}(M_X,s_\Phi)$ through the second and third integrals
in \eqref{e:def-f-epsilon-new}. There are sign ambiguities in the
definition of the reference framing at zeroes of $\Phi$ mentioned
just after \eqref{e:frame-by-theta-til}. Hence the imaginary part of
$\mathbb{CS}(M_X,s_\Phi)$ is well-defined only up to addition of
a half-integer, but this ambiguity will disappear for
$\exp(4\pi \mathbb{CS}(M_X,s_\Phi))$.
\end{proof}

To state the main result of this section, we need to introduce
an auxiliary space. For each point $(X,\Phi)$ in
$\tilde{\mathcal{H}}_g(1,\ldots,1)$, we attach the data
of a choice of isotopy class of $g-1$ simple, pairwise disjoint
curves in $M_X$ whose endpoints are the zeroes of $\Phi$.
The resulting space
$\tilde{\mathcal{H}}^*_g(1,\ldots,1)$
is locally isomorphic to
$\tilde{\mathcal{H}}_g(1,\ldots,1)$,
and there is a natural projection map to
$\tilde{\mathcal{H}}_g(1,\ldots,1)$
corresponding to forgetting the added data. Note that each
connected component of
$\tilde{\mathcal{H}}^*_g(1,\ldots,1)$
covers
$\tilde{\mathcal{H}}_g(1,\ldots,1)$
by this projection map.

\begin{theorem}\label{t:local-ftn}
The expression $\exp(4\pi \mathbb{CS}(M_X,s_\Phi))$ determines
a globally well-defined function
$\exp(4\pi \mathbb{CS}):\tilde{\mathcal{H}}^*_g(1,\ldots,1)\to\mathbb{C}$.
\end{theorem}

\begin{proof}
Given a point in $\tilde{\mathcal{H}}^*_g(1,\ldots,1)$,
we use Theorem \ref{t:framing-ext} to construct a standard
singular admissible framing on $M_X$, whose $\L^2$ curves
are isotopic  to the given $g-1$ curves. It is clear from the
construction that any two such framings are related by a
homotopy, which is an isotopy of the corresponding set of
curves $\L$. It then follows from Propositions \ref{p:im-ind2}
and \ref{p:ind-signs} that the value of
$\exp(4\pi \mathbb{CS}(M_X,s_\Phi))$
is uniquely determined by this data.
\end{proof}

\begin{remark}\label{r:independence-s}
The proof of the main theorem in Section
\ref{s:proof-main-theorem} will show that, in fact, the function
$\exp(4\pi\mathbb{CS})$, restricted to any
connected component of
$\tilde{\mathcal{H}}^*_g(1,\ldots,1)$,
descends to a well-defined function on
$\tilde{\mathcal{H}}_g(1,\ldots,1)$.
Restricting to a different connected component of
$\tilde{\mathcal{H}}^*_g(1,\ldots,1)$
will give a function on
$\tilde{\mathcal{H}}_g(1,\ldots,1)$
differing from the original by a multiplicative constant.
\end{remark}

\section{Variation of the invariant $\mathbb{CS}$}
	\label{s:variation}

Suppose we are given a contractible open set $U$
in $\tilde{\mathcal{H}}^*_g(1,\ldots,1)$. By the results of
the previous section, the invariant $\mathbb{CS}(M_X,s_\Phi)$
determines a function $\mathbb{CS}:U\to\mathbb{C}$,
which is well-defined up to addition of $\frac{1}{2}ni$ for $n\in\mathbb{Z}$.
In this section we find expressions for the derivatives
$\partial\mathbb{CS}$ and $\bar{\partial}\mathbb{CS}$
of this function.

\subsection{Basic notations for variation}\label{ss:variation}

Each point $u\in U$ determines a compact marked
Riemann surface $X_u$ together with a holomorphic 1-form
$\Phi_u$ on $X_u$. We fix a basepoint $u_0\in U$, and
for simplicity we write $X=X_{u_0}$ and similarly below.
We will always assume that $X_u$ is uniformized
by a marked Schottky group,
$X_u=\Gamma_u\backslash\Omega_u$,
where $\Gamma_u$ is a marked Schottky group with
marked normalized generators $\{L_1(u),\ldots,L_g(u)\}$
and ordinary set $\Omega_u$. The group $\Gamma_u$
simultaneously defines a marked Schottky hyperbolic 3-manifold
$M_u:=M_{X_u}=\Gamma_u\backslash H^3$.
For each $u\in U$, we have a quasi-conformal mapping
$f_{u}:\Omega\to \Omega_{u}$.
Define $P^\epsilon_u:\Omega_u\to H^3$
to be the map translating points along the integral
curve $\phi_t$  of $\nabla_{\overline{g}_u}r_u$ emanating
from $z\in \Omega_u$ for $\epsilon$ units of time, where
$\bar{g}$ and $r$ are defined as in subsection \ref{ss:Schottky}.
Then we define a map
$\mathbf{f}_u:\cup_{0<\epsilon<a_1}P^\epsilon(\Omega)\to H^3$
by
\begin{equation}\label{e:bf-f}
\mathbf{f}_u\vert_{P^\epsilon(\Omega)}
	=P^\epsilon_u\circ \mathbf{f}_u \circ (P^\epsilon)^{-1}.
\end{equation}
(Here $a_1=\frac{a}{4}$, where $a$ is defined
as in subsection \ref{ss:Schottky}.)
This map extends
to a diffeomorphism $\mathbf{f}_{u}:H^3\to H^3$,
satisfying $\mathbf{f}_u\circ\gamma=\gamma_u\circ\mathbf{f}_u$
for all $\gamma_u\in\Gamma_u$.

Corresponding to the family $\Phi_u$ and the given homotopy class
of $g-1$ curves in $M_u$ connecting the zeroes of $\Phi_u$,
we take a smooth family
of sections $s_u:=s_{\Phi_u}:(M_u\setminus \L_u)\cup\L_u \to F(M_u)$,
constructed as in Theorem \ref{t:local-ftn}. Here $\L^2$ is taken
to be isotopic to the given $g-1$ curves, and $\L_u=\mathbf{f}_u(\L)$.
We also denote
by $\L_u$ and $s_u$ the corresponding liftings
$\L_u\subset H^3$ and
$s_u:(H^3\setminus \L_u)\cup \L_u\to
	F(H^3)\cong PSL_2(\mathbb{C})$.
The family defines a map
$s:U\times H^3
	=\{(u,x)\ \vert \ u\in U, x\in (H^3\setminus \L_u)\cup \L_u\}
	\to {PSL}_2(\mathbb{C})$.
We let $K$ be the unique map
$K:U\times s_0((H^3\setminus \L)\cup\L)\to{PSL}_2(\mathbb{C})$
satisfying
\[
K\circ(\mathrm{id},s_0)=s\circ(\mathrm{id},\mathbf{f}),
\]
where
$s\circ(\mathrm{id},\mathbf{f})(u,x)
	:=s_u\circ \mathbf{f}_u(x)$
for	$(u,x)\in U\times H^3$.
As observed in subsection \ref{ss:boundary}, the closure
$\overline{s_0(H^3\setminus \L)}$ of	$s_0(H^3\setminus \L)$
in $PSL_2(\mathbb{C})$ provides a natural compactification
of $s_0(H^3\setminus\L)$, and $K$ extends smoothly to
$U\times(\overline{s_0(H^3\setminus\L)}\cup s_0(\L))$
(we also denote the extension by $K$).
Note that the generators $L_{r}(u)$ of $\Gamma_u$,
$r=1,\ldots,g$, can be considered as giving holomorphic
functions
\[
L_r:U\to{PSL}_2(\mathbb{C}).
\]

We let $\mathbf{D}$ be a fundamental domain for the
action of $\Gamma$ on $H^3$, such that
$\partial\mathbf{D}\subset H^3$ consists of $2g$
smooth surfaces $D_r$, $-L_r(0)D_r$, for $r=1,\ldots,g$
(the negative sign indicates opposite orientation).
Define
$\mathbf{D}_u:=\mathbf{f}_u(\mathbf{D})$.
Considering $H^3$ as $\{(t,x,y)\in\mathbb{R}^3\vert t>0\}$,
define $D$ and $C_r$ to be the intersection of the closure
of $\mathbf{D}$ and $D_r$ respectively
with the set $t=0$. Then $D$ is
a fundamental domain of the action of $\Gamma$ on
$\Omega$, $\partial D$ consists of smooth curves
$C_r$, $-L_r(0)C_r$, and we define $D_u:=f_u(D)$.
We denote
$\mathbf{D}':=\mathbf{D}\setminus\L$,
and define
$\Delta:=s_0(\mathbf{D}')$.
As above, the closure
$\overline{\Delta}$ of $\Delta$ in $PSL_2(\mathbb{C})$
provides a natural compactification of $\Delta$. Let
$T_r:=\overline{s_0(D_r)}$ for $r=1,\ldots,g$. The
boundary components of $\overline{\Delta}$ consist of
$B^0\cup B^1\cup B^2$ as defined in the subsection
\ref{ss:boundary}, and $\cup_{r=1}^g (T_r-L_r(0)T_r)$.
We denote by $\mathbf{D}^\epsilon$ and
$\overline{\Delta}^\epsilon$ the subsets of $\mathbf{D}$
and $\overline{\Delta}$ respectively corresponding to
$M^\epsilon$. Define
$D^\epsilon:=\mathbf{D}\cap P^\epsilon(D)$.
The boundary components
of $\overline{\Delta}^\epsilon$ consist of
$B^{0,\epsilon}\cup B^1\cup B^{2,\epsilon}$
and
$\cup_{r=1}^g (T_r^{\epsilon}-L_r(0) T_r^{\epsilon})$
where
$B^{0,\epsilon}$ is diffeomorphic to $B^0$, and $B^{2,\epsilon}$
and
$\cup_{r=1}^g (T_r^{\epsilon}-L_r(0) T_r^{\epsilon})$
are subsets of  $B^2$ and $\cup_{r=1}^g (T_r-L_r(0) T_r)$
respectively. The notations $\mathbf{D}'_u$, $\Delta_u$,
etc. denote the corresponding constructions for $\mathbf{D}_u$.

Since we will always be working in a fixed fundamental
domain $\mathbf{D}_u$, from now on, we will write
$\L_u=\L^1_u\cup\L^2_u$ to
mean the intersection $\L_u\cap\mathbf{D}_u$.
The
boundary points of $\L^1_u \cup \L^2_u$
consist of finitely many matched pairs $y_j(u)$ and $L_{r(j)}(u)y_j(u)$,
$j\in \mathcal{J}$, together with $2g-2$ points which are the zeros
of the holomorphic $1$-form $\Phi_u$.
We may assume that every curve in
$\L^1_u$ has exactly two points $y_j(u)$, $L_{r(j)}(u)y_j(u)$
in its boundary, and
we assume the orientation of $\L^1_u$ given by the reference
framing $\kappa_u^1$ is such that the component connecting
$L_{r(j)}(u)y_j(u)$ to $y_j(u)$ is oriented towards $y_j(u)$.

Under the canonical map from $\tilde{\mathcal{H}}^*_g(1,\ldots,1)$
to $\mathfrak{S}_g$, a holomorphic tangent vector
$\varpi$ in $T^{1,0}U$ at $u_0$ maps to a holomorphic
tangent vector in $T^{1,0}\mathfrak{S}_g$, which corresponds
to a harmonic Beltrami differential
$\mu\in \mathcal{H}^{-1,1}(\Omega,\Gamma)$.
Then $\mu$ defines a quasi-conformal mapping
$f_{w\mu}:X\to X_w$
for all $w$ in some neighborhood $W$ of the origin in $\mathbb{C}$.
There exists a holomorphic family $\{\Phi(w)\}$, where
$\Phi(w)$ is a holomorphic 1-form on $X_w$, such that
the derivative at $u_0$ of the complex curve in $U$ given
by the family $\{(X_w,\Phi(w))\}$ is $\varpi$.
(Here we are using the local isomorphism of
$\tilde{\mathcal{H}}^*_g(1,\ldots,1)$
and $\tilde{\mathcal{H}}_g(1,\ldots,1)$.)
In this way
we obtain a complex curve $u:W\to U$, such that
$\frac{\partial u}{\partial w}=\varpi$ and
$\frac{\partial u}{\partial \bar{w}}=0$
(with $w$ a local coordinate in $W$).

For the curve $u:W\to U$ we define
$f:W\times\Omega\to\mathbb{C}$ by
$f(w,z)=f_{u(w)}(z)=f_{w\mu}(z)$
and
$\mathbf{f}
:=W\times H^3\to H^3$ by $\mathbf{f}(w,x)=\mathbf{f}_{u(w)}(x)$.
We also define
\[
H=K\circ(u,\mathrm{id}):
	W\times(\overline{s_0(H^3\setminus\L)}\cup s_0(\L))
	\to{PSL}_2(\mathbb{C}),
\]
and
\[
\sigma=s\circ(u,\mathbf{f}):
	W\times H^3\to {PSL}_2(\mathbb{C}).
\]

\subsection{Contributions of boundaries}

For technical reasons we consider the holomorphic variation of
$\overline{\mathbb{CS}}$
rather than
$\mathbb{CS}$.

To derive a variation formula for
$\overline{\mathbb{CS}}$,
we start with the following equality:
\begin{align}\label{e:var-1}
\begin{split}
0=
&\int'_{\overline{\Delta}^{\epsilon}} H^*d \Cbar
=\int'_{\overline{\Delta}^{\epsilon}} (d_W+d_\Delta)H^*\Cbar
= d_W\int'_{\overline{\Delta}^{\epsilon}} H^*\Cbar
	-\int'_{\partial\overline{\Delta}^{\epsilon}} H^*\Cbar\\
=& d_W\int'_{\overline{\Delta}^\epsilon} H^*\Cbar
	- \int'_{{B}^{0,\epsilon}\cup {B}^{1}\cup {B}^{2,\epsilon}}
			H^*\Cbar
	- \sum_{r=1}^g
		\int'_{{T}_r^{\epsilon}-L_r T_r^{\epsilon}} H^*\Cbar.
\end{split}
\end{align}
Here the notation $\int'_{\overline{\Delta}^{\epsilon}}$ denotes
the \emph{partial integral}: we consider the integrand as a
form on $\overline{\Delta}$ taking values in forms on $W$, and
integrate over $\{w\}\times\overline{\Delta}^{\epsilon}$,
obtaining a 1-form on $W$. The notation $d=d_W+d_\Delta$
denotes the splitting of $d$ on $W\times\Delta$ in the obvious way.
Note that we use the orientation from $W\times\Delta$; for
this reason, we have
$d_W\int'_{\overline{\Delta}^{\epsilon}} H^*\Cbar
=\int'_{\overline{\Delta}^{\epsilon}} d_W H^*\Cbar$,
but when we apply Stokes' theorem, we have
$\int'_{\overline{\Delta}^{\epsilon}} d_\Delta H^*\Cbar
=-\int'_{\partial\overline{\Delta}^{\epsilon}} H^*\Cbar$.
We use a similar convention for partial integrals throughout this
section.

The next three lemmas deal with the partial integrals
$\int'_{{B}^{1}} H^*\Cbar$,
$\int'_{{B}^{2,\epsilon}} H^*\Cbar$,
and
$\int'_{{B}^{0,\epsilon}} H^*\Cbar$
respectively.

\begin{lemma}\label{l:L1}
Let $u:W\to U$ be a complex curve as defined above, with $w\in W$.
Then we have the following equality of 1-forms over $W$:
\begin{equation*}
\int'_{B^1} H^*\Cbar
	- \frac{1}{2\pi}d_W\int'_{\L^1} \sigma^*(\theta_1+i\theta_{23})
	=
 - \frac{1}{2\pi}\sum_{y\in\partial\L^1}
				\sigma^*(\theta_1+i\theta_{23})\vert_y.
\end{equation*}
\end{lemma}

\begin{proof}
Recall that the integral over $ {B}^1$ is independent
of $\epsilon$ for small $\epsilon>0$.
As in the proof of Proposition \ref{p:im-ind}, we have the
diffeomorphism
\[
\psi: W\times \L^1\times  S^1 \, \longrightarrow \, W\times B^1
\]
defined by
\[
\psi(w,y,v)=\{w\}\times H(w,\cdot)^{-1}
		( e_1(y),
		\cos (v)e_2(y) +\sin (v)e_3(y),
		- \sin (v)e_2(y)+\cos (v) e_3(y))
\]
for $w\in W$,  $y\in\L^1$ and $v\in S^1$. The notation
$H(w,\cdot)^{-1} $ denotes the inverse of $H(w,\cdot)$
restricted to its image. The orientation of
$\L^1\times S^1$
is given by
$(\frac{\partial}{\partial y},\frac{\partial}{\partial v})$.
As in the proof of Propositions \ref{p:real-f-epsilon} and \ref{p:im-ind},
we have
$(\psi^*H^*\theta_i)(\frac{\partial}{\partial v})
	=(\psi^*H^*\theta_{1i})(\frac{\partial}{\partial v})=0$
($i=1,2,3$) and
$(\psi^*H^*\theta_{23})(\frac{\partial}{\partial v})=-1$.
It follows that
\[
\psi^*H^* \Cbar
	= -\frac{1}{4\pi^2} \psi^*H^*(d(\theta_1\wedge\theta_{23}))
		- \frac{i}{4\pi^2} \psi^*H^* (\theta_{23}\wedge d\theta_{23}).
\]
Let $q:W\times \L^1\times S^1\to W\times \L^1$
be the natural projection.  Then
\begin{equation*}
\psi^*H^*\theta_{23}
	=-dv+ q^*\sigma^*\theta_{23}.
\end{equation*}
It follows that
$\psi^*H^*(\theta_{23}\wedge d\theta_{23})
	=  -dv\wedge d(q^*\sigma^*\theta_{23})$.
From the above orientation convention, by Stokes' theorem, we have
\begin{align}\label{e:real-f}
\begin{split}
&\frac{1}{4\pi^2}\int'_{\L^1\times S^1}
				\psi^*H^*(d(\theta_1\wedge\theta_{23}))\\
=& \frac{1}{4\pi^2}d_W\int'_{\L^1\times S^1}
				\psi^*H^*(\theta_1\wedge\theta_{23})
	- \frac{1}{4\pi^2}\int'_{\partial\L^1\times S^1}
				\psi^*H^*(\theta_1\wedge\theta_{23})\\
=&-\frac{1}{2\pi} d_W\int'_{\L^1} \sigma^*\theta_{1}
	+ \frac{1}{2\pi} \sum_{y\in\partial\L^1}
					\sigma^*\theta_{1}\vert_y.
\end{split}
\end{align}
and
\begin{align}\label{e:im-f3}
\begin{split}
&\frac{1}{4\pi^2}\int'_{\L^1\times S^1}
				\psi^*H^*(\theta_{23}\wedge d\theta_{23})\\
=& -\frac{1}{4\pi^2}\int'_{ \L^1\times S^1}
				dv\wedge d(q^*\sigma^*\theta_{23})
=-\frac{1}{2\pi} \int'_{\L^1}
				d (\sigma^*\theta_{23})\\
=&-\frac{1}{2\pi} d_W\int'_{\L^1}
				\sigma^*\theta_{23}
	+ \frac{1}{2\pi} \sum_{y\in\partial\L^1}
				\sigma^*\theta_{23}\vert_y.
\end{split}
\end{align}
Combining \eqref{e:real-f} and \eqref{e:im-f3} proves the lemma.
\end{proof}

\begin{lemma}\label{l:L2}
We have the following equality of 1-forms over $W$:
\begin{align*}
\lim_{\epsilon\to 0} \big(\,
	 \int'_{B^{2,\epsilon}} H^*\Cbar
	 +\frac{1}{2\pi}d_W\int'_{\L^{2,\epsilon}}
	 			&\sigma^*(\theta_1+i\theta_{23})
		\, \big)
= \lim_{\epsilon\to 0}
	\frac{1}{2\pi} \sum_{y\in\partial\L^{2,\epsilon}}
					\sigma^*(\theta_1+i\theta_{23})\vert_y.
\end{align*}
\end{lemma}

\begin{proof}
We
define the map $\tilde{H}:W\times {\Bg} \to \Bg_u$ by
$\tilde{H}
	=\mathcal{A}^{-1}\circ H \circ (\mathrm{id},\mathcal{A})$.
As in the proof of Propositions \ref{p:real-f-epsilon} and
\ref{p:im-ind}, by \eqref{e:change-framing-real} and
\eqref{e:change-framing-im} and denoting
$\mathcal{A}^{-1}(B^{2,\epsilon})$ by $\tilde{B}^{2,\epsilon}$,
\begin{align*}
\int'_{B^{2,\epsilon}} & H^*\Cbar
=\  \int'_{\tilde{B}^{2,\epsilon}}
		(\mathrm{id},\mathcal{A})^*H^*\Cbar
=\ \int'_{\L^{2,\epsilon}\times S^1}
		\psi^*\tilde{H}^*\mathcal{A}^* \Cbar\\
=\ & - \frac{1}{4\pi^2}\int'_{\L^{2,\epsilon}\times S^1}
			\psi^*\tilde{H}^*(d(\theta_1\wedge\theta_{23}))
			+{i} \psi^*\tilde{H}^* (\theta_{23}\wedge d\theta_{23})\\
- &\frac{1}{4\pi^2}\int'_{\L^{2,\epsilon}\times S^1}
			\psi^*\tilde{H}^* d \big(
				\sum_{j=1}^3 \theta_j \wedge
					(a_{j1} A_2\cdot dA_3 + a_{j2} A_3\cdot dA_1
						+ a_{j3} A_1\cdot d A_2) \big)\\
- &\frac{1}{4\pi^2}\int'_{\L^{2,\epsilon}\times S^1}
			 i\psi^*\tilde{H}^*
			 	d(\theta_{12}\wedge d\hat{A}_1 \cdot \hat{A}_2
				+\theta_{13}\wedge d\hat{A}_{1}\cdot \hat{A}_3
				+\theta_{23}\wedge d\hat{A}_2\cdot\hat{A}_3)\\
- &\frac{1}{4\pi^2} \int'_{\L^{2,\epsilon}\times S^1}
			\frac{i}6\psi^*\tilde{H}^*  \mathrm{Tr}( (A^{-1} dA)^3).
\end{align*}
By
$\psi^*\tilde{H}^*(\theta_i)(\frac{\partial}{\partial v})=0$,
$\psi^*\tilde{H}^*(\theta_{ij})(\frac{\partial}{\partial v})=0$,
$\psi^*\tilde{H}^* (dA_j)(\frac{\partial}{\partial v})=0$, and
$\psi^*\tilde{H}^*(d\hat{A}_k)(\frac{\partial}{\partial v})=0$,
all the integrals vanish except the integral of
$\psi^*\tilde{H}^* d(\theta_{23}\wedge d\hat{A}_2\cdot\hat{A}_3)$
for the terms on the last three lines of the above equalities.
But we have the following equality:
\begin{align*}
&\ \ \frac{1}{4\pi^2}\int'_{\L^{2,\epsilon}\times S^1}
		\psi^*\tilde{H}^*
			d(\theta_{23}\wedge d\hat{A}_2\cdot\hat{A}_3) \\
=&  \frac{1}{2\pi} d_W\int'_{\L^{2,\epsilon}}
		\psi^*\tilde{H}^*
			\big(d\hat{A}_2\cdot\hat{A}_3\big)
	 -\frac{1}{2\pi}\sum_{y\in\partial \L^{2,\epsilon}}
		\psi^*\tilde{H}^* \big( d\hat{A}_2\cdot\hat{A}_3 \big)\vert_y.
\end{align*}
The first term in the second line can be shown to give
the variation with respect to $w$ of the sum of the total
rotation angles of $\hat{A}_2$ around $\hat{A}_1$
along the components of $\L^{2,\epsilon}$. But by our
assumptions on the framing, $A$ limits to the identity at
the boundary $\partial\L^2\cap D$. Hence the limit of
this term as $\epsilon\to 0$ gives an integer, which is
invariant under the deformation. The last term is $0$
since the contributions from boundary points in the
interior of $M$ cancel by an invariance under identification
by the $L_r(u(w))$, and at the remaining boundary points,
$\psi^*\tilde{H}^* (d\hat{A}_2\cdot\hat{A}_3)\vert_y \to 0$
as $\epsilon\to 0$.
From these equalities, we have
\begin{equation*}
\lim_{\epsilon\to 0}\int'_{B^{2,\epsilon}} H^*\Cbar
= \lim_{\epsilon\to 0}\ - \frac{1}{4\pi^2}\int'_{\L^{2,\epsilon}\times S^1}
			\psi^*\tilde{H}^*(d(\theta_1\wedge\theta_{23}))
				+{i} \psi^*\tilde{H}^* (\theta_{23}\wedge d\theta_{23}).
\end{equation*}
Now, repeating the derivation in \eqref{e:real-f} and \eqref{e:im-f3}
and recalling that $\psi$ is an orientation reversing diffeomorphism
in this case, completes the proof.
\end{proof}

Now we deal with the partial integral over $B^{0,\epsilon}$.
Recall that $D^\epsilon$ is the subset in
$\mathbf{D}=\mathbf{D}_0$
corresponding to $X^\epsilon$.  First we have

\begin{lemma}\label{p:B0epsilon}
We have the following equality of 1-forms on $W$:
\begin{equation*}
\begin{split}
\int'_{B^{0,\epsilon}}  H^*\Cbar
& = \int'_{D^\epsilon\setminus\L}\sigma^*\Cbar \\
 & =  \frac{1}{4\pi^2} \int'_{(D^\epsilon\setminus\L)}
 	\big(\, d_D \omega_{23}\wedge (\chi_1+i\chi_{23})
		+\omega_{23}\wedge (d_D(\chi_1+i\chi_{23}) +
		id_W\omega_{23}) \, \big)
\end{split}
\end{equation*}
where $d=d_W+d_D$ over $W\times D^\epsilon$. Here
$\chi_1$ and $\chi_{23}$ are defined by
$ \sigma^*\theta_{1}=\omega_1+\chi_1$
and
$ \sigma^*\theta_{23}=\omega_{23}+ \chi_{23}$,
where
$\omega_1\vert_{TW}=\omega_{23}\vert_{TW}=0$
and
$\chi_1\vert_{TD^\epsilon}=\chi_{23}\vert_{TD^\epsilon}=0$.
\end{lemma}
Note that $\omega_1$, $\omega_{23}$, $\chi_1$ and
$\chi_{23}$ depend on $w\in W$.

\begin{proof}
By Proposition \ref{p:new express}, it is easy to see
\begin{align*}
 \sigma^* \Cbar
 = & -\frac{1}{4\pi^2} \big(\,
 d(\omega_{23}+\chi_{23})\wedge( \omega_1+i\omega_{23}
 	+\chi_1+i\chi_{23})\, \big) \\
&\quad+\frac{i}{4\pi^2} \big(\,
	d(\omega_{1}+\chi_{1})\wedge( \omega_1
		+i\omega_{23} +\chi_1+i\chi_{23})\, \big).
\end{align*}
Now, note that $d_{D}\omega_1=0$ by Lemma \ref{l:second-fun},
and that $\omega_1$ vanishes on tangent vectors of
$(D^\epsilon\setminus\L)$.
Also note that
\[
0=d_W\int_{D^\epsilon\setminus\L}'\omega_{23}\wedge\omega_1
=\int_{D^\epsilon\setminus\L}'(d_W\omega_{23})\wedge\omega_1
-\int_{D^\epsilon\setminus\L}'\omega_{23}\wedge(d_W\omega_1),
\]
so
$\int_{D^\epsilon\setminus\L}'
	\omega_{23}\wedge(d_W\omega_1)=0$.
Now, recalling that the orientation of $D^\epsilon$
is opposite to that of $B^{0,\epsilon}$,
the result follows from direct computation.
\end{proof}

\subsection{Limit of contribution over $B^{0,\epsilon}$}

Now we want to push the expression of Lemma
\ref{p:B0epsilon} down to the boundary
$D\subset \hat{\mathbb{C}}\subset \partial H^3$.
This will be accomplished in Proposition \ref{p:B0}. First
we need to prove some preliminary results, which will also
be useful later.

By the uniformization of $X$ by $\Gamma$, we identify
$X$ with $\Gamma\backslash \Omega$. Then the hyperbolic
metric $g_{X}$ of constant curvature $-1$ on $X$
(or the flat metric $g_X$ of area $1$ in the case
that $X$ has genus 1) gives a
metric $e^{\phi(z)} |dz|^2$ on $\Omega$, invariant
under the action of $\Gamma$. The invariance implies that
\begin{equation}\label{e:trans-phi}
\phi(z)=\phi(\gamma z) +\log |\gamma'(z)|^2
\end{equation}
for all $z\in \Omega$ and $\gamma\in\Gamma$.

\begin{proposition}\label{p:t_asymptotic}
The set $D^\epsilon$ in $H^3$ is given by
$D^\epsilon
	=\{\, (t,x,y)\in H^3 \, |
		\, t= \mathfrak{t}({\epsilon},x,y)\, \}$,
where $\mathfrak{t}$ is a function satisfying
\[
\mathfrak{t}({\epsilon},x,y)
	=\epsilon e^{-\frac{\phi(x,y)}2}
				+k(\epsilon,x,y)\epsilon^3,
\]
where $k$, $k_x$ and $k_y$ exist and are bounded on
$\mathbf{D}\cup D$.
\end{proposition}

\begin{proof}
Let us recall that there is a unique defining function $r$
over a collar neighborhood $N$ of $X$ in $\overline{M}$
such that the rescaled metric $\bar{g}:=r^2{g}_M$
extends smoothly to $\overline{M}$, its restriction
to $X$ is the hyperbolic metric $g_{X}$ and
$|dr|_{\bar{g}}^2=1$. Let us denote the lifted defining
function over the inverse image of $N$ in $H^3$ by the
same notation $r$, and write $\hat{r}:=\frac{r}t$. Then
the three conditions on $\bar{g}$ imply that $\hat{r}$
extends smoothly to $\mathbf{D}\cup D$,
$ \lim_{t\to 0}\hat{r}(t,x,y)=e^{\frac{\phi}2}$,
and
$
\hat{r}_t
	=-\frac{1}{2}(\hat{r}_t^2
				+ \hat{r}_x^2+ \hat{r}_y^2)
				\hat{r}^{-1}t
$
respectively. Since $\vert\hat{r}_t\vert\leq C t$ for a
uniform constant $C$, we have
$\vert(\hat{r}-e^{\frac{\phi}{2}})_t\vert\leq Ct$.
Since $\hat{r}$ is smooth on $\mathbf{D}\cup D$, this
means
$\vert\hat{r}-e^{\frac{\phi}{2}}\vert\leq Ct^2$
for a uniform constant $C$, and therefore
\begin{equation}\label{e:bdy-defining-fn}
\hat{r}(t,x,y)=e^{\frac{\phi(x,y)}2}+\alpha(t,x,y)t^2,
\end{equation}
where $\alpha$ is uniformly bounded.

Similarly, since $\hat{r}_t=0$ on $D$, we have
$\hat{r}_{xt}=0$ on $D$. Again, since $\hat{r}$ is smooth,
we obtain
$\vert(\hat{r}-e^{\frac{\phi}{2}})_x\vert\leq Ct^2$
for some uniform constant $C$. This implies
\[
\hat{r}_x(t,x,y)=(e^{\frac{\phi(x,y)}{2}})_x+\alpha_x(t,x,y) t^2
\]
where $\alpha_x$ is uniformly bounded, and similarly
for $\hat{r}_y$. This implies the claimed expression for
$\mathfrak{t}$. Now, since
$k=-\alpha e^{-\frac{\phi}{2}}\hat{r}^{-3}$,
and $\hat{r}$ is nowhere zero on $D$, the result follows.
\end{proof}

A holomorphic $1$-form $\Phi$ with only simple zeroes over
$X$ is given by $h(z)dz$ over $\Omega$ with
$h(\gamma z) \gamma'(z)=h(z)$ for $\gamma\in\Gamma$.
The phase function $e^{i\theta(z)}:=h(z)/|h(z)|$ is well
defined over
$\Omega\setminus \cup_{\gamma\in\Gamma} \gamma Z$
where $Z:=\{z_1,\ldots, z_{2g-2}\}$ denotes the zero
set of $h(z)$ in a fixed fundamental domain $D$ of $\Gamma$.
The transformation law of $h(z)$ implies
\begin{equation}\label{e:trans-the}
i\theta(z)=i\theta(\gamma z) +\log \frac{\gamma'(z)}{|\gamma'(z)|}
\end{equation}
for $\gamma\in\Gamma$. Note that $\theta$ is defined
only up to an integer multiple of $2\pi$. By \eqref{e:trans-phi},
\eqref{e:trans-the}, it follows that
$e^{\phi(z)/2+i\theta(z)} dz=\omega_2+i\omega_3$
is invariant under the action of $\Gamma$; in particular,
\[
\omega_2= e^{\phi/2} (\cos\theta dx -\sin\theta dy),
\qquad
\omega_3= e^{\phi/2} (\sin\theta dx+ \cos\theta dy)
\]
provides us with an orthonormal invariant co-frame
$(\omega_2, \omega_3)$ over
$\Omega\setminus \cup_{\gamma\in\Gamma} \gamma Z$.
Now we obtain an orthonormal framing
\[
\F_{\Phi}=(f_2,f_3)
\qquad
\text{where} \quad f_2=\omega_2^*, f_3=\omega^*_3
\]
over $D':=D \setminus Z$.

Near a zero $z_k\in Z$, $h(z)$ has an expression
$h(z)=(z-z_k)\tilde{h}_k(z)$ such that $\tilde{h}_k(z)$ is
non-vanishing at $z_k$. Now we put
$e^{i\tilde{\theta}_k(z)}:= \tilde{h}_k(z)/|\tilde{h}_k(z)|$.
Since $\tilde{h}_k(z)$ is non-vanishing at $z=z_k$,
$\tilde{\theta}_k(z_k)$ is well-defined only up to an
integer multiple of $2\pi$. As in \eqref{e:frame-by-theta-til},
we define
\begin{equation}\label{e:frame-by-theta-til1}
\begin{split}
\tilde{\omega}_2
= e^{\phi/2} &(\cos({\tilde{\theta}_k}/{2}) dx
				 -\sin({\tilde{\theta}_k}/{2}) dy), \\
\tilde{\omega}_3
= e^{\phi/2} &(\sin({\tilde{\theta}_k}/{2}) dx
				+ \cos({\tilde{\theta}_k}/{2}) dy)
\end{split}
\end{equation}
at $z_k\in Z$.
Then the duals $(\tilde{f}_2,\tilde{f}_3)$ of
$(\tilde{\omega}_2,\tilde{\omega}_3)$ define an
orthonormal framing at $z_k\in Z$. That this
orthonormal framing is well-defined up to sign follows
from the fact that $h(\gamma z)\gamma'(z)=h(z)$
and the from the following equality for $\gamma\in\Gamma$
and $z,z_k\in\Omega$:
\[
(\gamma z -\gamma z_k)
= (z-z_k)\gamma_z(z)^{\frac12}\gamma_z(z_k)^{\frac12}.
\]

\begin{proposition}\label{p:omega23} The one form
$\omega_{23}$ on $\mathbf{D}'$ extends smoothly to a
form on $\mathbf{D}'\cup D'$. We have
\begin{equation*}
\lim_{t\to 0}\omega_{23}
= \frac{i}{2} \big(\,
	(\phi-2i\theta)_z dz - (\phi+2i\theta)_{\bar{z}} d\bar{z}
	\,\big),
\end{equation*}
where the convergence in the global coordinate on $H^3$ is uniform on
$\mathbf{D}'\cup D'$.
\end{proposition}

Note that the extension of $\omega_{23}$ to $D'$ coincides with
the connection form of the hyperbolic metric $e^\phi |dz|^2$, with
respect to our choice of orthonormal frame $\F_{\Phi}$.

\begin{proof}
By the Koszul formula, we have
\begin{equation*}\label{e:omega23}
\omega_{23}
= g([e_2,e_3],e_2) \omega_2
+ g([e_2,e_3], e_3) \omega_3
\end{equation*}
for an orthonormal frame $(e_1,e_2,e_3)$ where $e_1$
is orthogonal to $TD^\epsilon$. By the asymptotics of the
boundary defining function $r$ in \eqref{e:bdy-defining-fn},
we have
\begin{align*}
e_1=  t (1+&\tfrac{1}{4}t^2(\phi_x^2+\phi_y^2))^{-\frac12}\,
	(\tfrac{1}{2}t\phi_x \partial_x
	+\tfrac{1}{2}t\phi_y\partial_y
	+\partial_t)  +O(t^3),
\\
e_2 &=\alpha_{22} \bar{e}_2 +\alpha_{23} \bar{e}_3,
\quad
e_3 =\alpha_{32} \bar{e}_2 +\alpha_{33} \bar{e}_3
\end{align*}
with
\[
\bar{e}_2=  t ( 1+\tfrac{1}{4}t^2\phi_x^2)^{-\frac12}\,
	( \partial_x
	-\tfrac{1}{2}t\phi_x \partial_t) + O(t^3),
\qquad
\bar{e}_2=  t ( 1+\tfrac{1}{4}t^2\phi_y^2)^{-\frac12}\,
	( \partial_y
	-\tfrac{1}{2}t\phi_y \partial_t) + O(t^3).
\]
Here and below, we use $O(t^k)$ to indicate
a function of the form $a(t,x,y)t^k$ with respect to the
global coordinate on $H^3$, where $a$ is uniformly
bounded in $\mathbf{D}'\cup D'$. To compute
$g([e_2,e_3],e_2)$, $g([e_2,e_3],e_3)$, we consider
$[e_2,e_3]$ first. By an elementary computation,
\begin{align}\label{e:e2e3}
\begin{split}
[e_2,e_3]
=& \ \ (\alpha_{22}\alpha_{33}-\alpha_{23}\alpha_{32})
		[\bar{e}_2,\bar{e}_3] \\
&+ \big(\, \alpha_{22} \bar{e}_2(\alpha_{32})
		-\alpha_{32} \bar{e}_2(\alpha_{22})
		+\alpha_{23} \bar{e}_3(\alpha_{32})
		-\alpha_{33}\bar{e}_3(\alpha_{22})\, \big) \bar{e}_2 \\
&+ \big(\, \alpha_{22} \bar{e}_2(\alpha_{33})
		-\alpha_{32} \bar{e}_2(\alpha_{23})
		+\alpha_{23} \bar{e}_3(\alpha_{33})
		-\alpha_{33}\bar{e}_3(\alpha_{23})\, \big) \bar{e}_3.
\end{split}
\end{align}
Using Proposition \ref{p:t_asymptotic}, we have
\[
[\bar{e}_2,\bar{e}_3]
= (\tfrac{1}{2}t^2\phi_y)\partial_x
	-(\tfrac{1}{2}t^2\phi_x)\partial_y + O(t^3),
\]
from which we also have
\begin{align*}\label{e:right1}
g([\bar{e}_2,\bar{e}_3],e_2)
=&\alpha_{22}(\tfrac{1}{2}t\phi_y)
	+\alpha_{23}(-\tfrac{1}{2}t\phi_x) +O(t^2)
= \tfrac12 t( \cos\theta\, \phi_y +\sin\theta\, \phi_x) +O(t^2),\\
g([\bar{e}_2,\bar{e}_3],e_3)
=&\alpha_{32}(\tfrac{1}{2}t\phi_y)
	+\alpha_{33}(-\tfrac{1}{2}t\phi_x) +O(t^2)
= \tfrac12 t( \sin\theta\, \phi_y -\cos\theta\, \phi_x) +O(t^2) .
\end{align*}
Here we used the fact $\alpha_{22}=\alpha_{33}=\cos\theta+O(t)$,
$\alpha_{23}=-\alpha_{32}=-\sin\theta +O(t)$.
Denoting by $E$ the sum of the terms in the second and third lines
on the right hand side of \eqref{e:e2e3},
\begin{align*}
g(E, e_2)=& -t (\cos\theta\, \theta_x -\sin\theta\, \theta_y) +O(t^2),\\
g(E, e_3)=&  -t (\sin \theta\, \theta_x +\cos\theta\, \theta_y) + O(t^2).
\end{align*}
Finally we need
\begin{equation*}
\omega_2= t^{-1} (\cos\theta dx -\sin\theta dy +O(t)),
\qquad
\omega_3= t^{-1}(\sin\theta dx+\cos\theta dy +O(t)).
\end{equation*}
Combining all the proved equalities, we have
\begin{align*}
\omega_{23}
=& \ \ (\frac12\cos\theta \, \phi_y
			+\frac12\sin\theta\, \phi_x
			-\cos\theta\,\theta_x
			+\sin\theta\,\theta_y)
		(\cos\theta\, dx-\sin\theta\, dy) \\
&+(\frac12\sin\theta \, \phi_y
			-\frac12\cos\theta\, \phi_x
			-\sin\theta\,\theta_x
			-\cos\theta\,\theta_y)
		(\sin\theta\, dx+\cos\theta\, dy) + O(t)\\
=& \  d\theta +  \frac12(\phi_y dx -\phi_x dy) + O(t).
\end{align*}
This completes the proof.
\end{proof}

Now, we define $c_1=\chi_1(\frac{\partial}{\partial w})$
and $c_{23}=\chi_{23}(\frac{\partial}{\partial w})$, where
$\chi_1$ and $\chi_{23}$ were defined in Lemma
\ref{p:B0epsilon}, and $w$ is a local coordinate in $W$.
We will write $'$ for the derivative with respect to $w$, for
instance, $\phi'=\frac{\partial}{\partial w} \phi$.

\begin{proposition}\label{p:c123} The functions $c_1,c_{23}$
on $W\times \mathbf{D}'$ extend smoothly to functions on
$W\times (\mathbf{D}'\cup D')$. We have
\begin{equation*}
\lim_{t\to 0} c_1 = -\frac12 \phi'\circ f,
\qquad \lim_{t\to 0} c_{23}
	=\theta'\circ f + i(\frac{\phi}2-i\theta)_z f' ,
\end{equation*}
and the convergence in the global coordinate on $H^3$ is
uniform on $\mathbf{D}'\cup D'$. We also have
$\lim_{t\to0}\chi_1(\frac{\partial}{\partial \bar{w}})
	=\bar{c}_1$,
$\lim_{t\to0}\chi_{23}(\frac{\partial}{\partial \bar{w}})
	=\bar{c}_{23}$.
\end{proposition}

\begin{proof}
Observe that $c_1$ is given by
\[
(s\circ (u,\mathbf{f}))^* \theta_1
		(\tfrac{\partial}{\partial w})
=\theta_1(s_* u_*\tfrac{\partial}{\partial w})
	+ s^* \theta_1( \mathbf{f}_*u_*\tfrac{\partial}{\partial w})
  =\omega_1 ( \mathbf{f}'),
\]
where the second equality holds since
$s_*u_*\tfrac{\partial}{\partial w}$ is vertical. Recall that the
level surface $D^\epsilon$ is given by
$\{\,
	(t,x,y) \in {H}^3\,
	| \,  t=\t(\epsilon,x,y)= \epsilon e^{-\frac{\phi(x,y)}{2}}
	+O(\epsilon^3)
  \, \}$,
and that the definition of $\mathbf{f}$ near the boundary
given by \eqref{e:bf-f}
involves translation along gradient curves for $r$.
Since translation from $D$ to $D^\epsilon$ introduces an
error of $O(\epsilon^2)$, and since $f_z$ and $f_z'$ are
bounded on $D$, we have
\[
\mathbf{f}(w,(t,z))
=\big(\mathfrak{t}\big(r(t,z),f(z)\big),f(z)\big) + O(t^2).
\]
Here and below we understand $O(t^2)$ to be uniform as
discussed in the previous proposition. Therefore we have
\[
\mathbf{f}'= f' \tfrac{\partial}{\partial z}
	- \big(
		\tfrac{1}{2}t
			(\phi'\circ f+ \phi_z f')
	\big) \tfrac{\partial}{\partial t}
	+  O(t^2).
\]
The one form $\omega_1$ is the dual of the first component
$e_1$ of the orthonormal frame over the level surface
$D^\epsilon$ so that
\[
  \omega_1
  ={{(1+\tfrac{1}{4}t^2(\phi_x^2 +\phi_y^2))^{-1/2}}} t^{-1}
  	\big( \tfrac{1}{2}t\phi_z dz
		+\tfrac{1}{2}t\phi_{\bar{z}} d\bar{z}
		+dt \big) + O(t^3).
\]
Hence, we have
\[
\omega_1(\mathbf{f}')
= \frac12 \phi_z f'
	-\frac12 \phi'\circ f -\frac12 \phi_z f' + O(t)
= -\frac12 \phi'\circ f  + O(t) ,
\]
from which it follows
\begin{equation}\label{e:c1}
  \lim_{t\to 0} c_1
  =\lim_{t\to 0}\omega_1(\mathbf{f}')
  = -\frac12 \phi'\circ f.
\end{equation}
As above, $c_{23}$ is given by
\[
(s\circ (u,\mathbf{f}))^* \theta_{23}
  		(\tfrac{\partial}{\partial w})
=\theta_{23}(s_*u_*\tfrac{\partial}{\partial w})
  	+ s^* \theta_{23}( \mathbf{f}_*u_*\tfrac{\partial}{\partial w})
  =\theta_{23}(s') + \omega_{23} ( \mathbf{f}').
\]
Now, we have $\theta_{23}=\mathcal{L}_{g^{-1}}^*(-2(ih)^*)$
and
$\lim_{t\to 0}(\mathcal{L}_g)_*s_*u_*(\frac{\partial}{\partial w})
=-\frac{1}{2}\theta_w(ih)$, where
$h=\left(\begin{smallmatrix} 1 & 0 \\ 0 & -1\end{smallmatrix}\right)
	\in \mathfrak{sl}_2(\mathbb{C})$
and $\mathcal{L}_g$ is the left translation by
$g\in {PSL}_2(\mathbb{C})$ (see Section 3 of \cite{Yo}).	Hence,
\begin{equation}\label{e:c23}
  \lim_{t\to 0}c_{23}
  = \lim_{t\to 0} \,
  	\big(\theta_{23}(s')+\omega_{23} ( \mathbf{f}')\big)
  = \theta'\circ f + i(\frac{\phi}2-i\theta)_z f'.
\end{equation}
The equalities \eqref{e:c1} and \eqref{e:c23} complete the proof
of the first two equalities. Replacing $\tfrac{\partial}{\partial w}$
with $\tfrac{\partial}{\partial \bar{w}}$
in the computations above gives the last part of the statement.
\end{proof}

We denote by the same notations $\omega_{23}$,
$c_1$, $c_{23}$, the restriction to $W\times D'$ of the
extensions of  $\omega_{23}$, $c_1$, $c_{23}$ respectively,
obtained in Propositions \ref{p:omega23} and \ref{p:c123}.

Now let us introduce some additional notation.
The local coordinate expression for the members of the family
$\{\Phi(w)\}$ can be identified with a map
$h:\{(w,\Omega_{u(w)}): w\in W\}\to \mathbb{C}$.
We define $z_k:W\to \mathbb{C}$ to be the
coordinates in $\Omega_{u(w)}$ of the zeroes
of $\Phi(w)$, that is, $h(w,z_k(w))=0$ for all $w\in W$.
Near each $z_k(w)$, we define $\tilde{h}_k$ by
\begin{equation}\label{e:tilde-h}
h(w,f(w,z))=\big(f(w,z)-f(w,z_k(w))\big)\,\tilde{h}_k(w,f(w,z))
\end{equation}
for all $w\in W$.

\begin{proposition}\label{p:B0} The limit of the 1-form
\begin{align}\label{e:integral_limit1}
\lim_{\epsilon\to 0} \int'_{ B^{0,\epsilon}}
& H^*\Cbar
\end{align}
over $W$ is finite, and its $(1,0)$ part equals
\begin{align}\label{e:integral_limit2}
  \frac{1}{4\pi^2} \int'_{ D'}
  \big(\, d_{D} \omega_{23}\wedge (c_1+ic_{23}) dw
  	+\omega_{23}\wedge (d_D(c_1+ic_{23})\wedge dw
	+i \partial_w \omega_{23})
  \, \big)
\end{align}
where $d=d_W+d_D=\partial_w+\overline{\partial_w} + d_D$
over $W\times D$.
\end{proposition}

\begin{proof}
We have that
\[
\lim_{\epsilon\to 0} \int'_{ B^{0,\epsilon}}
 H^*\Cbar
=\lim_{\epsilon\to 0} \int'_{ s(P_\epsilon(D'))}
 H^*\Cbar
=\lim_{\epsilon\to 0} \int'_{ D'}
 P_\epsilon^*s^*H^*\Cbar.
\]
Propositions \ref{p:omega23} and \ref{p:c123}, and the
definition of admissible singularity, show that $s^*H^*\Cbar$
extends continuously to $D'$, and is uniformly bounded.
Therefore we can exchange limit and integral in the last
integral.
Hence, the integral \eqref{e:integral_limit2}
equals the $(1,0)$ part of \eqref{e:integral_limit1} by Lemma
\ref{p:B0epsilon}, Propositions \ref{p:omega23} and \ref{p:c123}.
Now we prove the integral \eqref{e:integral_limit2} is
finite. By equation \eqref{e:tilde-h}, near $z_k\in Z$ we have
\begin{align*}
  (c_1+ic_{23})(z)
  =&-\frac12(\, ({\phi}-2i\theta)\circ f \,)'(z)\\
  =&\, \frac12 \frac{{f}'(z) -f'(z_k) -f_z(z_k)({z}_k)'}{z-z_k}
  		-\frac12 (\, (\phi -\log \tilde{h})\circ f\,)'(z).
\end{align*}
Note that
$d_D\omega_{23}$ is a constant times the volume form
and $c_1+ic_{23}$ is singular at $Z$ by the above equality,
but its wedge product with the volume form is integrable. For the
second term, we use the following formula,
\begin{align*}
d_D(c_1&+ic_{23})\wedge dw +i \partial_w \omega_{23}\\
&=-(\phi'_{\overline{\zeta}}\circ f \overline{f}_z
	+ \phi_{\zeta\overline{\zeta}}\circ f\overline{f}_z f')
			d z\wedge d w
 -(\phi'_{\overline{\zeta}}\circ f \overline{f}_{\overline{z}}
	+ \phi_{\zeta\overline{\zeta}}\circ f\overline{f}_{\overline{z}} f')
			d \overline{z}\wedge d w
\end{align*}
where $\zeta=f(z)$, which can be derived
from Propositions \ref{p:omega23} and \ref{p:c123}. Although
$\omega_{23}$ is singular at $Z$, its wedge product with
the expression above is integrable. This shows that the integral
\eqref{e:integral_limit2} is finite, hence $(1,0)$ part of
\eqref{e:integral_limit1} is finite. Similarly, the $(0,1)$ part of
\eqref{e:integral_limit1} is equal to the complex conjugate of
\eqref{e:integral_limit2} and is therefore also finite.
\end{proof}

\subsection{Holomorphic variation of $\overline{\mathbb{CS}}$}

We begin this subsection with

\begin{proposition}\label{p:hol-var}
Over $W\subset\mathbb{C}$,
we have
\begin{align*}
	d\,(u^*\overline{\mathbb{CS}})
	=& \int'_{{B}^{0}} H^*\Cbar
		- \frac{1}{2\pi} \sum_{y\in \partial\L^1}
		\sigma^*(\theta_1+i\theta_{23})\big\vert_{y} \\
	& \quad + \frac{1}{2\pi}\sum_{y\in\partial \L^2}
			 \sigma^*(\theta_1+i\theta_{23})\big\vert_{y}
		+ \sum_{r=1}^g \int'_{T_{r}-L_r T_{r}} H^*\Cbar.
\end{align*}
Here the sums over $\partial\L^1$ and $\partial\L^2$ are taken
with signs inherited from the orientations on $\L^1$ and $\L^2$.
\end{proposition}
\begin{proof}First, note that $\lim_{\epsilon\to 0}d_W(u^*\mathbb{CS}^\epsilon) =d_W(u^*\mathbb{CS})$ since the diverging term $\frac{2}{\pi}(1-g) \log\epsilon$
in Definition \ref{d:CS} vanishes under $d_W$.
By Proposition \ref{p:B0}, the partial integral over $B^{0,\epsilon}$
in converges to a finite limit as $\epsilon\to 0$.
By  Lemma \ref{l:L2} and a similar analysis in the proof of
Proposition \ref{p:c123}, the right hand side of the equality in Lemma 5.2  also converges as $\epsilon\to 0$.
Hence this is also true for the
last terms in \eqref{e:var-1} given by
the sum of the partial integrals over $(T_r^{\epsilon}-L_rT_r^{\epsilon})$.
Taking $\epsilon\to 0$ on both sides of \eqref{e:var-1}, and
using Lemmas \ref{l:L1} and \ref{l:L2} we have the result.
\end{proof}

The remainder of this section is devoted to finding an
explicit expression for
$d\,\overline{\mathbb{CS}}(\varpi)$
in the case that $\varpi\in T^{1,0}U$ at $u_0$ is a holomorphic
tangent vector. The final result is given
in Theorem \ref{t:var-conCS}.

\begin{lemma}\label{l:(0,1)form}
For the holomorphic curve $u:W\to U$, we have
\begin{equation*}
{\sigma}^*(\theta_1+i\theta_{23})\big\vert_{y_j(0)}
 				- {\sigma}^*(\theta_1+i\theta_{23})\big\vert_{L_{r(j)}(u(0))y_j(0)}
			=- (L_{r(j)}\circ u)^*(\theta_1+i\theta_{23})
\end{equation*}
where
$(L_{r(j)}\circ u)^*(\theta_1+i\theta_{23})$ is a $(0,1)$-form on $W$ for $j\in\mathcal{J}$.
\end{lemma}

\begin{proof}
For brevity we write $L_{r(j)}(w):=L_{r(j)}(u(w))$ and $y_j:=y_j(0)$.
The map
$w\mapsto$ $\sigma(w,L_{r(j)}(0)y_j)$
$=L_{r(j)}(w)\sigma(w,y_j)$
is the composition of the maps
\[
 W
 \xrightarrow { L_{r(j)}\times\sigma(y_j)}
 	PSL_2(\mathbb{C})\times PSL_2(\mathbb{C})
 \xrightarrow {G} PSL_2(\mathbb{C})
\]
where $G$ denotes the multiplication map. Since
$\theta_1+i\theta_{23}$ is a bi-invariant $1$-form
on $ PSL_2(\mathbb{C})$, we obtain
$G^*(\theta_1+i\theta_{23})
	=p_1^*(\theta_1+i\theta_{23})+p_2^*(\theta_1+i\theta_{23})$
where $p_i$ denotes the projection onto $i$-th factor
$PSL_2(\mathbb{C})$. It follows that
\[
\sigma(L_{r(j)}(0) y_j)^*(\theta_1+i\theta_{23})
=((L_{r(j)}\circ u)\sigma(y_j))^*(\theta_1+i\theta_{23})
=(L_{r(j)}\circ u)^*(\theta_1+i\theta_{23})+\sigma(y_j)^*(\theta_1+i\theta_{23}).
\]
Hence,
\[
\sigma^*(\theta_1+i\theta_{23})\big\vert_{y_j}
	-  \sigma^*(\theta_1+i\theta_{23})\big\vert_{L_{r(j)}(0)y_j}
= -(L_{r(j)}\circ u)^*(\theta_1+i\theta_{23}).
\]
Since $L_{r(j)}\circ u:W\to {PSL}_2(\mathbb{C})$ is a
holomorphic map, and $\theta_1+i\theta_{23}$ is a
$(0,1)$-form on ${PSL}_2(\mathbb{C})$ (see the section
3 of \cite{Yo}), the statement follows.
\end{proof}

\begin{lemma}\label{p:boundary-pair}
The partial integral
$\sum_{r=1}^g\int'_{T_r-L_rT_r} H^*{\Cbar}$
is a $(0,1)$-form over $W$.
\end{lemma}
\begin{proof}
For each $w\in W$ and $x\in D^r$,
\[
\begin{split}
H(w, s_0(L_r(0)x))
	&=s(u(w),\mathbf{f}(w,L_r(0) x))\\
	&=s(u(w),L_r(w)\mathbf{f}(w,x))
			=L_r(w) s(u(w),\mathbf{f}(w,x))
			=L_r(w) H(w,s_0(x)),
\end{split}
\]
where $L_r(w):=L_r(u(w))$. Hence
$H:W\times L_rT_r\to PSL_2(\mathbb{C})$
can be considered as the composition of the maps
\begin{equation*}
W\times T_r
\xrightarrow {L_r\times H}
	PSL_2(\mathbb{C})\times PSL_2(\mathbb{C})
\xrightarrow {G} PSL_2(\mathbb{C})
\end{equation*}
where
$(L_r\times H)(w,s_0(x))= (L_r(w),H(w,s_0(x)))$
and $G$ denotes the multiplication map.
The pull back of $\Cbar$ by $G$ is given by
$G^* \Cbar
	=p_1^*\Cbar+(G^* \Cbar)^{2,1}
		+ (G^* \Cbar)^{1,2}+ p_2^* \Cbar$,
where
$p_i:PSL_2(\mathbb{C})\times PSL_2(\mathbb{C})
		\to PSL_2(\mathbb{C})$,
$i=1,2$, are the projections from the two factors,
and where superscripts on a form indicate the degree in
the two factors. Taking the pull back of $G^*\Cbar$
by $L_r\times H$, we have
\[
(G(L_r\times H))^* \Cbar
= L_r^*\Cbar + (L_r\times H)^*(G^*\Cbar)^{2,1}
	+ (L_r\times H)^*(G^*\Cbar)^{1,2}
	+ H^*\Cbar.
\]
Hence we have the following equality for the partial integrals:
\begin{align*}
\int'_{T^r} H^* \Cbar - \int'_{L_rT^r} H^*\Cbar
	= -\int'_{T^r} (L_r\times H)^* (G^*\Cbar)^{1,2}.
\end{align*}
Since the map
$w\in W\mapsto L_r(w)\in PSL_2(\mathbb{C})$
is holomorphic, the $dw$ term in
$(L_r\times H)^* (G^*\Cbar)^{1,2}$
vanishes under the above partial integration. Hence the
$1$-form on $W$ obtained by the partial integration of
$\int'_{T^r-L_rT^r} H^*{\Cbar}$ does not involve $dw$,
that is, it is of type $(0,1)$.
\end{proof}

From now on, $\dot{\phantom{x}}$ will denote the derivative with
respect to $w$ at $w=0$, for instance,
$\dot{\phi}=\left.\frac{\partial}{\partial w}\right\vert_{w=0}\phi$.
By the results on varying the hyperbolic metric in \cite{Ahl61}, we have,
for all $z\in\Omega$,
\begin{equation}\label{e:Ahlfors}
\dot{\phi} +\phi_z \dot{f} +\dot{f}_z=0.
\end{equation}
(The same is true for the flat metric of area 1 in the case
that the genus of $X$ is 1.)
From this, we also have
\begin{align}\label{e:phi-var}
\dot{\phi}_z +\phi_{zz} \dot{f} +\phi_z \dot{f}_z +\dot{f}_{zz}=0 ,
\ \qquad
\dot{\phi}_{\bar{z}} +\phi_{z\bar{z}} \dot{f}
		+\phi_z \dot{f}_{\bar{z}} +\dot{f}_{z\bar{z}}=0.
\end{align}
Since $2i\theta=\log h-\log \bar{h}$,
\begin{equation}\label{e:theta-pole}
2i\theta_z = \frac{h_z}{h},
\qquad
2i\theta_{\bar{z}} =  -\frac{\bar{h}_{\bar{z}}}{\bar{h}},
\qquad
\theta_{z\bar{z}}=0 .
\end{equation}
Since $\Phi$ is a holomorphic family, we also have
\begin{equation}\label{e:theta-trivial}
 \qquad \dot{\theta}_{\bar{z}}= 0.
\end{equation}
It will be convenient in what follows to make the definition
$\psi:=\phi-2i\theta$.

\begin{lemma}\label{l:inv}
The following terms are invariant under the action of $\Gamma$,
\begin{equation*}
	-\dot{f}_{z} - (2i\theta)\dot{} -(2i\theta)_z \dot{f}
 	= \dot{\psi} +\psi_z \dot{f}.
\end{equation*}
\end{lemma}
\begin{proof}
The equality follows from
$\dot{\phi} +\phi_z \dot{f} +\dot{f}_z=0$.
To see the invariance under the action of $\Gamma$, we note
\begin{align*}
	(\phi-2i\theta)\dot{}\,(z)
 	&= (\phi-2i\theta)\dot{}\,(\gamma z)
 		+ (\phi-2i\theta)_z(\gamma z) \dot{\gamma}(z),
	\\
 	(\phi-2i\theta)_z(z)
	&= (\phi-2i\theta)_z(\gamma z)\gamma_z(z),
\end{align*}
which follow from \eqref{e:trans-phi} and \eqref{e:trans-the}.
Combining these and
$\dot{f} \circ \gamma = \dot{\gamma}+\gamma_z \dot{f}$
completes the proof.
\end{proof}

From now on, for convenience, we abbreviate $z_k(0)$ to $z_k$,
and $\dot{z}_k(0)$ to $\dot{z}_k$.

\begin{proposition}\label{p:var-CS1}
For $\varpi\in T^{1,0}U$ at $u_0\in U$,
we have
\begin{align*}
	\partial\, \overline{\mathbb{CS}}(\varpi)
	=&\, \frac{1}{4\pi^2} \int_{D} d\omega_{23}
			\wedge (c_1+ic_{23}) +\omega_{23}
			\wedge ( d(c_1+ic_{23}) -i \dot{\omega}_{23} )\\
& -\frac{1}{4\pi} \sum_{z_k\in Z}
		\big(\, \dot{f}_z+\frac12(\log\tilde{h}_k)\dot{}
		+\frac12(\log\tilde{h}_k)_z\dot{f}
		-(\phi-\frac12\log\tilde{h}_k)_z f_z\dot{z}_k\, \big)(z_k)
\end{align*}
where $Z$ denotes the set of zeros of $\Phi$ in the
fundamental domain $D$ of $\Gamma$ and
$\tilde{h}_k$ is defined by equation \eqref{e:tilde-h}.
\end{proposition}

\begin{proof}
By Proposition \ref{p:hol-var} and Lemma \ref{l:(0,1)form},
$\partial\, \overline{\mathbb{CS}}(\varpi)$
is equal to the evaluation of the one form
\begin{align*}
	&\int'_{{B}_{0}} H^*\overline{C}
	+n(j)\frac{1}{2\pi} \sum_{j\in \mathcal{J}}
		(L_{r(j)}\circ u)^*(\theta_1+i\theta_{23})\\
	& \quad +\frac{1}{2\pi} \sum_{y\in (\partial\L^2\cap D)}
		\sigma^*(\theta_1+i\theta_{23})\big\vert_{y}
	+\sum_{r=1}^g \int'_{({T}_{r}-L_r T_{r})} H^*\overline{C}
\end{align*}
on $\tfrac{\partial}{\partial w}$.
Here $n(j)$ is the index of the singularity at the corresponding
component of $\L$, so $n(j)=1$ or $-1$
if the points $y_j(0)$, $L_{r(j)}(0)y_j(0)$ are in $\partial\L^1$
or $\partial\L^2$ respectively. By Lemma \ref{p:B0epsilon} and
Proposition \ref{p:B0}, the evaluation of the first term on
$\tfrac{\partial}{\partial w}$ is given by
\begin{equation*}
	\frac{1}{4\pi^2} \int_{D} d\omega_{23}\wedge (c_1+ic_{23})
		+\omega_{23}\wedge
				\big( d(c_1+ic_{23}) -i \dot{\omega}_{23}\big).
\end{equation*}
The second and fourth terms vanish on $\tfrac{\partial}{\partial w}$,
since they are $(0,1)$-forms by Lemmas \ref{l:(0,1)form}
and \ref{p:boundary-pair}.
Using Lemma \ref{l:L2}, and following the proof of
Proposition \ref{p:c123}, we find that the third term
evaluated on $\tfrac{\partial}{\partial w}$ is given by
\begin{align*}
	&\bigg(\frac{1}{2\pi} \sum_{y\in (\partial\L^2\cap D)}
		\sigma^*(\theta_1+i\theta_{23})\big\vert_{y}\bigg)
		(\frac{\partial}{\partial w})
		\notag\\
	=& \frac{1}{4\pi}\sum_{z_k\in Z}
		\big(\,
			(\dot{\phi}-i\tilde{\theta}\,\dot{}\,)
			+(\phi - i\tilde{\theta})_z \dot{f}
			+(\phi-i\tilde{\theta})_z f_z \dot{z}_k
		\, \big)(z_k)\\
	=& \frac{1}{4\pi} \sum_{z_k\in Z}
		\big(\,
		-\dot{f}_z
		-\frac12(\log\tilde{h}_k)\dot{}
		-\frac12(\log\tilde{h}_k)_z\dot{f}
		+(\phi-\frac12\log\tilde{h}_k)_z f_z\dot{z}_k
		\, \big)(z_k).\notag
\end{align*}
Here the last equality follows from \eqref{e:frame-by-theta-til1}
and \eqref{e:Ahlfors}. This completes the proof.
\end{proof}

\begin{proposition} \label{p:var-B0}
The following equality holds:
\begin{align*}
	&\frac{1}{4\pi^2} \int_{D}  d\omega_{23}\wedge (c_1+ic_{23})
	+\omega_{23}\wedge ( d(c_1+ic_{23}) -i \dot{\omega}_{23})\\
	=& -\frac{1}{2\pi^2} \lim_{\delta\to 0} \int_{D_\delta}
	( \phi_{zz} -  \frac12 \phi_z^2-2\theta_z^2 -2i\theta_{zz})
	\mu\ d^2z\\
  & \quad -\frac{1}{4\pi} \sum_{z_k\in Z}\big(\,
  		2\dot{f}_z +(\log\tilde{h}_k)\dot{}
			+(\log\tilde{h}_k)_z\dot{f}+(\phi-\log\tilde{h}_k)_z f_z \dot{z}_k
	\, \big)(z_k),
\end{align*}
where $D_\delta$ is a subset of
$D$ whose $\delta$-open neighborhoods of $Z$ are
removed and
$d^2z=\frac{i}2 dz\wedge d\bar{z}$.
\end{proposition}
Note that, since circles are preserved under holomorphic
change of coordinates, the limit as $\delta\to 0$ is independent
of the choice of local coordinates.

\begin{proof}
By Proposition \ref{p:c123},
\begin{align}\label{e:c123}
\begin{split}
	c_1+ic_{23}
	=& - \frac12\big(\, \psi\circ f \, \big)\,\dot{}
	\\
	=& -\frac12 (\dot{\phi} +\phi_z \dot{f}
				-2i\dot{\theta} -2i\theta_z \dot{f})
	= (i\dot{\theta}+i\theta_z \dot{f} +\frac12 \dot{f}_z)
\end{split}
\end{align}
where we used \eqref{e:Ahlfors} for the third equality.
From \eqref{e:c123}, we can also derive
\begin{align}\label{e:der-c}
	d(c_1+ic_{23})
	= -\frac12\big(\,
		( \dot{\psi}_z + \psi_{zz} \dot{f}+ \psi_z \dot{f}_z ) dz
		+ (\dot{\psi}_{\bar{z}}
		+ \psi_{z\bar{z}} \dot{f}
		+ \psi_z \dot{f}_{\bar{z}} )
				d\bar{z}
		\, \big).
\end{align}
By Proposition \ref{p:omega23},
\begin{align}\label{e:var-omega}
	-i\dot{\omega}_{23}
	= \frac12\big(\,
		(\dot{\psi}_z +\psi_{zz} \dot{f}+\psi_z \dot{f}_z)dz
		+(- \dot{{\bar{\psi}}}_{\bar{z}}-\bar{\psi}_{z\bar{z}} \dot{f}
			+ \psi_z \dot{f}_{\bar{z}} ) d\bar{z}
		\, \big).
\end{align}
Again by Proposition \ref{p:omega23},
\begin{equation*}
	d\omega_{23}= -i\phi_{z\bar{z}} dz\wedge d\bar{z}
				 = -i\psi_{z\bar{z}} dz\wedge d\bar{z} .
\end{equation*}
Combining this and \eqref{e:c123}, \eqref{e:der-c},
\eqref{e:var-omega}, we get
\begin{equation*}
	d\omega_{23}\wedge(c_1+ic_{23})
	= -i \psi_{z\bar{z}} (i\dot{\theta} +i\theta_z \dot{f} +\frac12 \dot{f}_z)
		dz\wedge d\bar{z},
\end{equation*}
which is an invariant $(1,1)$-form under the action of $\Gamma$
by Lemma \ref{l:inv} and
\begin{equation*}
	\omega_{23}\wedge (d(c_1+c_{23}) -i\dot{\omega}_{23})
	= -\frac{i}2 \psi_z(\dot{\phi}_{\bar{z}}+\phi_{z\bar{z}}\dot{f})
			dz\wedge d\bar{z}
	=  \frac{i}2 \psi_z (\phi_z \dot{f}_{\bar{z}}+\dot{f}_{z\bar{z}})
			dz\wedge d\bar{z}
\end{equation*}
where we used \eqref{e:phi-var} for the last equality.

By the above equalities,
\begin{align}\label{e:der-b1}
\begin{split}
	&\int_{D_\delta} d\omega_{23}\wedge(c_1+ic_{23})
	= -i \int_{D_\delta} \psi_{z\bar{z}}
		(i\dot{\theta} +i\theta_z \dot{f} +\frac12 \dot{f}_z)\
			dz\wedge d\bar{z}
	\\
	=&  i\int_{D_\delta}
			\psi_z (i\theta_z \dot{f}_{\bar{z}} +\frac12 \dot{f}_{z\bar{z}})\
				dz\wedge d\bar{z}
		+i \int_{\partial D_\delta}
			\psi_z(i\dot{\theta} +i\theta_z \dot{f} +\frac12 \dot{f}_z)
				dz
	\\
	=&  i\int_{D_\delta}
			(\psi_z i\theta_z -\frac12\psi_{zz}) \dot{f}_{\bar{z}}\
				dz\wedge d\bar{z}
		+i \int_{\partial D_\delta}
			\psi_z(i\dot{\theta} +i\theta_z \dot{f} +\frac12 \dot{f}_z)
			dz
	+\frac12\psi_z\dot{f}_{\bar{z}} d\bar{z}
\end{split}
\end{align}
where $\partial D_\delta$ has the induced orientation from
$D_\delta$.
In the integral over $\partial D_\delta$, the contributions from $C^r$ and
$-L_r(0)C^r$ cancel, since the integrands concerned are invariant.
We also have
\begin{align}\label{e:der-b2}
\begin{split}
&\int_{D_\delta}
	\omega_{23} \wedge  (d(c_1+c_{23}) -i\dot{\omega}_{23})
=  \frac{i}2  \int_{D_\delta}
			\psi_z (\phi_z \dot{f}_{\bar{z}}+\dot{f}_{z\bar{z}})
				\ dz\wedge d\bar{z}\\
=&  	\frac{i}2 \int_{D_\delta}
			(\psi_z \phi_z - \psi_{zz}) \dot{f}_{\bar{z}}
				\  dz\wedge d\bar{z}
	+ \frac{i}2 \int_{\partial D_\delta}
				\psi_z \dot{f}_{\bar{z}} d\bar{z},
\end{split}
\end{align}
where once again the contributions from $C^r$ and
$-L_r(0)C^r$ cancel in the integral over $\partial D_\delta$.
By \eqref{e:c123}, \eqref{e:der-b1} and \eqref{e:der-b2},
\begin{align}\label{e:der-b3}
\begin{split}
& \int_{D_\delta}  d\omega_{23}\wedge (c_1+ic_{23})
		+\omega_{23}\wedge ( d(c_1+ic_{23}) -i \dot{\omega}_{23} )\\
=& 	\frac{i}2 \int_{D_\delta}
		(\psi_z \bar{\psi}_z - 2\psi_{zz}) \dot{f}_{\bar{z}}
			\  dz\wedge d\bar{z}
	-\frac{i}{2} \int_{\partial D_\delta}
		\psi_z(\, \psi\circ f \,)\,\dot{}\, dz
+i \int_{\partial D_\delta} \psi_z\dot{f}_{\bar{z}} d\bar{z} .
\end{split}
\end{align}
For the last integral on the right hand side of \eqref{e:der-b3},
we have
\begin{align}\label{e:boundary-last}
\begin{split}
i \int_{\partial D_\delta} \psi_z\dot{f}_{\bar{z}}\, d\bar{z}
=& -i \sum_{z_k\in Z}
		\int_{|z-z_k|=\delta}
			\big(\, -\frac{1}{z-z_k}+(\phi-\log\tilde{h}_k)_z\, \big)
				\,\dot{f}_{\bar{z}}\, d\bar{z}\\
=&  -i \sum_{z_k\in Z}
		\int_{|z-z_k|=\delta}
			-\frac{1}{z-z_k}\,\dot{f}_{\bar{z}}\, d\bar{z}
	+O(\delta)
= O(\delta).
\end{split}
\end{align}
To analyze the second integral on the right hand side of
\eqref{e:der-b3}, we use \eqref{e:tilde-h}. This implies
that, near $z_k\in Z$, we have
\begin{equation}\label{e:tilde-n-b1}
(\, ({\phi}-2i\theta)\circ f \,)\,\dot{}(z)
= - \frac{\dot{f}(z) -\dot{f}(z_k) -f_z(z_k)\dot{z}_k}{z-z_k}
	+ (\, (\phi -\log \tilde{h}_k)\circ f\,)\,\dot{}(z).
\end{equation}
Therefore, we can rewrite the second integral of \eqref{e:der-b3} as
\begin{align*}
& -\frac{i}{2} \int_{\partial D_\delta}
	\psi_z(\, ({\phi}-2i\theta)\circ f \,)\,\dot{}\, dz \\
=& \frac{i}{2}\sum_{z_k\in Z}
	\Big(\,
		\int_{|z-z_k|=\delta}
			(-\frac{1}{z-z_k}+(\phi-\log\tilde{h}_k)_z)
				\, \big(\,
					- \frac{\dot{f}(z) -\dot{f}(z_k) -f_z(z_k)\dot{z}_k}{z-z_k}
				\, \big) \, dz\\
		&\qquad \qquad
		+ \int_{|z-z_k|=\delta}
 			(-\frac{1}{z-z_k}+(\phi-\log\tilde{h}_k)_z)\,
			\big(\, (\phi -\log \tilde{h}_k)\circ f\,)\,\dot{}(z) \, \big) \, dz
	\, \Big) \notag \\
=&-\pi   \sum_{z_k\in Z}
	\big(\,
		\dot{f}_z (z_k)
		+ (\phi-\log\tilde{h}_k)_z(z_k) f_z(z_k) \dot{z}_k
		- ((\phi -\log \tilde{h}_k)\circ f\,)\,\dot{}(z_k)
	\, \big)
	+O(\delta) \notag\\
=& -\pi \sum_{z_k\in Z}
	\big(\,
		2\dot{f}_z
		+(\log\tilde{h}_k)\dot{}
		+(\log\tilde{h}_k)_z\dot{f}
		+(\phi-\log\tilde{h}_k)_z f_z \dot{z}_k
	\, \big)(z_k)
	+O(\delta).\notag
\end{align*}
Combining this with \eqref{e:der-b3} and \eqref{e:boundary-last},
we conclude
\begin{align*}
&\lim_{\delta\to 0}
	\int_{D_\delta}
		d\omega_{23}\wedge (c_1+ic_{23})
		+\omega_{23}\wedge ( d(c_1+ic_{23})
		-i \dot{\omega}_{23} )\\
=&\,  \lim_{\delta\to 0}
	\int_{D_\delta}
		( \phi_z^2-2\phi_{zz}+4\theta_z^2+4i\theta_{zz})
			\dot{f}_{\bar{z}} \ d^2z\\
&		 -\pi \sum_{z_k\in Z}
			\big(\,
				2\dot{f}_z
				+(\log\tilde{h}_k)\dot{}
				+(\log\tilde{h}_k)_z\dot{f}
				+(\phi-\log\tilde{h}_k)_z f_z \dot{z}_k
			\, \big)(z_k).
\end{align*}
Recalling that $\dot{f}_{\bar{z}}=\mu$
completes the proof.
\end{proof}

Note that we have the formulas
\begin{equation*}
\mathcal{S}(J^{-1}) =\phi_{zz} -\frac12 \phi_z^2, \qquad
\mathcal{S}(h_{\Phi})= \frac{h_{zz}}{h}-\frac32\frac{h_z^2}{h^2}
	=2\theta_z^2 +2i\theta_{zz}
\end{equation*}
where $\mathcal{S}$ denotes the Schwarzian derivative,
$J:H^2\to \Omega$ is the universal covering map of
$\Omega$,
(or $J:\mathbb{C}\to\Omega$ in the case of
genus 1), and $h_\Phi$ is a multi-valued function
such that $dh_\Phi=\Phi$.
By these formulas and Propositions \ref{p:var-CS1} and \ref{p:var-B0},
we have the following theorem.

\begin{theorem}\label{t:var-conCS} For $\varpi\in T^{1,0}U$
at $u_0\in U$, and the corresponding
$\mu\in \mathcal{H}^{-1,1}(\Omega,\Gamma)$,
\begin{align*}
\partial\, \overline{\mathbb{CS}}(\varpi)
=& -\frac{1}{2\pi^2} \lim_{\delta\to 0} \int_{D_\delta}
	\big(\mathcal{S}(J^{-1}) -\mathcal{S}(h_{\Phi})\big)
		\mu \, d^2z\\
& \quad
-\frac{1}{4\pi} \sum_{z_k\in Z}
	\big(\, 3\dot{f}_z
			+\frac32(\log\tilde{h}_k)\dot{}
			+\frac32(\log\tilde{h}_k)_z\dot{f}
			-\frac12(\log\tilde{h}_k)_z f_z \dot{z}_k\, \big)(z_k).
\end{align*}
\end{theorem}

\begin{corollary}\label{c:var-CS} For $\varpi\in T^{1,0}U$
at $u_0\in U$, and the corresponding
$\mu\in \mathcal{H}^{-1,1}(\Omega,\Gamma)$,
\begin{align*}
\partial\, CS (\varpi)
=& \ \frac{i}{4\pi^2}\lim_{\delta\to 0}
\int_{D_\delta}\mathcal{S}(h_\Phi)\mu \, d^2 z\\
&\quad
-\frac{i}{8\pi} \sum_{z_k\in Z}
	\big(\, 3\dot{f}_z
		+\frac32(\log\tilde{h}_k)\dot{}
		+\frac32(\log\tilde{h}_k)_z\dot{f}
		-\frac12(\log\tilde{h}_k)_z f_z \dot{z}_k\, \big)(z_k).
\end{align*}
\end{corollary}
\begin{proof}
This follows from directly from Theorem \ref{t:var-conCS}, since we have
\[
\mathbb{CS}(M_X,s_\Phi)=\frac{1}{\pi^{2}} W(M)+ 2iCS(M_X,s_\Phi),
\]
and, by \cite{KS}, \cite{ZT87}, it is known that
$\partial W
=\frac14 \phi_z^2-\frac12 \phi_{zz}
=-\frac{1}{2}\mathcal{S}(J^{-1})$.
\end{proof}

\section{Regularized Polyakov integral over $X$}\label{s:intrinsic}

In this section, we introduce an regularized integral defined
in terms of the metric $g_X$ and the holomorphic
$1$-form $\Phi$ over $X$.
We assume that $g_X$ is the
hyperbolic metric if the genus of $X$ is greater than 1,
and the flat metric of area 1 if the genus equals 1.
We assume that $\Phi$
has only simple zeroes and we denote by $Z$ its zero set.

Now we define
\begin{equation}\label{e:def-I}
\begin{split}
I(X,\Phi)
=\lim_{\delta\to 0} \big(\
	\int_{X_\delta} |\psi_z|^2\, d^2z
	+\frac{i}2\sum_{p_k\in Z}\int_{S_\delta(z_k)} &	
		\frac{(\phi-2\log |h|)(z)}{\bar{z}-\bar{z}_k}d\bar{z}
	\, \big)\\
&-\pi \sum_{p_k\in Z} (\phi-\log |\tilde{h}_k|)(z_k).
\end{split}
\end{equation}
Here $z$, in the integral around $p_k$, represents a local
coordinate near $p_k$, with $z_k=z(p_k)$.
The set $X_\delta$ denotes the complement of $\delta$-open
discs $|z-z_k|<\delta$ centered at each $z_k\in Z$ in $X$, and
$S_\delta(z_k)$ denotes a part of $\partial X_\delta$ which is
the $\delta$-circle centered at $z_k$ with the induced orientation
from $X_\delta$. Note that each of the terms in \eqref{e:def-I} are
independent of the choice of  local coordinates, by the transformation
laws given in subsection \ref{ss:adm-sing}. Note also that, since
circles are preserved under change of coordinates, the limit as
$\delta\to 0$ is independent of the choice of local coordinates.
Hence $I$ is a well-defined function on $\mathcal{H}_g(1,\ldots,1)$.

Suppose that $\varpi$ is a tangent vector at
$u_0\in\mathcal{H}_g(1,\ldots,1)$,
and that $U$ is a neighborhood of $u_0$.
We define a corresponding curve $u:W\to U$, for $W\subset \mathbb{C}$,
and a corresponding deformation map
$f(w,\cdot):X\to X_w$ for each $w\in W$,
in the same way as in subsection \ref{ss:variation}. We also define
the local coordinate expressions $h$, $\tilde{h}_k$
and $z_k$ in the same way as the discussion before the
equation \eqref{e:tilde-h}, except that here we do
not assume a global uniformization coordinate,
only local coordinates near the zeroes of $\Phi$.
For convenience we abbreviate $z_k(0)$ to $z_k$,
and $\dot{z}_k(0)$ to $\dot{z}_k$.

\begin{theorem}\label{t:intrin} For
$\varpi\in T^{1,0}\mathcal{H}_g(1,\ldots,1)$
at the point $(X,\Phi)$
and
the corresponding $\mu\in \mathcal{H}^{-1,1}(X)$,
\[
\begin{split}
\partial I (\varpi) = 2 \lim_{\delta\to 0}
&\int_{X_\delta}
	\big(\, \phi_{zz} -\frac12 \phi_z^2 -2\theta_z^2-2i\theta_{zz}\, \big)
	\mu\, d^2z \\
&+ \pi \sum_{p_k\in Z}
	\big( 3 \dot{f}_z
		+ \frac32(\log\tilde{h}_k)\,\dot{}
		+\frac32(\log \tilde{h}_k)_z \dot{f}
		-\frac12 (\log \tilde{h}_k)_z f_z \dot{z}_k \big) (z_k) .
\end{split}
\]
Here $\phi_{zz} -\frac12 \phi_z^2 -2\theta_z^2-2i\theta_{zz}$
is a meromorphic quadratic differential over $X$.
\end{theorem}

\begin{proof}
The domain $X_{w,\delta}$ is given by deleting the
$\delta$-discs centered at the $f(w,z_k(w))$ for $z_k\in Z$.
Its boundaries are given by the circles $S_\delta(f(w,z_k(w)))$.
Now we consider the pre-image domain, denoted by the same notation, of
$X_{w,\delta}$ by $f_w$ in $X$ which has boundaries denoted
by $B_\delta(z_k(w))$. Let us take $\delta_0$ such that
the $\delta_0$-disc centered at $z_k$ contains  $B_\delta(z_k(w))$
for each $z_k\in Z$, and take $w$ in an open neighborhood $W$
of the origin in $\mathbb{C}$. Then $X_{w,\delta}$ in $X$
decomposes into $X_{\delta_0}\cup A_{\delta_0,\delta}$.
Here $A_{\delta_0,\delta}=\cup_{z_k\in Z} A_{\delta_0,\delta}(z_k)$
where the region $A_{\delta_0,\delta}(z_k)$ has two boundaries
$S_{\delta_0}(z_k)$ and $B_\delta(z_k(w))$.

For the integral $|\psi_z|^2d^2z$ over $A_{\delta_0,\delta}$, we have
\begin{align*}
&\int_{A_{\delta_0,\delta(z_k)}} |\psi_z|^2 d^2z \\
=& \int_{A_{\delta_0,\delta(z_k)}} |(\phi-\log\tilde{h}_k)_z|^2\, d^2z -\int_{A_{\delta_0,\delta(z_k)}} \frac{(\phi-\log \bar{\tilde{h}}_k)_{\bar{z}}}{{z}-z_k(w)}\, d^2z
-\int_{A_{\delta_0,\delta}(z_k)} \frac{(\phi-\log {h})_{{z}}}{\bar{z}-\bar{z}_k(w)}\, d^2z \\
=& \int_{A_{\delta_0,\delta(z_k)}} |(\phi-\log\tilde{h}_k)_z|^2\, d^2z + \frac{i}{2}\int_{\partial A_{\delta_0,\delta(z_k)}} \frac{(\phi-\log \bar{\tilde{h}}_k)}{{z}-z_k(w)}\, d{z}
-\frac{i}{2}\int_{\partial A_{\delta_0,\delta}(z_k)} \frac{(\phi-2\log |h| )}{\bar{z}-\bar{z}_k(w)}\, d\bar{z} .
\end{align*}
Hence,
\begin{align}\label{e:each term}
&\int_{X_{w,\delta}} |\psi_z|^2\, d^2z  +\frac{i}2\sum_{p_k\in Z}\int_{B_\delta(z_k(w))} \frac{(\phi-2\log |h|)}{\bar{z}-\bar{z}_k(w)}d\bar{z}\notag \\
=&\int_{X_{\delta_0}} |\psi_z|^2\, d^2z +   \int_{A_{\delta_0,\delta}} |(\phi-\log\tilde{h}_k)_z|^2\, d^2z
 + \frac{i}{2}\int_{\partial A_{\delta_0,\delta}} \frac{(\phi-\log \bar{\tilde{h}}_k)}{{z}-{z_k(w)}}\, d{z}\\
 &+\frac{i}{2}\sum_{p_k\in Z}\int_{S_{\delta_0}(z_k)} \frac{(\phi-2\log| h| )}{\bar{z}-\bar{z}_k(w)}\, d\bar{z} \notag
\end{align}
where $B_\delta(z_k(w))$ and $S_{\delta_0}(z_k)$
have the orientation induced from $A_{\delta_0,\delta}(z_k)$
and $X_{\delta_0}$ respectively.

Now, we consider the holomorphic variation of each of the
terms on the right hand side of \eqref{e:each term}.
First, we deal with the term
$I_{\delta_0}=\int_{X_{\delta_0}} |\psi_z|^2\, d^2z$.
For this, observe that
\begin{align*}
\delta_\mu \big( \psi_z dz \big) =& \ \big( \dot{\psi}_z + \psi_{zz} \dot{f} \big) dz + \psi_z (\dot{f}_z dz + \dot{f}_{\bar{z}} d\bar{z}),\\
\delta_\mu \big( \bar{\psi}_{\bar{z}} d\bar{z} \big) =& \ \big( \dot{\bar{\psi}}_{\bar{z}} +  \bar{\psi}_{\bar{z}z} \dot{f} \big) d\bar{z}
=\big( \dot{\phi}_{\bar{z}} +  \phi_{\bar{z}z} \dot{f} \big) d\bar{z}.
\end{align*}
Here, $\delta_\mu$ denotes the Lie derivative. See Section 2.3 of \cite{MT} for details.
Combining these facts with \eqref{e:phi-var} and Lemma \ref{l:inv}, we have
\begin{align}\label{e:var1}
\begin{split}
\partial I_{\delta_0} (\varpi)
=\ & -\frac{i}{2} \int_{X_{\delta_0}} \psi_z (\phi_z \dot{f}_{\bar{z}} +\dot{f}_{z\bar{z}})\ dz\wedge d\bar{z}\\
           &  -\frac{i}{2} \int_{X_{\delta_0}} \bar{\psi}_{\bar{z}} (\dot{f}_{zz} + (2i\theta)_z\dot{} +((2i\theta)_z \dot{f})_z)\ dz\wedge d\bar{z}.
\end{split}
\end{align}
Let us denote the two terms on the right hand side of
\eqref{e:var1} by $(\partial I_{\delta_0} (\varpi))_i$ for $i=1,2$.
Recalling that $\psi_z \dot{f}_{\bar{z}} d\bar{z}$ is an invariant $(0,1)$-form, we have
\begin{align*}
(\partial I_{\delta_0} (\varpi))_1=-\frac{i}{2}   \big(\ \int_{X_{\delta_0}} \psi_z \phi_z \dot{f}_{\bar{z}}\ dz\wedge d\bar{z}
-\int_{X_{\delta_0}} \psi_{zz} \dot{f}_{\bar{z}}\ dz\wedge d\bar{z} + \int_{\partial X_{\delta_0}} \psi_z \dot{f}_{\bar{z}}\ d\bar{z} \ \big)
\end{align*}
where $\partial X_{\delta_0}$ has the induced orientation from $X_{\delta_0}$. For $(\partial I_{\delta_0} (\varpi))_2$,
by Lemma \ref{l:inv} and \eqref{e:theta-pole}, \eqref{e:theta-trivial},
\begin{align*}
(\partial I_{\delta_0}(\varpi))_2= \frac{i}{2}    \big(\ &\int_{X_{\delta_0}} \psi_{z\bar{z}} ( \dot{f}_{z} + (2i\theta)\,\dot{} +(2i\theta)_z \dot{f} ) \ dz\wedge d\bar{z}\\
 & \qquad - \int_{\partial X_{\delta_0}}  \bar{\psi}_{\bar{z}}  ( \dot{f}_{z} + (2i\theta)\,\dot{} +(2i\theta)_z \dot{f} ) \ d\bar{z} \ \big)\\
= -\frac{i}{2}   & \big(\  \int_{X_{\delta_0}} \psi_{z} ( \dot{f}_{z\bar{z}} +(2i\theta)_z \dot{f}_{\bar{z}} \ ) \ dz\wedge d\bar{z}\\
& \qquad + \int_{\partial X_{\delta_0}} ( \dot{f}_{z} + (2i\theta)\,\dot{} +(2i\theta)_z \dot{f} ) \ \big( \psi_{z} dz + \bar{\psi}_{\bar{z}}  d\bar{z} \big) \ \big).
\end{align*}
Dealing with  the term $\psi_{z}\dot{f}_{z\bar{z}}$ as before,
\begin{align*}
(\partial I_{\delta_0} (\varpi))_2=-\frac{i}{2}    \big(\   & \int_{X_{\delta_0}} \psi_{z} (2i\theta)_z \dot{f}_{\bar{z}}  \ dz\wedge d\bar{z}
 - \int_{X_{\delta_0}} \psi_{zz} \dot{f}_{\bar{z}}\ dz\wedge d\bar{z}\\
 & + \int_{\partial X_{\delta_0}} \psi_z \dot{f}_{\bar{z}}\ d\bar{z}
+ \int_{\partial X_{\delta_0}} ( \dot{f}_{z} + (2i\theta)\,\dot{} +(2i\theta)_z \dot{f} ) \ \big( \psi_{z} dz + \bar{\psi}_{\bar{z}}  d\bar{z} \big) \ \big).
\end{align*}
Combining computations for $(\partial I_{\delta_0}(\varpi))_1$ and $(\partial I_{\delta_0}(\varpi))_2$, we get
\begin{align}\label{e:delta0-total}
\begin{split}
\partial I_{\delta_0} (\varpi)=-\frac{i}{2}    \big(\  &\int_{X_{\delta_0}} {\psi}_z \bar{\psi}_{{z}} \dot{f}_{\bar{z}}\ dz\wedge d\bar{z}
    - 2\int_{X_{\delta_0}} \psi_{zz} \dot{f}_{\bar{z}}\ dz\wedge d\bar{z} \\
     & +2 \int_{\partial X_{\delta_0}} \psi_z \dot{f}_{\bar{z}}\ d\bar{z}
    + \int_{\partial X_{\delta_0}} ( \dot{f}_{z} + (2i\theta)\,\dot{} +(2i\theta)_z \dot{f} ) \ \big( \psi_{z} dz + \bar{\psi}_{\bar{z}}  d\bar{z} \big) \  \big).
\end{split}
\end{align}
Now let us deal with the integrals over $\partial X_{\delta_0}$. First, by \eqref{e:boundary-last} we have
\begin{equation}\label{e:bound-delta0-first}
\int_{\partial X_{\delta_0}} \psi_z \dot{f}_{\bar{z}}\ d\bar{z}=\int_{\partial X_{\delta_0}} (\phi-2i\theta)_z \dot{f}_{\bar{z}}\ d\bar{z} = O(\delta_0) .
\end{equation}
For the other boundary integral given in the last line of
\eqref{e:delta0-total}, using \eqref{e:tilde-n-b1}, near
$p_k\in Z$  we have
\begin{align*}
&-( \dot{f}_{z} + (2i\theta)\,\dot{}+(2i\theta)_z \dot{f} ) \ \big( (\phi-2i\theta)_{z} dz + (\phi+2i\theta)_{\bar{z}}  d\bar{z} \big) \\
=&(\, (\phi-2i\theta)\circ f \, )\,\dot{} \ \big( (\phi-2i\theta)_{z} dz + (\phi+2i\theta)_{\bar{z}}  d\bar{z} \big) \\
=&\big(\, - \frac{\dot{f}(z) -\dot{f}(z_k) -f_z(z_k)\dot{z}_k}{z-z_k}+ ( (\phi-\log \tilde{h}_k)\circ f)\,\dot{}  \, \big)\\
&\cdot \, \big(\, (-\frac{1}{z-z_k}+(\phi-\log\tilde{h}_k)_z) dz +(-\frac{1}{\bar{z}-\bar{z}_k}+(\phi-\log\bar{\tilde{h}}_k)_{\bar{z}}) d\bar{z} \, \big) .
\end{align*}
Using this and some computation as before, we obtain
\begin{align}\label{e:bound-delta0}
\begin{split}
&-\frac{i}{2}\int_{\partial X_{\delta_0}} ( \dot{f}_{z} + (2i\theta)\,\dot{} +(2i\theta)_z \dot{f} ) \ \big( (\phi-2i\theta)_{z} dz +  (\phi+2i\theta)_{\bar{z}}  d\bar{z} \big)  \\
=& -\frac{i}{2}\sum_{p_k\in Z}\Big(\, \int_{|z-z_k|=\delta_0}\big(\, - \frac{\dot{f}(z) -\dot{f}(z_k) -f_z(z_k)\dot{z}_k}{z-z_k} \, \big) \big(\, -\frac{1}{z-z_k} dz -\frac{1}{\bar{z}-\bar{z}_k} d\bar{z} \, \big)\\
 & \qquad\qquad\qquad +\big(\, - \frac{\dot{f}(z) -\dot{f}(z_k) -f_z(z_k)\dot{z}_k}{z-z_k} \, \big) \big(\, (\phi-\log\tilde{h}_k)_z dz +(\phi-\log\bar{\tilde{h}}_k)_{\bar{z}} d\bar{z} \, \big)  \\
 &\qquad\qquad \qquad + ((\phi-\log \tilde{h}_k)\circ f)\dot{}(z)\, \big(\, -\frac{1}{z-z_k} dz -\frac{1}{\bar{z}-\bar{z}_k} d\bar{z} \, \big) \, \Big) +O(\delta_0) \\
= &\ \pi \sum_{p_k\in Z} (\phi-\log \tilde{h}_k)_z(z_k) f_z(z_k) \dot{z}_k  +O(\delta_0).
\end{split}
\end{align}
By \eqref{e:delta0-total}, \eqref{e:bound-delta0-first} and \eqref{e:bound-delta0},
\begin{align}\label{e:first term}
\begin{split}
&\partial \Big(\, \int_{X_{\delta_0}} |\psi_z|^2 \, d^2z \, \Big) (\varpi)\\
 =& \int_{X_{\delta_0}} (  {2}\phi_{zz}-\phi_z^2  -4\theta_z^2-4i\theta_{zz}) \dot{f}_{\bar{z}}\ d^2z + \pi \sum_{p_k\in Z} (\phi-\log \tilde{h}_k)_z(z_k) f_z(z_k) \dot{z}_k  +O(\delta_0).
\end{split}
\end{align}

The holomorphic variation of the second term  $\int_{A_{\delta_0,\delta}} |(\phi-\log\tilde{h}_k)_z|^2\, d^2z$ on the right hand side of \eqref{e:each term} can be analyzed as above, but the integrand
$|(\phi-\log\tilde{h}_k)_z|^2$ is regular over $A_{\delta_0,\delta}$ for any $\delta>0$. Hence, we can see that
\begin{align}\label{e:second term}
\partial \Big(\, \lim_{\delta\to 0} \int_{A_{\delta_0,\delta}} |(\phi-\log\tilde{h}_k)_z|^2\, d^2z \, \Big) (\varpi) = O(\delta_0).
\end{align}

The limit of the third term on the right hand side of \eqref{e:each term} as $\delta\to 0$ is given by
\begin{align}\label{e:third-1}
\sum_{p_k\in Z} \frac{i}{2}\int_{|z-z_k|=\delta_0} \frac{(\phi-\log \bar{\tilde{h}}_k)}{{z}-{z_k(w)}}\, d{z} +\pi (\phi_w-\log \bar{\tilde{h}}_{k,w})(f(w,z_k(w)))
\end{align}
where $\phi_w$, $\tilde{h}_{k,w}$ denote (local) functions over $X_w$.
For the holomorphic variation of the first term in \eqref{e:third-1}, we have
\begin{align*}
&\partial \Big(\,  \frac{i}{2}\int_{|z-z_k|=\delta_0} \frac{(\phi-\log \bar{\tilde{h}}_k)}{{z}-{z_k(w)}}\, d{z} \, \Big)(\varpi)\\
=&\frac{i}{2}\int_{|z-z_k|=\delta_0} -\frac{\dot{f}(z)-\dot{f}(z_k)-f_z(z_k)\dot{z_k}}{(z-z_k)^2} (\phi-\log\bar{\tilde{h}}_k) \, dz\\
& \qquad \qquad \quad + \frac{ \dot{\phi}+\phi_z\dot{f}}{z-z_k}\, dz + \frac{\phi-\log\bar{\tilde{h}}_k}{z-z_k} (\dot{f}_z dz+\dot{f}_{\bar{z}} d\bar{z})\\
=& -\pi (\dot{\phi}+\phi_z \dot{f}+\phi_zf_z\dot{z}_k)(z_k) +O(\delta_0).
\end{align*}
For the second term in \eqref{e:third-1}, we have
\[\partial \Big(\pi (\phi_w-\log \bar{\tilde{h}}_{k,w})
	\big(f\big(w,z_k(w)\big)\big)\Big)(\varpi)
= \pi (\dot{\phi}+\phi_z \dot{f}+\phi_zf_z\dot{z}_k)(z_k).
\]
Hence,
\begin{align}\label{e:third term}
\partial \Big(\, \lim_{\delta\to 0}\frac{i}{2}\int_{\partial A_{\delta_0,\delta}} \frac{(\phi-\log \bar{\tilde{h}}_k)}{{z}-{z_k(w)}}\, d{z} \, \Big) (\varpi)= O(\delta_0).
\end{align}

In a similar way, we can show the following equality for the fourth term   on the right hand side of \eqref{e:each term},
\begin{align}\label{e:fourth term}
\begin{split}
&\partial \Big(\,  \frac{i}{2}\sum_{p_k\in Z}\int_{S_{\delta_0}(z_k)} \frac{(\phi-2\log| h| )}{\bar{z}-\bar{z}_k(w)}\, d\bar{z} \, \Big) (\varpi)\\
=& \pi \sum_{p_k\in Z} (\dot{f}_z-(\phi-\log\tilde{h}_k)\,\dot{} -(\phi-\log \tilde{h}_k)_z \dot{f})(z_k) +O(\delta_0).
\end{split}
\end{align}

 Combining the equalities \eqref{e:first term}, \eqref{e:second term}, \eqref{e:third term}, and \eqref{e:fourth term}, we have
\begin{align}\label{e:total-integral}
&\partial \Big(\, \lim_{\delta\to 0} \big(\ \int_{X_{w,\delta}} |\psi_z|^2\, d^2z  +\frac{i}2\sum_{p_k\in Z}\int_{B_\delta(z_k(w))} \frac{(\phi-2\log |h|)(z)}{\bar{z}-\bar{z}_k(w)}d\bar{z}\, \big) \, \Big) (\varpi)\notag\\
=& \ \lim_{\delta_0\to 0}  \int_{X_{\delta_0}} (  {2}\phi_{zz}-\phi_z^2  -4\theta_z^2-4i\theta_{zz}) \dot{f}_{\bar{z}}\ d^2z  \\
&+ \pi \sum_{p_k\in Z} (\dot{f}_z-(\phi-\log\tilde{h}_k)\,\dot{} -(\phi-\log \tilde{h}_k)_z \dot{f} + (\phi-\log \tilde{h}_k)_zf_z \dot{z}_k)(z_k).\notag
\end{align}
Finally combining  \eqref{e:each term}, \eqref{e:total-integral} and the following equality
\begin{align*}
\partial \Big(\, -\pi \sum_{p_k\in Z} &(\phi-\log |\tilde{h}_k|)(z_k) \, \Big) (\varpi)\\
&= -\pi \sum_{p_k\in Z} \big(\, (\phi-\frac12 \log\tilde{h}_k)\,\dot{}
	+(\phi-\frac12 \log\tilde{h}_k)_z (\dot{f} +f_z \dot{z}_k)\, \big)(z_k)
\end{align*}
completes the proof.

\end{proof}

\section{Holomorphic variation of $\tau_B$}

In this section, we prove the following theorem.

\begin{theorem}\label{t:var-log T} For
$\varpi\in T^{1,0}\tilde{\mathcal{H}}_g(1,\ldots,1)$
at a point corresponding to a marked Riemann surface $X$
and a holomorphic 1-form $\Phi$ on $X$,
and the corresponding $\mu\in \mathcal{H}^{-1,1}(X)$, we have
\begin{equation}\label{e:var-log T}
\begin{split}
\partial  \log \tau_B^{24} (\varpi) = &\frac{4}{\pi}\lim_{\delta\to 0} \int_{X_\delta} (R_B-R_\Phi) \, \mu \, d^2z\\
&\ \ + \sum_{p_k\in Z}\big( 6 \dot{f}_z+ 3(\log\tilde{h}_k)\,\dot{}+3(\log \tilde{h}_k)_z \dot{f} - (\log \tilde{h}_k)_z f_z \dot{z}_k \big) (z_k).
\end{split}
\end{equation}
\end{theorem}

\begin{proof}
By the chain rule, first we have
\begin{equation*}
\partial \log \tau_B (\varpi)
=\sum_{1\leq i \leq g, 2\leq k \leq 2g-2} \Big(\, \frac{\partial  \log \tau_B}{\partial A_i} \frac{\partial A_i}{\partial \mu} +
\frac{\partial \log \tau_B}{\partial B_i}  \frac{\partial B_i}{\partial \mu}
+ \frac{\partial \log \tau_B}{\partial Z_k}  \frac{\partial Z_k}{\partial \mu} \, \Big),
\end{equation*}
where $A_i$, $B_i$, $Z_k$ are the coordinates
on $\tilde{\mathcal{H}}_g(1,\ldots,1)$ given in equation \eqref{e:coordinates}.
The holomorphic variation of the coordinates $A_i$ is given by
\begin{align*}
\frac{\partial A_i}{\partial \mu} =& \int_{a_i} (\dot{h} +h_z\dot{f}+h\dot{f}_z) \,dz + h\dot{f}_{\bar{z}}\, d\bar{z},
\end{align*}
and similar equalities hold for $B_i$, $Z_k$.  Combining these and the defining equations of $\tau_B$ in \eqref{e:def-B-tau-0},
\begin{align}\label{e:left-hand-side-Tau}
\begin{split}
 &\partial \log \tau_B^{24} (\varpi)\\
=& \frac{2i}{\pi} \Big(\, -\sum_{i=1}^g\int_{b_i} \frac{(R_B-R_\Phi)}{h}\, dz \cdot \big(\,  \int_{a_i} (\dot{h} +h_z\dot{f}+h\dot{f}_z) \,dz + h\dot{f}_{\bar{z}}\, d\bar{z}\,\big) \\
&\ \ \quad  +\sum_{i=1}^g  \int_{a_i} \frac{(R_B-R_\Phi)}{h}\, dz \cdot \big(\,  \int_{b_i} (\dot{h} +h_z\dot{f}+h\dot{f}_z) \,dz + h\dot{f}_{\bar{z}}\, d\bar{z}\,\big) \\
&\ \ \quad  +\sum_{k=1}^{2g-2} \int_{|z-z_k|=\delta} \frac{(R_B-R_\Phi)}{h}\, dz \cdot \big(\,  \int^{z_k}_{z_1} (\dot{h} +h_z\dot{f}+h\dot{f}_z) \,dz + h\dot{f}_{\bar{z}}\, d\bar{z}\,\big) \, \Big).
\end{split}
\end{align}

On the other hand, for the integral on the right hand side of \eqref{e:var-log T} we have
\begin{align*}
&\frac{4}{\pi}\int_{X_\delta} (R_B-R_\Phi)\, \mu \, d^2z\\
=&\frac{2i}{\pi}\int_{X_\delta} \frac{(R_B-R_\Phi)}{h}\,dz \wedge \big(\, (\dot{h} +h_z\dot{f}+h\dot{f}_z) \,dz +  (h \dot{f})_{\bar{z}}\, d\bar{z} \, \big)\\
=&-\frac{2i}{\pi}\int_{X_\delta} d \Big(\, \int^z_{z_1} (\dot{h} +h_z\dot{f}+h\dot{f}_z) \,dz +  (h \dot{f})_{\bar{z}}\, d\bar{z} \, \cdot \,  \frac{(R_B-R_\Phi)}{h}\, dz  \, \Big).
\end{align*}
Here  $(\dot{h} +h_z\dot{f}+h\dot{f}_z) \,dz +  (h \dot{f})_{\bar{z}}\, d\bar{z}$ is a globally well-defined $1$-form so that its line integral defines a well-defined function.
As in the proof of the Riemann's bilinear relation to the last line in the above equalities, we have
\begin{align}\label{e:right-hand-side-Tau}
\begin{split}
& \frac{4}{\pi}\int_{X_\delta} (R_B-R_\Phi)\, \mu \, d^2z  \\
=& \frac{2i}{\pi} \Big(\, -\sum_{i=1}^g \int_{b_i} \frac{(R_B-R_\Phi)}{h}\, dz \cdot \big(\,  \int_{a_i} (\dot{h} +h_z\dot{f}+h\dot{f}_z) \,dz + h\dot{f}_{\bar{z}}\, d\bar{z}\,\big) \\
&\ \ \quad  + \sum_{i=1}^g \int_{a_i} \frac{(R_B-R_\Phi)}{h}\, dz \cdot \big(\,  \int_{b_i} (\dot{h} +h_z\dot{f}+h\dot{f}_z) \,dz + h\dot{f}_{\bar{z}}\, d\bar{z}\,\big) \\
&\  \ \quad  + \sum_{k=1}^{2g-2} \int_{|z-z_k|=\delta}  \big(\,  \int^{z}_{z_1} (\dot{h} +h_z\dot{f}+h\dot{f}_z) \,dz + h\dot{f}_{\bar{z}}\, d\bar{z}\,\big)\, \frac{(R_B-R_\Phi)}{h}\, dz  \, \Big).
\end{split}
\end{align}
Near $z_k\in Z$ where we have $\Phi(z)=(z-z_k)\tilde{h}_k dz$, $\frac{R_\Phi}{h}$ has the following expression
\begin{align*}
\frac{R_\Phi}{h}(z) = -\frac32 \frac{1}{\tilde{h}(z_k)} \frac{1}{(z-z_k)^3}+\frac12\frac{\tilde{h}_{ z}(z_k)}{\tilde{h}^2(z_k)} \frac{1}{(z-z_k)^2}+ \Big(3\frac{\tilde{h}_{zz}(z_k)}{\tilde{h}^2(z_k)}-\frac{\tilde{h}_{z}^2(z_k)}{\tilde{h}^3(z_k)}\Big)\frac{1}{z-z_k}+ \cdots
\end{align*}
where $\tilde{h}=\tilde{h}_k$.
Now comparing \eqref{e:left-hand-side-Tau} with \eqref{e:right-hand-side-Tau}, in order to complete the proof, it is sufficient to show that
\begin{align}\label{e:remain-to-prove}
\begin{split}
\lim_{\delta\to 0} \, \frac{2i}{\pi}\, \sum_{p_k\in Z} \int_{|z-z_k|=\delta}  &\big(\,  \int^{z}_{z_1} (\dot{h} +h_z\dot{f}+h\dot{f}_z) \,dz + h\dot{f}_{\bar{z}}\, d\bar{z}\,\big)\,
\big(\frac{R_\Phi}{h}\big)_{s} \, dz \\
=& \sum_{p_k\in Z}\big( 6 \dot{f}_z+ 3(\log\tilde{h}_k)\,\dot{}+3(\log \tilde{h}_k)_z \dot{f} - (\log \tilde{h}_k)_z f_z \dot{z}_k \big) (z_k)
\end{split}
\end{align}
where $(\frac{R_\Phi}{h})_{s}:=\big(-\frac32 \frac{1}{\tilde{h}(z_k)} \frac{1}{(z-z_k)^3}+\frac12\frac{\tilde{h}_z(z_k)}{\tilde{h}^2(z_k)} \frac{1}{(z-z_k)^2}\big)$ with $\tilde{h}=\tilde{h}_k$ near $z_k$.
By some elementary computations, we have
\begin{multline*}
 \int_{|z-z_k|=\delta}  \big(\,  \int^{z}_{z_1} (\dot{h} +h_z\dot{f}+h\dot{f}_z) \,dz + h\dot{f}_{\bar{z}}\, d\bar{z}\,\big)\, \frac{1}{(z-z_k)^2}\, dz\\
=2\pi i (\dot{h} +h_z\dot{f}+ h\dot{f}_z)(z_k)+O(\delta) = -2\pi i (\tilde{h}_k f_z)(z_k) \dot{z}_k +O(\delta),
\end{multline*}
\begin{multline*}
 \int_{|z-z_k|=\delta}  \big(\,  \int^{z}_{z_1} (\dot{h} +h_z\dot{f}+h\dot{f}_z) \,dz + h\dot{f}_{\bar{z}}\, d\bar{z}\,\big)\, \frac{1}{(z-z_k)^3}\, dz\\
= \pi i (\dot{h}_z +h_{zz}\dot{f}+2h_z\dot{f}_z+ h\dot{f}_{zz})(z_k) +O(\delta)
 =\pi i \big(2\dot{f}_z \tilde{h}_k +\dot{\tilde{h}}_k+{\tilde{h}}_{kz} (\dot{f}-f_z \dot{z}_k)\big)(z_k) +O(\delta).
\end{multline*}
The equality \eqref{e:remain-to-prove} follows from these and this completes the proof.
\end{proof}

\section{Proof of Theorem \ref{t:main theorem-intro} for
	$\tilde{\mathcal{H}}_g(1,\ldots,1)$}\label{s:proof-main-theorem}

In this section we collect the formulae proved in the previous sections to prove Theorem \ref{t:main theorem-intro} for  $\tilde{\mathcal{H}}_g(1,\ldots,1)$.
For this, first we recall a property of
the Schwarzian derivative:
\begin{equation}\label{e:comp-Sch}
\mathcal{S}(h_1\circ h_2) =\mathcal{S}(h_1)\circ h_2 (h_2')^2 +\mathcal{S}(h_2).
\end{equation}
Let  $\pi_F:H^2\to X$ and $\pi_S:\Omega\to X$ denote the Fuchsian and Schottky uniformization maps respectively. Then the universal covering map
$J:H^2\to\Omega$ satisfies $\pi_F=\pi_S\circ J$.
(In case $g=1$, replace $H^2$ with $\mathbb{C}$.)
Applying this to the composition of multi-valued functions
$J^{-1}= \pi_F^{-1}\circ \pi_S$,
we obtain
\[
\mathcal{S}(J^{-1})=\mathcal{S}(\pi_F^{-1})\circ\pi_S (\pi_S')^2 -\mathcal{S}(\pi_S^{-1})\circ \pi_S (\pi_S')^2.
\]
Similarly applying \eqref{e:comp-Sch} to the composition of multi-valued functions $h_{\Phi} = (\int^z \Phi) \circ \pi_S$ for a local coordinate $z$ over $X$,
we obtain
\[
\mathcal{S}(h_\Phi) = \mathcal{S}(\int^z\Phi)\circ\pi_S (\pi_S')^2 -\mathcal{S}(\pi_S^{-1})\circ \pi_S (\pi_S')^2.
\]
Let us recall that $\mathcal{S}(\pi_F^{-1})$, $\mathcal{S}(\pi_S^{-1})$, $\mathcal{S}(\int^z\Phi)$ define the projective connections $R_F$, $R_S$, $R_\Phi$ over $X$ respectively.

Given a point $u_0\in U\subset\tilde{\mathcal{H}}^*_g(1,\ldots,1)$,
with $U$ a contractible open set,
and a tangent vector $\varpi\in T^{1,0}U$ at $u_0$,
we have a corresponding $\mu\in\mathcal{H}^{-1,1}(X)$,
family of deformations $f_{w\mu}$ and
holomorphic family of holomorphic 1-forms $\Phi(w)$. For
this family,
by Theorem \ref{t:var-conCS} and Corollary \ref{c:var-CS}, we have
\begin{align}\label{e:var-conCS-proj}
\begin{split}
\partial (4\pi \overline{\mathbb{CS}})(\varpi) =& -\frac{2}{\pi} \lim_{\delta\to 0} \int_{X_\delta} \big(\, (R_F-R_S) - (R_\Phi-R_S)\, \big) \, \mu\, d^2z \\
&-  \sum_{p_k\in Z}\big( 3 \dot{f}_z+ \frac32(\log\tilde{h}_k)\,\dot{}+\frac32(\log \tilde{h}_k)_z \dot{f} -\frac12 (\log \tilde{h}_k)_z f_z \dot{z}_k \big) (z_k) .
\end{split}
\end{align}
\begin{align}\label{e:var-CS-proj}
\begin{split}
\partial (4\pi \mathbb{CS})(\varpi) =& -\frac{2}{\pi} \lim_{\delta\to 0} \int_{X_\delta} \big(\, (R_F-R_S) + (R_\Phi-R_S)\, \big) \, \mu\, d^2z \\
&+  \sum_{p_k\in Z}\big( 3 \dot{f}_z+ \frac32(\log\tilde{h}_k)\,\dot{}+\frac32(\log \tilde{h}_k)_z \dot{f} -\frac12 (\log \tilde{h}_k)_z f_z \dot{z}_k \big) (z_k) .
\end{split}
\end{align}
By Theorem \ref{t:intrin}, we also have
\begin{align}\label{e:var-I-proj}
\begin{split}
\partial (\frac{1}{\pi}I) (\varpi) =&\  \frac{2}{\pi} \lim_{\delta\to 0}
\int_{X_\delta} ( R_F-R_\Phi)\, \mu\, d^2z \\
&+ \sum_{p_k\in Z}\big( 3 \dot{f}_z+ \frac32(\log\tilde{h}_k)\,\dot{}+\frac32(\log \tilde{h}_k)_z \dot{f} -\frac12 (\log \tilde{h}_k)_z f_z \dot{z}_k \big) (z_k),
\end{split}
\end{align}
where we have lifted the function $I$ from
${\mathcal{H}}_g(1,\ldots,1)$ to $U\subset\tilde{\mathcal{H}}^*_g(1,\ldots,1)$.
From \eqref{e:var-conCS-proj} and \eqref{e:var-I-proj}, it follows that $\exp(4\pi \mathbb{CS}+ \frac{1}{\pi}I)$ is a holomorphic function over $U$.
It is known that
\begin{align}\label{e:var-F-proj}
\partial (\log F^{24}) (\varpi)=\partial (\log F^{24}) (\mu) = \frac{4}{\pi} \int_X (R_B-R_S)\, \mu\, d^2z
\end{align}
from \cite{MT}, \cite{Z}. Here we also have lifted the function $F$ from
$\mathfrak{T}_g$ to $U\subset\tilde{\mathcal{H}}^*_g(1,\ldots,1)$.
Combining \eqref{e:var-CS-proj}, \eqref{e:var-I-proj}, and \eqref{e:var-F-proj}, the holomorphic variation of the holomorphic function
$\exp(4\pi \mathbb{CS}+ \frac{1}{\pi}I)\, F^{24}$ is given by
\begin{equation}\label{e:var-right-proj}
\begin{split}
 & \partial \log \big(\,  \exp(4\pi \mathbb{CS}+ \frac{1}{\pi}I)\, F^{24} \, \big)(\varpi) \\
= &\frac{4}{\pi} \lim_{\delta\to 0}
\int_{D_\delta} (R_B-R_\Phi)\, \mu\, d^2z \\
 &+ \sum_{p_k\in Z}\big( 6 \dot{f}_z+ 3(\log\tilde{h}_k)\,\dot{}+ 3(\log \tilde{h}_k)_z \dot{f} - (\log \tilde{h}_k)_z f_z \dot{z}_k \big) (z_k) .
 \end{split}
\end{equation}
By Theorem \ref{t:var-log T} and equation~\eqref{e:var-right-proj},
the two  functions $\tau_B^{24}$ (lifted to $U$)
and
$\exp(4\pi \mathbb{CS}+ \frac{1}{\pi}I)F^{24}$
have the same holomorphic variation for any holomorphic tangent
vector $\varpi$.
Consequently, the liftings of $\tau_B^{24}$ and
$\exp(4\pi \mathbb{CS}+ \frac{1}{\pi}I)F^{24}$
 to any connected component of $\tilde{\mathcal{H}}^*_g(1,\ldots,1)$
are equal up to a multiplicative constant. But the holomorphic function
$\tau_B^{24}$ descends
to $\tilde{\mathcal{H}}_g(1,\ldots,1)$, hence the holomorphic function
$\exp(4\pi \mathbb{CS}+ \frac{1}{\pi}I)F^{24}$
descends too.
This proves Theorem \ref{t:main theorem-intro}. The constant $c$
appearing in the theorem depends on our choice of
a connected component of $\tilde{\mathcal{H}}^*_g(1,\ldots,1)$.

\section{Proof of Theorem \ref{t:main theorem-intro} for
	$\tilde{H}_{g,n}(1,\ldots,1)$}
	\label{s:last section}

A point in $\tilde{H}_{g,n}(1,\ldots,1)$ corresponds to an
equivalence class of a compact Riemann
surface $X$ of genus $g$, together with a meromorphic function
$\lambda:X\to\mathbb{CP}^1$. The differential $d\lambda$ is
meromorphic, with $m$ simple zeros at the ramification points
$p_1,\ldots p_m\in X$ of $\lambda$, and with $n$
double poles at the preimages of infinity
$q_1,\ldots, q_n$. It has residue $0$ at each pole.
The proof of  Theorem \ref{t:main theorem-intro} given above
applies, with $d\lambda$ playing the
role of $\Phi$, with some modifications due to the
poles of $d\lambda$. We outline these modifications in this section.
Note that we will carry over all definitions and notations
from before, with any changes being noted below.

\subsection{Construction of framing}

Recall that $Z$ is the zero set of $d\lambda$ on $X$; denote also
by $P$ the set of poles of $d\lambda$ on $X$.
The framing on $X\setminus (Z\cup P)$ is constructed from $d\lambda$ as before;
it will now have a singularity of index $2$ at each point in $P$. Let
$z_{j\alpha}$ denote the co-ordinate of a pole of $d\lambda$ in a patch
$U_\alpha$. Then $h_\alpha=(z_\alpha-z_{j\alpha})^{-2}\tilde{h}_{j\alpha}$,
where $\tilde{h}_{j\alpha}$ is holomorphic and non-zero on $U_\alpha$.
Note that
\begin{equation}\label{e:h-tilde-id}
(\log\tilde{h}_{j\alpha})_z(z_\alpha(q_j))=0 \quad \text{for} \quad q_j\in P.
\end{equation}
As before, we define $\tilde{\theta}$ by
$e^{i\tilde{\theta}_\alpha}:=\tilde{h}_\alpha/\vert{\tilde{h}_\alpha}\vert$.
We now define the co-framing $(\tilde{\omega}_2,\tilde{\omega}_3)$
and framing $(\tilde{f}_2,\tilde{f}_3)$
as before, but
by means of $e^{\frac{1}{2}\phi_\alpha - i\tilde{\theta}_\alpha}dz_\alpha$
rather than $e^{\frac{1}{2}(\phi_\alpha + i\tilde{\theta}_\alpha)}dz_\alpha$.

For each point $q_j\in P$, the admissible extension of the framing on
$X$ will have an additional singular
curve with both endpoints at $q_j$. We denote the set of these
curves by $\L^2_P$. (Note that in this notation,
$\L^2$ and $\L^2_P$ are disjoint.)
We require the reference framing
$\kappa$ on each component
of this curve to satisfy $r^{-1}\kappa\to (\tilde{f}_2,\tilde{f}_3)$
at the outgoing endpoint, and $r^{-1}\kappa\to (\tilde{f}_2,-\tilde{f}_3)$
at the incoming endpoint, as $r\to0$ (identifying framings on
$\partial \overline{M}$ and $X$
as in subsection \ref{ss:existence-admissible}).
The proof of the existence of admissible extensions
(Theorem \ref{t:framing-ext}) must be modified as follows:
for the singular curve with endpoints at a pole $q_j$,
choose a small neighborhood of its intersection
with the level surface $X^{a_1}$. Take the subset of
$N_{[0,a_1)}$ consisting of all geodesics connecting
points of this neighborhood to $q_j$. The framing
on $N_{[0,a_1)}$ will be defined by parallel translation
outside of this neighborhood, as before. Inside the neighborhood,
we pick any framing which matches smoothly on the boundary,
which has $e_1$ orthogonal to the level surface,
and which has an index 1 singularity at each of the two
components of the singular curve. The framing is then extended
to the rest of $M$ as before. The resulting framing is an
admissible extension of the boundary framing given by $d\lambda$.

\subsection{Variation of the invariant $\mathbb{CS}$}

The fact that the singularity of the framing around $\L^2_P$ is
index 1 means that the corresponding boundary contributions computed
in Section \ref{s:chern-simons} appear with opposite sign to those from the
components of $\L^2$ with endpoints at the zeroes of $d\lambda$.
Since the contribution from $\L^2_P$ in the definition \ref{e:def-f-epsilon}
also appears with opposite sign, the results of Section \ref{s:chern-simons}
hold without change.

In the remaining sections, we have the following changes.
We have the following new contribution from $\L^2_P$ in
Proposition \ref{p:hol-var},
\[
- \frac{1}{2\pi}\sum_{y\in\partial \L^2_P}
			\sigma^*(\theta_1+i\theta_{23})\big\vert_{y}
\]
which leads to the new contribution in
Proposition \ref{p:var-CS1},
\begin{equation}\label{e:new-contribution-l2}
\begin{split}
	&\bigg(-\frac{1}{2\pi} \sum_{y\in (\partial\L^2_P\cap D)}
		\sigma^*(\theta_1+i\theta_{23})\big\vert_{y}\bigg)(\frac{\partial}{\partial w}) \\
	=& -\frac{1}{2\pi}\sum_{z_j\in P}
		\big(\, (\dot{\phi}+2i\tilde{\theta}\,\dot{}\,)
			+(\phi +2 i\tilde{\theta})_z \dot{f}+(\phi+2i\tilde{\theta})_z f_z \dot{z}_j\, \big)(z_j)\\
	=& -\frac{1}{2\pi} \sum_{z_j\in P}
		\big(\, -\dot{f}_z+(\log\tilde{h}_j)\dot{}+\phi_z f_z\dot{z}_j\, \big)(z_j)
\end{split}
\end{equation}
where we used equation
\eqref{e:h-tilde-id}. We also have the following new contribution
from the set $P$ in Proposition \ref{p:var-B0},
\begin{equation}\label{e:new-der-b4}
\begin{split}
& \frac{i}{8\pi^2} \lim_{\delta\to 0}\sum_{z_j\in P}  \int_{|z-z_j|=\delta} (\phi-2i\theta)_z(\, ({\phi}-2i\theta)\circ f \,)\,\dot{}\, dz \\
=& \frac{i}{8\pi^2}\lim_{\delta\to 0}\sum_{z_j\in P} \Big(\, \int_{|z-z_j|=\delta} (\frac{2}{z-z_j}+(\phi-\log\tilde{h}_j)_z)\, \big(\,  2\,\frac{\dot{f}(z) -\dot{f}(z_j) -f_z(z_j)\dot{z}_j}{z-z_j} \, \big) \, dz\\
 &\qquad \qquad\qquad+ \int_{|z-z_j|=\delta} (\frac{2}{z-z_j}+(\phi-\log\tilde{h}_j)_z)\,  \big(\, (\phi -\log \tilde{h}_j)\circ f\,)\,\dot{}(z) \, \big) \, dz \, \Big)  \\
=&-\frac{1}{4\pi}   \sum_{z_j\in P}\big(\, 4\dot{f}_z (z_j) -2 (\phi-\log\tilde{h}_j)_z(z_j) f_z(z_j) \dot{z}_j +2 ((\phi -\log \tilde{h}_j)\circ f\,)\,\dot{}(z_j) \, \big) \\
=& -\frac{1}{2\pi} \sum_{z_j\in P}\big(\, \dot{f}_z -(\log\tilde{h}_j)\dot{}-\phi_z f_z \dot{z}_j\, \big)(z_j) .
\end{split}
\end{equation}
Since the right sides of \eqref{e:new-contribution-l2} and \eqref{e:new-der-b4} cancel each other,  Theorem \ref{t:var-conCS} holds without modification.

For $I(X,d\lambda)$, we have to modify its definition by adding terms from the set $P$:
\begin{equation}\label{e:new-def-I}
\begin{split}
I(X,d\lambda) =&\lim_{\delta\to 0}  \big(\ \int_{X_\delta} |\psi_z|^2\, d^2z \\ & +\frac{i}2\sum_{p_k\in Z}\int_{S_\delta(z_k)} \frac{(\phi-2\log |h|)(z)}{\bar{z}-\bar{z}_k}d\bar{z}
-i\sum_{q_j\in P}\int_{S_\delta(z_j)} \frac{(\phi-2\log |h|)(z)}{\bar{z}-\bar{z}_j}d\bar{z}\, \big)\\ &\hspace{3cm} -\pi \sum_{p_k\in Z} (\phi-\log |\tilde{h}_k|)(z_k) +2\pi \sum_{q_j\in P} (\phi+2\log |\tilde{h}_j|)(z_j)  .
\end{split}
\end{equation}
Here $z$, in the integral around $p_k$ and $q_j$, represents a local coordinate near $p_k$, $q_j$
with $z_k=z(p_k)$, $z_j=z(q_j)$.
The holomorphic variation formula of this $I$ can be computed as in
the proof of Theorem \ref{t:intrin} with some new contributions from the set $P$. First, we have the following new contribution in \eqref{e:bound-delta0},
\begin{align}\label{e:new-bound-delta0}
\begin{split}
& -\frac{i}{2}\sum_{q_j\in P}\Big(\, \int_{|z-z_j|=\delta_0}4\,\big(\,  \frac{\dot{f}(z) -\dot{f}(z_j) -f_z(z_j)\dot{z}_j}{z-z_j} \, \big) \big(\, \frac{1}{z-z_j} dz +\frac{1}{\bar{z}-\bar{z}_j} d\bar{z} \, \big)\\
 & \qquad\qquad\qquad +2\,\big(\,  \frac{\dot{f}(z) -\dot{f}(z_j) -f_z(z_j)\dot{z}_j}{z-z_j} \, \big) \big(\, (\phi-\log\tilde{h}_j)_z dz +(\phi-\log\bar{\tilde{h}}_j)_{\bar{z}} d\bar{z} \, \big)  \\
 &\qquad\qquad \qquad + 2\, ((\phi-\log \tilde{h}_j)\circ f)\dot{}(z)\, \big(\, \frac{1}{z-z_j} dz +\frac{1}{\bar{z}-\bar{z}_j} d\bar{z} \, \big) \, \Big) +O(\delta_0) \\
= &\ -2\pi \sum_{q_j\in P} (\phi_z f_z)(z_j) \dot{z}_j  +O(\delta_0).
\end{split}
\end{align}
Secondly, we have the following new contribution to \eqref{e:fourth term},
\begin{align}\label{e:new-fourth term}
\begin{split}
\partial \Big(\,  -{i}\sum_{q_j\in P}\int_{S_{\delta_0}(z_j)} \frac{(\phi-2\log| h| )}{\bar{z}-\bar{z}_{j,w}}\, d\bar{z} \, \Big)(\varpi)
= 2\pi \sum_{q_j\in P} (\, \dot{f}_z-(\log\tilde{h}_j)\dot{}\, )(z_j) +O(\delta_0).
\end{split}
\end{align}
For the holomorphic variation of the last term in $I(X,d\lambda)$, we have
\begin{equation}\label{e:new-I-last}
\partial \Big(\, 2\pi \sum_{q_j\in P} (\phi+2\log |\tilde{h}_j|)(z_j) \, \Big) (\varpi) = 2\pi \sum_{q_j\in P}(\, -\dot{f}_z + (\log\tilde{h}_j)\dot{}\,  + \phi_zf_z\dot{z}_j)(z_j).
\end{equation}
Hence, the new contributions from \eqref{e:new-bound-delta0}, \eqref{e:new-fourth term}, and \eqref{e:new-I-last} cancel each other, and Theorem \ref{t:intrin} holds without
modification.

The holomorphic variation of $\log \tau_B^{24}$ is given by the same formula as in Theorem \ref{t:var-log T}.
The same proof as in Theorem \ref{t:var-log T} also works for this case since the holomorphic variations of $\int_{a_i} d\lambda$ and $\int_{b_i} d\lambda$ vanish.
The only possible difference in the proof is the contributions of residues of $\frac{R_{d\lambda}}{\lambda_z}$ at $q_j\in P$
in \eqref{e:left-hand-side-Tau} and \eqref{e:right-hand-side-Tau}. But one can see that this is regular at $q_j\in P$ and there is no contribution from the set $P$.

Finally, in Section \ref{s:proof-main-theorem}, we showed that Theorem \ref{t:main theorem-intro} follows from Theorems
\ref{t:var-conCS}, \ref{t:intrin}, and \ref{t:var-log T}. Since these Theorems still hold in this case, Theorem \ref{t:main theorem-intro}
follows in this case as well.


\bibliographystyle{plain}

\end{document}